\documentclass[reqno, eucal]{amsart}

\usepackage[mathscr]{eucal} 
\usepackage[dvips]{graphicx}
\usepackage[dvips]{color}
\usepackage{amsmath,amsfonts,amssymb,amsthm,amscd}
\usepackage{epic}
\usepackage{eepic}
\usepackage{longtable}
\usepackage{array}

\numberwithin{equation}{section}

\newtheorem{theorem}{Theorem}[section]
\newtheorem{lemma}[theorem]{Lemma}

\newtheorem{proposition}[theorem]{Proposition}

\theoremstyle{definition}

\newtheorem{example}[theorem]{Example}
\newtheorem{remark}[theorem]{Remark}

\newcommand{\geh}{\mathfrak{g}}

\newcommand{\ot}{\otimes}

\newcommand{\Z}{{\mathbb Z}}

\newcommand{\C}{{\mathbb C}}

\newcommand{\Rm}{{\mathscr R}}

\newcommand{\bt}{\boxtimes}

\newcommand{\W}{\mathcal{W}}

\begin{document}

\title{Tetrahedron equation and generalized quantum groups}

\author{Atsuo Kuniba}
\email{atsuo@gokutan.c.u-tokyo.ac.jp}
\address{Institute of Physics, University of Tokyo, Komaba, Tokyo 153-8902, Japan}

\author{Masato Okado}
\email{okado@sci.osaka-cu.ac.jp}
\address{Department of Mathematics, Osaka City University, 
3-3-138, Sugimoto, Sumiyoshi-ku, Osaka, 558-8585, Japan}

\author{Sergey Sergeev}
\email{Sergey.Sergeev@canberra.edu.au}
\address{Faculty of Education, Science, Technology, 
Engineering and Mathematics, University of Canberra, ACT 2106 Australia}

\maketitle

\begin{center}
{\it Dedicated to Professor Rodney Baxter on the occasion of his seventy-fifth birthday}
\end{center}

\vspace{0.5cm}
\begin{center}{\bf Abstract}
\end{center}

We construct $2^n$-families of 
solutions of the Yang-Baxter equation from 
$n$-products of three-dimensional $R$ and $L$ operators
satisfying the tetrahedron equation. 
They are identified with the quantum $R$ matrices for the 
Hopf algebras known as generalized quantum groups.
Depending on the number of $R$'s  and $L$'s involved in the product, 
the trace construction interpolates
the symmetric tensor representations of $U_q(A^{(1)}_{n-1})$ and the
anti-symmetric tensor representations of $U_{-q^{-1}}(A^{(1)}_{n-1})$,
whereas a boundary vector construction interpolates 
the $q$-oscillator representation of $U_q(D^{(2)}_{n+1})$ and 
the spin representation of $U_{-q^{-1}}(D^{(2)}_{n+1})$.
The intermediate cases are associated with  
an affinization of quantum superalgebras.

\vspace{0.4cm}

\section{Introduction}\label{a:sec1}
Tetrahedron equation \cite{Zam80} is a generalization of the 
Yang-Baxter equation \cite{Bax} and serves as a key to the 
quantum integrability in three dimensions (3D).
It represents a factorization condition on the scattering of straight strings 
in $(2+1)$-dimension and also as a sufficient condition for the 
commutativity of layer-to-layer transfer matrices in 3D lattice models.
Among several versions of the tetrahedron equation 
we are concerned with the following two types in this paper:
\begin{align*}
&\Rm_{1,2,4}\Rm_{1,3,5}\Rm_{2,3,6}\Rm_{4,5,6}=
\Rm_{4,5,6}\Rm_{2,3,6}\Rm_{1,3,5}\Rm_{1,2,4},\\
&{\mathscr L}_{1,2,4}{\mathscr L}_{1,3,5}{\mathscr L}_{2,3,6}\Rm_{4,5,6}=
\Rm_{4,5,6}{\mathscr L}_{2,3,6}{\mathscr L}_{1,3,5}{\mathscr L}_{1,2,4}.
\end{align*}
Here $\Rm$ and ${\mathscr L} $ are linear operators 
on $F \otimes F \otimes F$ and $V \otimes V \otimes F$, respectively 
with some vector spaces $F$ and $V$.
The indices signify the components on which these operators act nontrivially.
The first equation is to hold in 
$\mathrm{End}(F^{\otimes 6})$ and the second one in  
$\mathrm{End}(V^{\otimes 3} \otimes F^{\otimes 3})$.
We refer to $\Rm$ and ${\mathscr L}$ as 3D $R$ and 3D $L$ for short.

The first remarkable example of 3D $R$ was proposed in \cite{Zam81}.
It was referred to as ``an extraordinary feat of intuition" by 
Baxter, who proved that it indeed satisfies the 
tetrahedron equation \cite{Bax83}.
It was actually another extraordinary feat which pioneered the subject and 
inspired subsequent developments that continue in earnest until today. 

The tetrahedron equation reduces to the Yang Baxter equation
\[
R_{1,2} R_{1,3} R_{2,3}  = R_{2,3}R_{1,3} R_{1,2}
\]
if the spaces $4,5,6$, which we call the {\em auxiliary spaces}, are 
evaluated away suitably.
It implies a certain connection between a class of solvable models 
in 2D and 3D by regarding the third direction in the latter as the 
internal degrees of freedom of local spins in the former.
Such a correspondence between 2D and 3D theories 
has been studied in a variety of contexts,  
e.g.~\cite{BSt, S97, KasV}, 
and highlighted by the celebrated interpretation/extension 
of the 2D chiral Potts model \cite{AMPTY, BP} 
and its generalizations \cite{BKMS,DJMM} in the 3D picture \cite{BB, SMS}. 

In this paper we study reductions of the 
tetrahedron equations to the Yang-Baxter equation 
for the distinguished example of the 3D $R$ and the 3D $L$
associated with the quantized algebra of functions on 
$\mathrm{SL}_3$ \cite{KV} and 
the $q$-oscillator algebra \cite{BS}.
The 3D $R$'s in these works are known to coincide \cite{KO1},
contain a parameter $q$ and correspond 
to choosing $F$ to be the 
$q$-bosonic Fock space 
$F = \bigoplus_{m \ge 0}\C |m\rangle$ and    
$V = \C^2$. 
See Section \ref{a:sec2} and the literatures cited therein for more description.

There are three kinds of freedom that one can introduce in 
performing the reduction.
First, the elimination of the auxiliary spaces can be done 
either by taking the {\em trace} \cite{BS} 
or matrix elements with respect to  
special {\em boundary vectors} \cite{KS}.
Curiously this freedom is known to reflect the boundary 
shape of the Dynkin diagram relevant 
to the final result as observed in \cite[Remark 7.2]{KS} and 
\cite[Remark 14]{KO3}.
Second, the reduction can be applied to the 
$n$-layer version of the tetrahedron equations for any $n\ge 1$.
Third, the resulting product of $n$ operators may consist of a 
{\em mixture} of $\Rm$'s and ${\mathscr L}$'s in any order.
This last freedom, which was pointed out in \cite{S09, KO2} but 
hitherto remained almost intact, 
is the theme of systematic investigation in this paper.
It leads to $2^n$-families of 
solutions to the Yang-Baxter equation 
corresponding to $(\Rm \; \text{or}\;  \mathscr{L})^n$.
They act on 
$\W \otimes \W$ where 
$\W = (F  \; \text{or}\; V)^{\otimes n}$ is an arbitrary
$n$-fold tensor product of $F$ and $V$ (Section \ref{ss:Rybe}). 
There is a similarity transformation 
exchanging $F\otimes V$ and $V \otimes F$ locally in $\W$,  
hence there are essentially $(n+1)$-tuple of solutions for each $n$   
(Section \ref{ss:er}).
Our principal result is Theorem \ref{th:main}, which clarifies their origin
as the quantum $R$ matrices for the Hopf algebras 
that we will also introduce in Section \ref{sec:gg}.
They include an affinization of quantum superalgebras \cite{FSV,Y91,Y94} as well as
a class of quantum affine algebras \cite{D86,Ji}.
In general, they offer
examples of {\em generalized quantum groups}.
This notion emerged through the classification of 
pointed Hopf algebras \cite{AS,Hec} 
and was first introduced in \cite{H}. 
For recent developments of generalized quantum groups, 
see for instance \cite{HY,AY,AYY,BY}.

By changing the portion of $\Rm$ and ${\mathscr L}$ or 
equivalently $V$ and $F$ in the $n$-product,  
the trace construction interpolates the quantum $R$ matrices for 
the symmetric tensor representations of the 
quantum affine algebra $U_q(A^{(1)}_{n-1})$ and the
anti-symmetric tensor representations of $U_{-q^{-1}}(A^{(1)}_{n-1})$.
Similarly the boundary vector construction interpolates 
the $q$-oscillator representation of $U_q(D^{(2)}_{n+1})$ and 
the spin representation of $U_{-q^{-1}}(D^{(2)}_{n+1})$.  
The intermediate cases are related to 
the quantum superalgebras (Section \ref{ss:qsa}).
These results generalize and synthesize the previous works
\cite{S09, BS, KS, KO2, KO3, KOproc, KOS}.
They indicate hidden quantum group structures in 3D integrable lattice models,
or put another way, hidden 3D structures in the quantum group theory.

The layout of the paper is as follows.
In Section \ref{a:sec2} we recall the 3D $R$, the 3D $L$ and the 
construction of the $2^n$-families of spectral parameter dependent solutions 
$S(\epsilon_1,\ldots, \epsilon_n)\, (\epsilon_i=0,1)$ 
of the Yang-Baxter equation
by various 2D reductions of their mixed $n$-products.
We explain the equivalence of $S(\epsilon_1,\ldots, \epsilon_n)$ 
under permutations of $\epsilon_1,\ldots, \epsilon_n$ 
and summarize the known results in Section \ref{ss:kr}.
This part serves as an extended version of the introduction.
In Section \ref{sec:gg} we introduce the generalized 
quantum groups 
${\mathcal U}_A={\mathcal U}_A(\epsilon_1,\ldots, \epsilon_n)$, 
${\mathcal U}_B = {\mathcal U}_B(\epsilon_1,\ldots, \epsilon_n)$ 
and their irreducible representations $\pi_x$.
They are relevant to the trace and a boundary vector construction,
respectively.
Precise relations to the quantum superalgebras 
$A_q(m,m')$ and $B_q(m,m')$ \cite{FSV} (see also \cite{Y91}) are explained in 
Section \ref{ss:qsa}.
The quantum $R$ matrices are defined via the 
commutativity with 
${\mathcal U}_A$ or ${\mathcal U}_B$ and a normalization condition.
In Section \ref{sec:qR} the main result of the paper,
Theorem \ref{th:main},  is presented which identifies the 
$S(\epsilon_1,\ldots, \epsilon_n)$ constructed in Section \ref{a:sec2}
with the quantum $R$ matrices introduced in Section \ref{ss:qr}.

The remaining Sections \ref{sec:proof}, \ref{sec:IrreA} and 
\ref{sec:IrreD} are devoted to a proof of Theorem \ref{th:main}.
Our strategy is to establish that $S(\epsilon_1,\ldots, \epsilon_n)$ satisfies the 
same characterization as the quantum $R$ matrices given in Section \ref{ss:qr}.
In Section \ref{sec:proof} we prove the commutativity of 
$S(\epsilon_1,\ldots, \epsilon_n)$ 
with ${\mathcal U}_A$ or ${\mathcal U}_B$.
It is vital to also ensure the irreducibility of the tensor product
representation $\pi_x \otimes \pi_y$ 
in order to characterize the 
$R$ matrices as their commutant.
Since no relevant result was found in the literature, 
we include a self-contained proof of the irreducibility for 
the ${\mathcal U}_A(\epsilon_1,\ldots, \epsilon_n)$-module
in Section \ref{sec:IrreA} and 
the ${\mathcal U}_B(\epsilon_1,\ldots, \epsilon_n)$-module 
in Section \ref{sec:IrreD} 
for $(\epsilon_1,\ldots, \epsilon_n)$ of the form 
$(1^\kappa,0^{n-\kappa})
=(\overbrace{1,\ldots,1}^\kappa,\overbrace{0,\ldots,0}^{n-\kappa})$.
In the course of the proof, we obtain the 
spectral decomposition of the quantum $R$ matrices 
for the ${\mathcal U}_A$ and ${\mathcal U}_B$ explicitly.
In particular (\ref{spd2}), (\ref{spd3}) and Proposition \ref{spd4}
are new results, which lead to the explicit formulas as in 
Examples \ref{ex:r10}--\ref{ex:rB}.

Throughout the paper we assume that $q$ is generic and 
use the following notations:
\begin{align*}
&(z;q)_m = \prod_{k=1}^m(1-z q^{k-1}),\;\;
(q)_m = (q; q)_m,\;\;
\binom{m}{k}_{\!\!q}= \frac{(q)_m}{(q)_k(q)_{m-k}},\\
&[m]=[m]_q = \frac{q^m-q^{-m}}{q-q^{-1}}, \;\;
[m]_q! = \prod_{k=1}^m [k]_q,\;\;
{m\brack k} = \frac{[m]!}{[k]![m-k]!},
\end{align*}
where the both $q$-binomials are to be understood as zero
unless $0 \le k \le m$.

\section{Families of solutions to Yang-Baxter equation}\label{a:sec2}
This section may still be regarded as a continuation of the introduction,
where we formulate our problem and list the preceding results 
precisely. 

\subsection{3D $R$}\label{subsec2.1}
Let $F$ and $F^\ast$ be a Fock space and its dual  
\begin{align*}
F = \bigoplus_{m\ge 0}\C |m \rangle,\quad 
F^\ast = \bigoplus_{m \ge 0} \C \langle m |
\end{align*}
whose pairing is given by $\langle m | m' \rangle = \delta_{m,m'}(q^2)_m$.
Define a linear operator 
$\Rm$ on $F\otimes F \otimes F$ by 
\begin{align}
&\Rm(|i\rangle \otimes |j\rangle \otimes |k\rangle) = 
\sum_{a,b,c\ge 0} \Rm^{a,b,c}_{i,j,k}
|a\rangle \otimes |b\rangle \otimes |c\rangle,\label{Rabc}\\
&\Rm^{a,b,c}_{i,j,k} =\delta^{a+b}_{i+j}\delta^{b+c}_{j+k}
\sum_{\lambda+\mu=b}(-1)^\lambda
q^{i(c-j)+(k+1)\lambda+\mu(\mu-k)}
\frac{(q^2)_{c+\mu}}{(q^2)_c}
\binom{i}{\mu}_{\!\!q^2}
\binom{j}{\lambda}_{\!\!q^2},\label{Rex}
\end{align}
where $\delta^j_{k}=\delta_{j,k}$ just to save the space.
The sum (\ref{Rex}) is over $\lambda, \mu \ge 0$ 
satisfying $\lambda+\mu=b$, which is also bounded by the 
condition $\mu\le i$ and $\lambda \le j$.
The $\Rm$ will simply be called 3D $R$ in this paper.
It satisfies the tetrahedron equation:
\begin{align}\label{TE}
\Rm_{1,2,4}\Rm_{1,3,5}\Rm_{2,3,6}\Rm_{4,5,6}=
\Rm_{4,5,6}\Rm_{2,3,6}\Rm_{1,3,5}\Rm_{1,2,4},
\end{align}
which is an equality in $\mathrm{End}(F^{\otimes 6})$.
Here $\Rm_{i,j,k}$ acts as $\Rm$ on the 
$i,j,k$ th components from the left in the 
tensor product $F^{\otimes 6}$.
With the space $F$ denoted by an arrow, 
the relation (\ref{TE}) is depicted as follows:

{\unitlength 0.1in
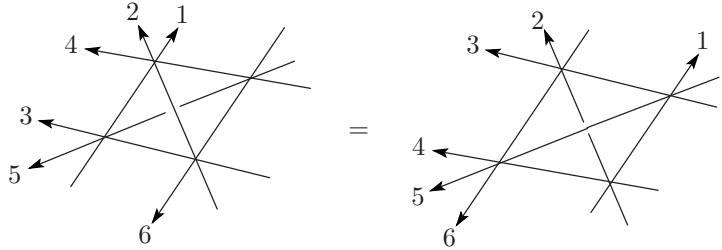
\begin{picture}(  6.8900, 15)( 0.4800,-27)
%
{\color[named]{Black}{%
\special{pn 8}%
\special{pa 3938 1476}%
\special{pa 3248 2496}%
\special{fp}%
\special{sh 1}%
\special{pa 3248 2496}%
\special{pa 3302 2452}%
\special{pa 3278 2452}%
\special{pa 3270 2430}%
\special{pa 3248 2496}%
\special{fp}%
}}%
\put(32.0900,-25.6000){\makebox(0,0){6}}%
%
{\color[named]{Black}{%
\special{pn 8}%
\special{pa 4620 1722}%
\special{pa 3930 1986}%
\special{fp}%
}}%
%
{\color[named]{Black}{%
\special{pn 8}%
\special{pa 3932 1990}%
\special{pa 3112 2316}%
\special{fp}%
\special{sh 1}%
\special{pa 3112 2316}%
\special{pa 3182 2310}%
\special{pa 3162 2296}%
\special{pa 3168 2272}%
\special{pa 3112 2316}%
\special{fp}%
}}%
\put(30.4700,-23.4600){\makebox(0,0){5}}%
%
{\color[named]{Black}{%
\special{pn 8}%
\special{pa 3950 2430}%
\special{pa 4510 1606}%
\special{fp}%
\special{sh 1}%
\special{pa 4510 1606}%
\special{pa 4456 1650}%
\special{pa 4480 1650}%
\special{pa 4488 1672}%
\special{pa 4510 1606}%
\special{fp}%
}}%
\put(45.3500,-15.2700){\makebox(0,0){1}}%
%
{\color[named]{Black}{%
\special{pn 8}%
\special{pa 3912 1956}%
\special{pa 3708 1478}%
\special{fp}%
\special{sh 1}%
\special{pa 3708 1478}%
\special{pa 3716 1548}%
\special{pa 3728 1528}%
\special{pa 3752 1532}%
\special{pa 3708 1478}%
\special{fp}%
}}%
%
{\color[named]{Black}{%
\special{pn 8}%
\special{pa 3938 2018}%
\special{pa 4142 2496}%
\special{fp}%
}}%
\put(36.7100,-14.2300){\makebox(0,0){2}}%
%
{\color[named]{Black}{%
\special{pn 8}%
\special{pa 4608 1878}%
\special{pa 3404 1580}%
\special{fp}%
\special{sh 1}%
\special{pa 3404 1580}%
\special{pa 3464 1614}%
\special{pa 3456 1592}%
\special{pa 3474 1576}%
\special{pa 3404 1580}%
\special{fp}%
}}%
\put(33.3300,-15.5300){\makebox(0,0){3}}%
%
{\color[named]{Black}{%
\special{pn 8}%
\special{pa 4302 2314}%
\special{pa 3128 2106}%
\special{fp}%
\special{sh 1}%
\special{pa 3128 2106}%
\special{pa 3190 2138}%
\special{pa 3182 2116}%
\special{pa 3198 2098}%
\special{pa 3128 2106}%
\special{fp}%
}}%
\put(30.5000,-20.8500){\makebox(0,0){4}}%
%
{\color[named]{Black}{%
\special{pn 8}%
\special{pa 2268 2248}%
\special{pa 1064 1950}%
\special{fp}%
\special{sh 1}%
\special{pa 1064 1950}%
\special{pa 1124 1984}%
\special{pa 1116 1962}%
\special{pa 1134 1946}%
\special{pa 1064 1950}%
\special{fp}%
}}%
\put(9.9300,-19.2300){\makebox(0,0){3}}%
%
{\color[named]{Black}{%
\special{pn 8}%
\special{pa 2482 1780}%
\special{pa 1308 1574}%
\special{fp}%
\special{sh 1}%
\special{pa 1308 1574}%
\special{pa 1370 1604}%
\special{pa 1362 1582}%
\special{pa 1378 1566}%
\special{pa 1308 1574}%
\special{fp}%
}}%
\put(12.3000,-15.5200){\makebox(0,0){4}}%
%
{\color[named]{Black}{%
\special{pn 8}%
\special{pa 1228 2294}%
\special{pa 1786 1468}%
\special{fp}%
\special{sh 1}%
\special{pa 1786 1468}%
\special{pa 1732 1512}%
\special{pa 1756 1512}%
\special{pa 1766 1534}%
\special{pa 1786 1468}%
\special{fp}%
}}%
\put(18.1200,-13.9000){\makebox(0,0){1}}%
%
{\color[named]{Black}{%
\special{pn 8}%
\special{pa 1994 2412}%
\special{pa 1584 1456}%
\special{fp}%
\special{sh 1}%
\special{pa 1584 1456}%
\special{pa 1592 1524}%
\special{pa 1606 1504}%
\special{pa 1630 1508}%
\special{pa 1584 1456}%
\special{fp}%
}}%
\put(15.5200,-13.7700){\makebox(0,0){2}}%
%
{\color[named]{Black}{%
\special{pn 8}%
\special{pa 2398 1638}%
\special{pa 1800 1874}%
\special{fp}%
}}%
%
{\color[named]{Black}{%
\special{pn 8}%
\special{pa 1722 1908}%
\special{pa 1012 2198}%
\special{fp}%
\special{sh 1}%
\special{pa 1012 2198}%
\special{pa 1082 2190}%
\special{pa 1062 2178}%
\special{pa 1066 2154}%
\special{pa 1012 2198}%
\special{fp}%
}}%
\put(9.3400,-22.2900){\makebox(0,0){5}}%
%
{\color[named]{Black}{%
\special{pn 8}%
\special{pa 2346 1462}%
\special{pa 1656 2482}%
\special{fp}%
\special{sh 1}%
\special{pa 1656 2482}%
\special{pa 1710 2438}%
\special{pa 1686 2438}%
\special{pa 1678 2416}%
\special{pa 1656 2482}%
\special{fp}%
}}%
\put(16.1700,-25.4700){\makebox(0,0){6}}%
\put(27.3500,-20.0800){\makebox(0,0){$=$}}%
\end{picture}}%

The 3D $R$ was obtained as the intertwiner of the quantized coordinate ring 
$A_q(sl_3)$ \cite{KV}\footnote{
The formula for it on p194 in \cite{KV} contains a misprint unfortunately.
Eq. (\ref{Rex}) here is a correction of it.}.  
It was also found 
from a quantum geometry consideration in a different gauge \cite{BS}. 
They were shown to be the same object in \cite[eq.(2.29)]{KO1}.
Appendix A in \cite{KO3} contains the recursion relations 
characterizing $\Rm$ and useful corollaries which will also be 
utilized in the present paper.
Here we note
\begin{align}\label{rtb}
\Rm = \Rm^{-1},\quad 
\Rm^{a,b,c}_{i,j,k}= \Rm^{c,b,a}_{k,j,i},\quad
\Rm^{a,b,c}_{i,j,k}=
\frac{(q^2)_i(q^2)_j(q^2)_k}{(q^2)_a(q^2)_b(q^2)_c}
\Rm_{a,b,c}^{i,j,k}.
\end{align}
The last property makes it consistent to define the
action of 3D $R$ on $F^*\otimes F^*\otimes F^*$ via
\begin{align*}
(\langle i | \otimes\langle j| \otimes\langle k| )\Rm
= \sum_{a,b,c\ge 0} \Rm^{a,b,c}_{i,j,k}
\langle a| \otimes\langle b| \otimes\langle c|.
\end{align*}
Let  ${\bf h}$ be the linear operator on $F$ and $F^\ast$ 
such that
${\bf h} |m\rangle = m |m\rangle$ and $\langle m | {\bf h} = \langle m | m$.
The factor $\delta^{a+b}_{i+j}\delta^{b+c}_{j+k}$ in 
(\ref{Rex}) implies 
\begin{align}\label{rz}
[\Rm, \,x^{{\bf h}_1}(xy)^{{\bf h}_2}y^{{\bf h}_3} ]=0,
\end{align}
where ${\bf h}_i$ denotes the one acting nontrivially on 
$\overset{i}{F}$ in 
$\overset{1}{F}\otimes \overset{2}{F}\otimes\overset{3}{F}$.

\subsection{Boundary vectors}\label{ss:bv}
Let us introduce the following vectors in $F$ and $F^\ast$:
\begin{align}
&|\chi_1\rangle = \sum_{m\ge 0}\frac{|m\rangle}{(q)_m},\quad
\langle \chi_1| = \sum_{m\ge 0}\frac{\langle m|}{(q)_m},\quad
|\chi_2\rangle = \sum_{m\ge 0}\frac{|2m\rangle}{(q^4)_m},
\quad
\langle \chi_2| = \sum_{m\ge 0}\frac{\langle 2m|}{(q^4)_m}.
\label{xk}
\end{align}
We further set 
$|\chi_s(z)\rangle  = z^{{\bf h}/s}|\chi_s\rangle$ and 
$\langle \chi_s(z) | = \langle \chi_s |z^{{\bf h}/s}$ for $s=1,2$,
where the factor $1/s$ is just a matter of normalization of 
the spectral parameter $z$. 
They are called {\em boundary vectors}.
The following property \cite{KS}, which actually reduces to $x=y=1$ case
by (\ref{rz}),  will play a key role:
\begin{equation}\label{cv}
\begin{split}
{\mathscr R}(|\chi_s(x)\rangle \otimes|\chi_s(xy)\rangle \otimes|\chi_s(y)\rangle)
&=|\chi_s(x)\rangle \otimes|\chi_s(xy)\rangle \otimes|\chi_s(y)\rangle
\in F\otimes F\otimes F,\\
(\langle \chi_s(x)| \otimes \langle \chi_s(xy)| \otimes \langle \chi_s(y)|){\mathscr R}
&= \langle \chi_s(x)| \otimes \langle \chi_s(xy)| \otimes \langle \chi_s(y)|
\in F^*\otimes F^*\otimes F^*.
\end{split}
\end{equation}

\subsection{3D $L$}
Now we proceed to the 3D $L$ \cite{BS}.
Let $V = \C v_0 \oplus \C v_1$ and define ${\mathscr L}$ by
\begin{align}\label{rino}
{\mathscr L} &=({\mathscr L}_{\alpha, \beta}^{\gamma,\delta})
 \in \mathrm{End}(V \otimes V \otimes F),
\quad
{\mathscr L}(v_\alpha \otimes v_\beta \otimes |m\rangle)
= \sum_{\gamma,\delta}v_\gamma \otimes v_\delta \otimes 
{\mathscr L}_{\alpha, \beta}^{\gamma,\delta}|m\rangle,
\end{align}
where ${\mathscr L}_{\alpha, \beta}^{\gamma,\delta}
\in \mathrm{End}(F)$ are zero except the following six cases:
\begin{align}\label{lak}
{\mathscr L}_{0, 0}^{0,0}&= {\mathscr L}_{1,1}^{1,1} = 1,\;\;
{\mathscr L}_{0, 1}^{0,1} = -q{\bf k},\;\;
{\mathscr L}_{1,0}^{1,0} = {\bf k},\;\;
{\mathscr L}_{1,0}^{0,1} = {\bf a}^-,\;\;
{\mathscr L}^{1,0}_{0,1} = {\bf a}^+.
\end{align}
The operators ${\bf a}^\pm, {\bf k} \in \mathrm{End}(F) $ 
are called $q$-oscillators and act on $F$ by
\begin{align}
{\bf a}^+|m\rangle = |m+1\rangle,\quad
{\bf a}^-|m\rangle = (1-q^{2m})|m-1\rangle,\quad
{\bf k}|m\rangle = q^m|m\rangle.
\label{ac1}
\end{align}
Thus ${\bf k}= q^{\bf h}$ in terms of ${\bf h}$ 
defined around (\ref{rz}).
They satisfy the relations
\begin{align}
{\bf k}\,{\bf a}^{\pm} = q^{\pm 1}{\bf a}^{\pm}\,{\bf k},\quad
{\bf a}^+{\bf a}^- = 1-{\bf k}^2,\quad
{\bf a}^-{\bf a}^+ = 1-q^2 {\bf k}^2.
\label{ac2}
\end{align}
The ${\mathscr L}$ will simply be called 3D $L$ in this paper.
It may be regarded as 
a six-vertex model having the $q$-oscillator valued Boltzmann weights.
From this viewpoint the last two relations are viewed as 
a quantization of the so-called free-fermion condition
\cite[eq.(10.16.4)]{Bax}$|_{\omega_7=\omega_8=0}$.
We will also use the notation similar to (\ref{Rabc}) to 
express (\ref{rino}) as
\begin{equation}\label{Lex}
\begin{split}
&{\mathscr L}(v_\alpha \otimes v_\beta \otimes |m\rangle) = 
\sum_{\gamma,\delta, j} {\mathscr L}^{\gamma,\delta, j}_{\alpha, \beta, m}
v_\gamma \otimes v_\delta \otimes |j\rangle,\\
&{\mathscr L}^{0, 0, j}_{0, 0, m}
={\mathscr L}^{1, 1, j}_{1, 1, m}=\delta^j_m,\quad
{\mathscr L}^{0, 1, j}_{0, 1, m}= -\delta^j_mq^{m+1},\quad
{\mathscr L}^{1, 0, j}_{1, 0, m}=\delta^j_mq^m,\\
&{\mathscr L}^{0, 1, j}_{1, 0, m}=\delta^j_{m-1}(1-q^{2m}),\quad
{\mathscr L}^{1, 0, j}_{0, 1, m}=\delta^j_{m+1}.
\end{split}
\end{equation}
The other ${\mathscr L}^{\gamma,\delta, j}_{\alpha, \beta, m}$ are zero.
The 3D $L$ satisfies the $RLLL$ type tetrahedron equation \cite{BS}:
\begin{align}\label{RLLL}
{\mathscr L}_{1,2,4}{\mathscr L}_{1,3,5}{\mathscr L}_{2,3,6}\Rm_{4,5,6}=
\Rm_{4,5,6}{\mathscr L}_{2,3,6}{\mathscr L}_{1,3,5}{\mathscr L}_{1,2,4}.
\end{align}
This is an equality in 
$\mathrm{End}(\overset{1}{V}\otimes \overset{2}{V} \otimes \overset{3}{V}
\otimes \overset{4}{F} \otimes \overset{5}{F} \otimes \overset{6}{F})$,
where 
$\overset{1}{V},\overset{2}{V},\overset{3}{V}$ are copies of $V$ and 
$\overset{4}{F},\overset{5}{F},\overset{6}{F}$ are the ones for $F$.
The indices of ${\mathscr R}$ and ${\mathscr L}$ signify the 
components of the tensor product on which these operators act nontrivially.
With the space $V$ denoted by an dotted arrow, 
the relation (\ref{RLLL}) is depicted as follows:

{\unitlength 0.1in
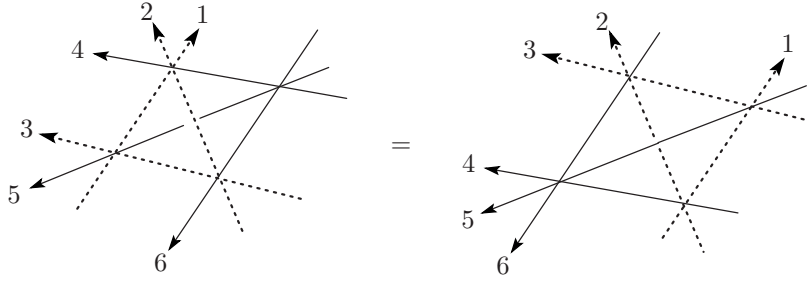
\begin{picture}(  6.8900, 16.5)( 0.4800,-28.5)
%
{\color[named]{Black}{%
\special{pn 8}%
\special{pa 4320 1504}%
\special{pa 3546 2652}%
\special{fp}%
\special{sh 1}%
\special{pa 3546 2652}%
\special{pa 3600 2608}%
\special{pa 3576 2608}%
\special{pa 3566 2586}%
\special{pa 3546 2652}%
\special{fp}%
}}%
\put(29.6700,-21.0300){\makebox(0,0){$=$}}%
\put(17.1000,-27.0900){\makebox(0,0){6}}%
%
{\color[named]{Black}{%
\special{pn 8}%
\special{pa 2530 1490}%
\special{pa 1754 2638}%
\special{fp}%
\special{sh 1}%
\special{pa 1754 2638}%
\special{pa 1808 2594}%
\special{pa 1784 2594}%
\special{pa 1776 2572}%
\special{pa 1754 2638}%
\special{fp}%
}}%
\put(9.4100,-23.5200){\makebox(0,0){5}}%
%
{\color[named]{Black}{%
\special{pn 8}%
\special{pa 1828 1990}%
\special{pa 1030 2316}%
\special{fp}%
\special{sh 1}%
\special{pa 1030 2316}%
\special{pa 1098 2310}%
\special{pa 1078 2296}%
\special{pa 1084 2272}%
\special{pa 1030 2316}%
\special{fp}%
}}%
%
{\color[named]{Black}{%
\special{pn 8}%
\special{pa 2588 1686}%
\special{pa 1914 1952}%
\special{fp}%
}}%
\put(16.3700,-13.9400){\makebox(0,0){2}}%
%
{\color[named]{Black}{%
\special{pn 13}%
\special{pa 2136 2542}%
\special{pa 1674 1466}%
\special{dt 0.045}%
\special{sh 1}%
\special{pa 1674 1466}%
\special{pa 1682 1536}%
\special{pa 1694 1516}%
\special{pa 1718 1520}%
\special{pa 1674 1466}%
\special{fp}%
}}%
\put(19.2900,-14.0800){\makebox(0,0){1}}%
%
{\color[named]{Black}{%
\special{pn 13}%
\special{pa 1272 2426}%
\special{pa 1900 1496}%
\special{dt 0.045}%
\special{sh 1}%
\special{pa 1900 1496}%
\special{pa 1846 1540}%
\special{pa 1870 1540}%
\special{pa 1880 1562}%
\special{pa 1900 1496}%
\special{fp}%
}}%
\put(12.7500,-15.9000){\makebox(0,0){4}}%
%
{\color[named]{Black}{%
\special{pn 8}%
\special{pa 2682 1848}%
\special{pa 1362 1614}%
\special{fp}%
\special{sh 1}%
\special{pa 1362 1614}%
\special{pa 1424 1646}%
\special{pa 1416 1624}%
\special{pa 1432 1606}%
\special{pa 1362 1614}%
\special{fp}%
}}%
\put(10.0800,-20.0700){\makebox(0,0){3}}%
%
{\color[named]{Black}{%
\special{pn 13}%
\special{pa 2442 2374}%
\special{pa 1088 2038}%
\special{dt 0.045}%
\special{sh 1}%
\special{pa 1088 2038}%
\special{pa 1148 2072}%
\special{pa 1140 2050}%
\special{pa 1158 2034}%
\special{pa 1088 2038}%
\special{fp}%
}}%
\put(33.2200,-21.9000){\makebox(0,0){4}}%
%
{\color[named]{Black}{%
\special{pn 8}%
\special{pa 4730 2448}%
\special{pa 3410 2214}%
\special{fp}%
\special{sh 1}%
\special{pa 3410 2214}%
\special{pa 3472 2244}%
\special{pa 3464 2222}%
\special{pa 3480 2206}%
\special{pa 3410 2214}%
\special{fp}%
}}%
\put(36.4100,-15.9200){\makebox(0,0){3}}%
%
{\color[named]{Black}{%
\special{pn 13}%
\special{pa 5074 1958}%
\special{pa 3720 1620}%
\special{dt 0.045}%
\special{sh 1}%
\special{pa 3720 1620}%
\special{pa 3780 1656}%
\special{pa 3772 1634}%
\special{pa 3790 1618}%
\special{pa 3720 1620}%
\special{fp}%
}}%
\put(40.2000,-14.4500){\makebox(0,0){2}}%
%
{\color[named]{Black}{%
\special{pn 13}%
\special{pa 4320 2114}%
\special{pa 4550 2652}%
\special{dt 0.045}%
}}%
%
{\color[named]{Black}{%
\special{pn 13}%
\special{pa 4290 2046}%
\special{pa 4062 1508}%
\special{dt 0.045}%
\special{sh 1}%
\special{pa 4062 1508}%
\special{pa 4070 1576}%
\special{pa 4082 1556}%
\special{pa 4106 1562}%
\special{pa 4062 1508}%
\special{fp}%
}}%
\put(49.9200,-15.6200){\makebox(0,0){1}}%
%
{\color[named]{Black}{%
\special{pn 13}%
\special{pa 4336 2578}%
\special{pa 4964 1650}%
\special{dt 0.045}%
\special{sh 1}%
\special{pa 4964 1650}%
\special{pa 4910 1694}%
\special{pa 4934 1694}%
\special{pa 4944 1716}%
\special{pa 4964 1650}%
\special{fp}%
}}%
\put(33.1800,-24.8300){\makebox(0,0){5}}%
%
{\color[named]{Black}{%
\special{pn 8}%
\special{pa 4314 2082}%
\special{pa 3392 2448}%
\special{fp}%
\special{sh 1}%
\special{pa 3392 2448}%
\special{pa 3462 2442}%
\special{pa 3442 2428}%
\special{pa 3448 2406}%
\special{pa 3392 2448}%
\special{fp}%
}}%
%
{\color[named]{Black}{%
\special{pn 8}%
\special{pa 5088 1782}%
\special{pa 4312 2080}%
\special{fp}%
}}%
\put(35.0100,-27.2400){\makebox(0,0){6}}%
\end{picture}}%

Viewed as an equation on ${\mathscr R}$, 
(\ref{RLLL}) is equivalent to the 
intertwining relation of the irreducible representations of the 
quantized coordinate ring $A_q(sl_3)$ \cite[eq.(2.15)]{KO1} 
in the sense that
the both lead to the same solution given in (\ref{Rex}) 
up to an overall  normalization.

\subsection{$n$-layer version of tetrahedron equation}

In order to treat $\Rm$ and ${\mathscr L}$ on an equal footing
we introduce the notation
\begin{align*}
&W^{(0)}= F,\quad W^{(1)}= V,\\
&{\mathscr S}^{(0)}= \Rm,\quad {\mathscr S}^{(1)} = {\mathscr L},
\quad
{\mathscr S}^{(0)\,a,b,c}_{\phantom{(0)}\, i,j,k} 
= \Rm^{a,b,c}_{i,j,k},\quad
{\mathscr S}^{(1)\,a,b,c}_{\phantom{(1)}\, i,j,k} 
= {\mathscr L}^{a,b,c}_{i,j,k}.
\end{align*}
Then ${\mathscr S}^{(\epsilon)} \in \mathrm{End}(
W^{(\epsilon)} \otimes W^{(\epsilon)} \otimes F)$ for $\epsilon=0,1$.
From (\ref{Rex}) and (\ref{Lex}) it obeys the conservation law:
\begin{align}\label{ice0} 
{\mathscr S}^{(\epsilon)\, a,b,c}_{\phantom{(\epsilon)}\, i,j,k} = 0 
\;\; \text{unless}\;\; (a+b,b+c)=(i+j,j+k).
\end{align} 
The tetrahedron equations of $RRRR$ type (\ref{TE})
and $RLLL$ type (\ref{RLLL}) are summarized as 
\begin{align}\label{TE2}
{\mathscr S}^{(\epsilon)}_{1,2,4}{\mathscr S}^{(\epsilon)}_{1,3,5}{\mathscr S}^{(\epsilon)}_{2,3,6}\Rm_{4,5,6}=
\Rm_{4,5,6}{\mathscr S}^{(\epsilon)}_{2,3,6}{\mathscr S}^{(\epsilon)}_{1,3,5}{\mathscr S}^{(\epsilon)}_{1,2,4}
\quad (\epsilon=0,1),
\end{align}
which is an equality in 
$\mathrm{End}(
W^{(\epsilon)} \otimes W^{(\epsilon)}\otimes W^{(\epsilon)}
\otimes F\otimes F \otimes F)$.

Let $n$ be a positive integer.
Given an arbitrary sequence 
$(\epsilon_1,\ldots, \epsilon_n) \in \{0,1\}^n$, we set 
\begin{align}\label{Wdef}
&\W= W^{(\epsilon_1)} \otimes \cdots \otimes W^{(\epsilon_n)}.
\end{align}
Regarding (\ref{TE2}) as a one-layer relation,
we extend it to the $n$-layer version.
Let $\overset{\alpha_i}{W^{(\epsilon_i)}},
\overset{\beta_i}{W^{(\epsilon_i)}},
\overset{\gamma_i}{W^{(\epsilon_i)}}$ be copies of $W^{(\epsilon_i)}$,
where $\alpha_i, \beta_i$ and $\gamma_i\,(i=1,\ldots, n)$
are just distinct labels.
Replacing the spaces $1,2,3$ by them in (\ref{TE2})
we have 
\begin{align*}
{\mathscr S}^{(\epsilon_i)}_{\alpha_i, \beta_i, 4}
{\mathscr S}^{(\epsilon_i)}_{\alpha_i, \gamma_i, 5}
{\mathscr S}^{(\epsilon_i)}_{\beta_i, \gamma_i, 6}\Rm_{4,5,6}=
\Rm_{4,5,6}
{\mathscr S}^{(\epsilon_i)}_{\beta_i, \gamma_i, 6}
{\mathscr S}^{(\epsilon_i)}_{\alpha_i, \gamma_i, 5}
{\mathscr S}^{(\epsilon_i)}_{\alpha_i, \beta_i, 4}
\end{align*}
for each $i$.
Thus for any $i$ one can carry $\Rm_{4,5,6}$ through  
${\mathscr S}^{(\epsilon_i)}_{\alpha_i, \beta_i, 4}
{\mathscr S}^{(\epsilon_i)}_{\alpha_i, \gamma_i, 5}
{\mathscr S}^{(\epsilon_i)}_{\beta_i, \gamma_i, 6}$ to the left 
converting it into the reverse order product
${\mathscr S}^{(\epsilon_i)}_{\beta_i, \gamma_i, 6}
{\mathscr S}^{(\epsilon_i)}_{\alpha_i, \gamma_i, 5}
{\mathscr S}^{(\epsilon_i)}_{\alpha_i, \beta_i, 4}$.
Repeating this $n$ times leads to
\begin{equation}\label{TEn}
\begin{split}
&\bigl({\mathscr S}^{(\epsilon_1)}_{\alpha_1, \beta_1, 4}
{\mathscr S}^{(\epsilon_1)}_{\alpha_1, \gamma_1, 5}
{\mathscr S}^{(\epsilon_1)}_{\beta_1, \gamma_1, 6}\bigr)
\cdots 
\bigl({\mathscr S}^{(\epsilon_n)}_{\alpha_n, \beta_n, 4}
{\mathscr S}^{(\epsilon_n)}_{\alpha_n, \gamma_n, 5}
{\mathscr S}^{(\epsilon_n)}_{\beta_n, \gamma_n, 6}\bigr)\Rm_{4,5,6}\\
&= 
\Rm_{4,5,6}
\bigl({\mathscr S}^{(\epsilon_1)}_{\beta_1, \gamma_1, 6}
{\mathscr S}^{(\epsilon_1)}_{\alpha_1, \gamma_1, 5}
{\mathscr S}^{(\epsilon_1)}_{\alpha_1, \beta_1, 4}
\bigr)
\cdots 
\bigl({\mathscr S}^{(\epsilon_n)}_{\beta_n, \gamma_n, 6}
{\mathscr S}^{(\epsilon_n)}_{\alpha_n, \gamma_n, 5}
{\mathscr S}^{(\epsilon_n)}_{\alpha_n, \beta_n, 4}
\bigr).
\end{split}
\end{equation}
This is an equality in 
$\mathrm{End}(\overset{\boldsymbol \alpha}{\W}\otimes 
\overset{\boldsymbol\beta}{\W}\otimes 
\overset{\boldsymbol\gamma}{\W}\otimes 
\overset{4}{F}\otimes 
\overset{5}{F}\otimes 
\overset{6}{F})$,
where 
${\boldsymbol\alpha}=(\alpha_1,\ldots, \alpha_n)$
is the array of labels and 
$\overset{\boldsymbol \alpha}{\W}= 
\overset{\alpha_1}{W^{(\epsilon_1)}}\otimes \cdots \otimes 
\overset{\alpha_n}{W^{(\epsilon_n)}}$.
The notations  
$\overset{\boldsymbol\beta}{\W}$ and $\overset{\boldsymbol\gamma}{\W}$
should be understood similarly.
They are just copies of $\W$ defined in (\ref{Wdef}).

{\unitlength 0.1in
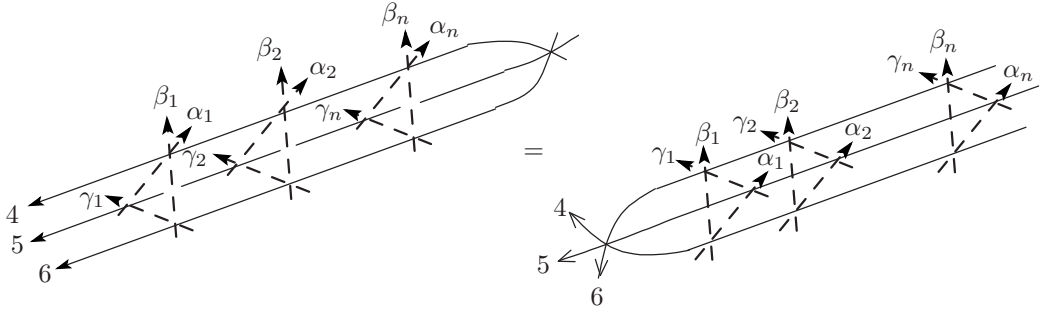
\begin{picture}( 17.7100,  18)(3,-40)
%
{\color[named]{Black}{%
\special{pn 8}%
\special{pa 3734 3274}%
\special{pa 5506 2630}%
\special{fp}%
}}%
%
{\color[named]{Black}{%
\special{pn 8}%
\special{pa 3896 3596}%
\special{pa 5666 2952}%
\special{fp}%
}}%
%
{\color[named]{Black}{%
\special{pn 8}%
\special{pa 3806 3436}%
\special{pa 5732 2736}%
\special{fp}%
}}%
%
{\color[named]{Black}{%
\special{pn 13}%
\special{pa 4314 3314}%
\special{pa 3848 3130}%
\special{da 0.070}%
\special{sh 1}%
\special{pa 3848 3130}%
\special{pa 3902 3172}%
\special{pa 3898 3150}%
\special{pa 3916 3136}%
\special{pa 3848 3130}%
\special{fp}%
}}%
%
{\color[named]{Black}{%
\special{pn 13}%
\special{pa 4024 3646}%
\special{pa 3978 3064}%
\special{da 0.070}%
\special{sh 1}%
\special{pa 3978 3064}%
\special{pa 3962 3132}%
\special{pa 3982 3118}%
\special{pa 4002 3130}%
\special{pa 3978 3064}%
\special{fp}%
}}%
%
{\color[named]{Black}{%
\special{pn 13}%
\special{pa 3936 3644}%
\special{pa 4306 3194}%
\special{da 0.070}%
\special{sh 1}%
\special{pa 4306 3194}%
\special{pa 4248 3232}%
\special{pa 4272 3234}%
\special{pa 4278 3258}%
\special{pa 4306 3194}%
\special{fp}%
}}%
%
{\color[named]{Black}{%
\special{pn 13}%
\special{pa 4756 3162}%
\special{pa 4290 2976}%
\special{da 0.070}%
\special{sh 1}%
\special{pa 4290 2976}%
\special{pa 4344 3020}%
\special{pa 4340 2996}%
\special{pa 4358 2982}%
\special{pa 4290 2976}%
\special{fp}%
}}%
%
{\color[named]{Black}{%
\special{pn 13}%
\special{pa 4466 3494}%
\special{pa 4420 2912}%
\special{da 0.070}%
\special{sh 1}%
\special{pa 4420 2912}%
\special{pa 4404 2980}%
\special{pa 4424 2966}%
\special{pa 4444 2978}%
\special{pa 4420 2912}%
\special{fp}%
}}%
%
{\color[named]{Black}{%
\special{pn 13}%
\special{pa 4378 3492}%
\special{pa 4748 3040}%
\special{da 0.070}%
\special{sh 1}%
\special{pa 4748 3040}%
\special{pa 4690 3080}%
\special{pa 4714 3082}%
\special{pa 4722 3104}%
\special{pa 4748 3040}%
\special{fp}%
}}%
%
{\color[named]{Black}{%
\special{pn 13}%
\special{pa 5586 2856}%
\special{pa 5120 2670}%
\special{da 0.070}%
\special{sh 1}%
\special{pa 5120 2670}%
\special{pa 5174 2714}%
\special{pa 5170 2690}%
\special{pa 5188 2676}%
\special{pa 5120 2670}%
\special{fp}%
}}%
%
{\color[named]{Black}{%
\special{pn 13}%
\special{pa 5294 3188}%
\special{pa 5248 2606}%
\special{da 0.070}%
\special{sh 1}%
\special{pa 5248 2606}%
\special{pa 5234 2674}%
\special{pa 5252 2660}%
\special{pa 5274 2672}%
\special{pa 5248 2606}%
\special{fp}%
}}%
%
{\color[named]{Black}{%
\special{pn 13}%
\special{pa 5208 3186}%
\special{pa 5578 2734}%
\special{da 0.070}%
\special{sh 1}%
\special{pa 5578 2734}%
\special{pa 5520 2774}%
\special{pa 5544 2776}%
\special{pa 5550 2798}%
\special{pa 5578 2734}%
\special{fp}%
}}%
%
{\color[named]{Black}{%
\special{pn 8}%
\special{pa 3726 3274}%
\special{pa 3698 3290}%
\special{pa 3642 3324}%
\special{pa 3616 3344}%
\special{pa 3590 3362}%
\special{pa 3568 3384}%
\special{pa 3548 3406}%
\special{pa 3530 3432}%
\special{pa 3514 3458}%
\special{pa 3500 3486}%
\special{pa 3488 3514}%
\special{pa 3476 3544}%
\special{pa 3466 3576}%
\special{pa 3450 3640}%
\special{pa 3442 3674}%
\special{pa 3436 3708}%
\special{pa 3428 3740}%
\special{pa 3428 3742}%
\special{fp}%
}}%
%
{\color[named]{Black}{%
\special{pn 8}%
\special{pa 3904 3596}%
\special{pa 3840 3608}%
\special{pa 3808 3612}%
\special{pa 3778 3616}%
\special{pa 3746 3618}%
\special{pa 3712 3620}%
\special{pa 3648 3620}%
\special{pa 3616 3616}%
\special{pa 3584 3610}%
\special{pa 3554 3604}%
\special{pa 3494 3580}%
\special{pa 3442 3548}%
\special{pa 3390 3508}%
\special{pa 3342 3464}%
\special{pa 3320 3440}%
\special{pa 3284 3404}%
\special{fp}%
}}%
%
{\color[named]{Black}{%
\special{pn 8}%
\special{pa 3806 3436}%
\special{pa 3776 3448}%
\special{pa 3748 3460}%
\special{pa 3688 3484}%
\special{pa 3658 3494}%
\special{pa 3568 3530}%
\special{pa 3418 3584}%
\special{pa 3388 3594}%
\special{pa 3358 3606}%
\special{pa 3298 3626}%
\special{pa 3266 3636}%
\special{pa 3236 3646}%
\special{pa 3220 3652}%
\special{fp}%
}}%
%
{\color[named]{Black}{%
\special{pn 8}%
\special{pa 3356 3436}%
\special{pa 3276 3398}%
\special{fp}%
\special{pa 3276 3398}%
\special{pa 3308 3478}%
\special{fp}%
}}%
%
{\color[named]{Black}{%
\special{pn 8}%
\special{pa 3412 3644}%
\special{pa 3438 3728}%
\special{fp}%
\special{pa 3438 3728}%
\special{pa 3476 3650}%
\special{fp}%
}}%
%
{\color[named]{Black}{%
\special{pn 8}%
\special{pa 3284 3596}%
\special{pa 3218 3654}%
\special{fp}%
\special{pa 3218 3654}%
\special{pa 3306 3656}%
\special{fp}%
}}%
\put(42.5700,-31.6900){\makebox(0,0)[lb]{$\alpha_1$}}%
\put(39.3500,-30.4800){\makebox(0,0)[lb]{$\beta_1$}}%
\put(38.3900,-31.4500){\makebox(0,0)[rb]{$\gamma_1$}}%
\put(43.4600,-28.7900){\makebox(0,0)[lb]{$\beta_2$}}%
\put(47.1600,-30.1600){\makebox(0,0)[lb]{$\alpha_2$}}%
\put(42.6500,-29.8400){\makebox(0,0)[rb]{$\gamma_2$}}%
\put(55.4500,-27.1000){\makebox(0,0)[lb]{$\alpha_n$}}%
\put(51.5900,-25.6500){\makebox(0,0)[lb]{$\beta_n$}}%
\put(50.8600,-26.6200){\makebox(0,0)[rb]{$\gamma_n$}}%
\put(31.3800,-36.6800){\makebox(0,0){5}}%
\put(32.2700,-33.7000){\makebox(0,0){4}}%
\put(34.2000,-38.3700){\makebox(0,0){6}}%
\put(8.4400,-33.7800){\makebox(0,0)[rb]{$\gamma_1$}}%
\put(19.1500,-26.9400){\makebox(0,0)[lb]{$\alpha_2$}}%
\put(13.8300,-31.6100){\makebox(0,0)[rb]{$\gamma_2$}}%
\put(25.4300,-24.7700){\makebox(0,0)[lb]{$\alpha_n$}}%
\put(20.9200,-29.2000){\makebox(0,0)[rb]{$\gamma_n$}}%
\put(3.6100,-34.1100){\makebox(0,0){4}}%
\put(5.3000,-37.1700){\makebox(0,0){6}}%
%
{\color[named]{Black}{%
\special{pn 13}%
\special{pa 916 3412}%
\special{pa 1288 2960}%
\special{da 0.070}%
\special{sh 1}%
\special{pa 1288 2960}%
\special{pa 1230 3000}%
\special{pa 1254 3002}%
\special{pa 1260 3024}%
\special{pa 1288 2960}%
\special{fp}%
}}%
%
{\color[named]{Black}{%
\special{pn 13}%
\special{pa 1218 3526}%
\special{pa 1170 2910}%
\special{da 0.070}%
\special{sh 1}%
\special{pa 1170 2910}%
\special{pa 1154 2978}%
\special{pa 1174 2962}%
\special{pa 1194 2974}%
\special{pa 1170 2910}%
\special{fp}%
}}%
%
{\color[named]{Black}{%
\special{pn 13}%
\special{pa 1304 3492}%
\special{pa 860 3316}%
\special{da 0.070}%
\special{sh 1}%
\special{pa 860 3316}%
\special{pa 914 3358}%
\special{pa 910 3336}%
\special{pa 928 3322}%
\special{pa 860 3316}%
\special{fp}%
}}%
%
{\color[named]{Black}{%
\special{pn 13}%
\special{pa 1488 3202}%
\special{pa 1896 2706}%
\special{da 0.070}%
\special{sh 1}%
\special{pa 1896 2706}%
\special{pa 1838 2746}%
\special{pa 1862 2748}%
\special{pa 1868 2770}%
\special{pa 1896 2706}%
\special{fp}%
}}%
%
{\color[named]{Black}{%
\special{pn 13}%
\special{pa 1820 3328}%
\special{pa 1768 2650}%
\special{da 0.070}%
\special{sh 1}%
\special{pa 1768 2650}%
\special{pa 1752 2718}%
\special{pa 1772 2702}%
\special{pa 1792 2714}%
\special{pa 1768 2650}%
\special{fp}%
}}%
%
{\color[named]{Black}{%
\special{pn 13}%
\special{pa 1914 3290}%
\special{pa 1426 3098}%
\special{da 0.070}%
\special{sh 1}%
\special{pa 1426 3098}%
\special{pa 1480 3140}%
\special{pa 1476 3118}%
\special{pa 1494 3104}%
\special{pa 1426 3098}%
\special{fp}%
}}%
%
{\color[named]{Black}{%
\special{pn 13}%
\special{pa 2164 2960}%
\special{pa 2534 2510}%
\special{da 0.070}%
\special{sh 1}%
\special{pa 2534 2510}%
\special{pa 2476 2548}%
\special{pa 2500 2550}%
\special{pa 2508 2574}%
\special{pa 2534 2510}%
\special{fp}%
}}%
%
{\color[named]{Black}{%
\special{pn 13}%
\special{pa 2466 3076}%
\special{pa 2418 2458}%
\special{da 0.070}%
\special{sh 1}%
\special{pa 2418 2458}%
\special{pa 2402 2526}%
\special{pa 2422 2512}%
\special{pa 2442 2524}%
\special{pa 2418 2458}%
\special{fp}%
}}%
%
{\color[named]{Black}{%
\special{pn 13}%
\special{pa 2552 3040}%
\special{pa 2108 2864}%
\special{da 0.070}%
\special{sh 1}%
\special{pa 2108 2864}%
\special{pa 2162 2908}%
\special{pa 2158 2884}%
\special{pa 2176 2870}%
\special{pa 2108 2864}%
\special{fp}%
}}%
%
{\color[named]{Black}{%
\special{pn 8}%
\special{pa 2744 2526}%
\special{pa 450 3362}%
\special{fp}%
\special{sh 1}%
\special{pa 450 3362}%
\special{pa 520 3358}%
\special{pa 500 3344}%
\special{pa 506 3320}%
\special{pa 450 3362}%
\special{fp}%
}}%
%
{\color[named]{Black}{%
\special{pn 8}%
\special{pa 2890 2856}%
\special{pa 594 3692}%
\special{fp}%
\special{sh 1}%
\special{pa 594 3692}%
\special{pa 664 3688}%
\special{pa 644 3674}%
\special{pa 650 3650}%
\special{pa 594 3692}%
\special{fp}%
}}%
%
{\color[named]{Black}{%
\special{pn 8}%
\special{pa 1150 3298}%
\special{pa 462 3550}%
\special{fp}%
\special{sh 1}%
\special{pa 462 3550}%
\special{pa 530 3546}%
\special{pa 512 3532}%
\special{pa 518 3508}%
\special{pa 462 3550}%
\special{fp}%
}}%
%
{\color[named]{Black}{%
\special{pn 8}%
\special{pa 2934 2662}%
\special{pa 2496 2822}%
\special{fp}%
}}%
%
{\color[named]{Black}{%
\special{pn 8}%
\special{pa 1762 3074}%
\special{pa 1246 3262}%
\special{fp}%
}}%
%
{\color[named]{Black}{%
\special{pn 8}%
\special{pa 2390 2848}%
\special{pa 1846 3046}%
\special{fp}%
}}%
\put(12.7100,-29.2800){\makebox(0,0)[lb]{$\alpha_1$}}%
\put(17.7800,-26.0600){\makebox(0,0)[rb]{$\beta_2$}}%
\put(12.3800,-28.5500){\makebox(0,0)[rb]{$\beta_1$}}%
\put(24.4600,-24.2000){\makebox(0,0)[rb]{$\beta_n$}}%
%
{\color[named]{Black}{%
\special{pn 8}%
\special{pa 2744 2518}%
\special{pa 2808 2510}%
\special{pa 2872 2504}%
\special{pa 2904 2502}%
\special{pa 2968 2502}%
\special{pa 3000 2504}%
\special{pa 3032 2508}%
\special{pa 3064 2514}%
\special{pa 3124 2534}%
\special{pa 3154 2546}%
\special{pa 3182 2558}%
\special{pa 3240 2588}%
\special{pa 3260 2598}%
\special{fp}%
}}%
%
{\color[named]{Black}{%
\special{pn 8}%
\special{pa 2890 2848}%
\special{pa 2954 2824}%
\special{pa 2984 2812}%
\special{pa 3014 2798}%
\special{pa 3042 2782}%
\special{pa 3066 2764}%
\special{pa 3088 2744}%
\special{pa 3108 2720}%
\special{pa 3124 2692}%
\special{pa 3138 2664}%
\special{pa 3162 2602}%
\special{pa 3174 2572}%
\special{pa 3194 2512}%
\special{pa 3206 2482}%
\special{pa 3214 2458}%
\special{fp}%
}}%
%
{\color[named]{Black}{%
\special{pn 8}%
\special{pa 2938 2654}%
\special{pa 3000 2636}%
\special{pa 3060 2616}%
\special{pa 3090 2604}%
\special{pa 3174 2562}%
\special{pa 3230 2530}%
\special{pa 3258 2512}%
\special{pa 3284 2494}%
\special{pa 3312 2478}%
\special{pa 3324 2470}%
\special{fp}%
}}%
\put(3.8500,-35.6400){\makebox(0,0){5}}%
\put(30.9000,-30.8900){\makebox(0,0){$=$}}%
\end{picture}}%

Here the broken arrows represent either 
usual arrow or dotted arrow depending on whether 
the corresponding $\epsilon_i$ is 0 or 1 and accordingly 
whether 
${\mathscr S}^{(\epsilon_i)}$ is ${\mathscr R}$ or ${\mathscr L}$.

The argument so far is just a 3D analogue of the well known 
fact in 2D that a single site relation $RLL = LLR$ for a local $L$ operator 
implies a similar relation for the $n$-site monodromy matrix
in the quantum inverse scattering method.

\subsection{Reduction to Yang-Baxter equation}\label{ss:Rybe}

In the $n$-layer tetrahedron equation (\ref{TEn}), 
the space $\overset{4}{F}\otimes 
\overset{5}{F}\otimes 
\overset{6}{F}$
will be referred to as {\em auxiliary space}.
One can reduce (\ref{TEn})  to
the Yang-Baxter equation by evaluating 
the auxiliary space away appropriately.
One natural way is to take the {\em trace} of (\ref{TEn})  
over the auxiliary space
after left multiplication of 
$x^{{\bf h}_4}(xy)^{{\bf h}_5}y^{{\bf h}_6}$ 
and right multiplication of $\Rm^{-1}_{4,5,6}$ \cite{BS}.
Another way is to evaluate (\ref{TEn}) 
between the boundary vectors 
$\langle \chi_s(x)| \otimes \langle \chi_s(xy)| \otimes \langle \chi_s(y)|$ 
and $|\chi_t\rangle \otimes|\chi_t\rangle \otimes|\chi_t\rangle \; 
(s,t=1,2)$ in (\ref{cv}) 
by regarding them as belonging to the 
auxiliary space and its dual\footnote{
Using $|\chi_t(x')\rangle \otimes
|\chi_t(x'y')\rangle \otimes
|\chi_t(y')\rangle $ 
just leads to a redefinition of $x,y$ due to 
(\ref{rz}) and (\ref{ice0}).} \cite{KS}.
By using  (\ref{rz}) and (\ref{cv})  it is easy to see that 
the result reduces to the Yang-Baxter equation
\begin{align}\label{sybe}
S_{\boldsymbol{\alpha, \beta}}(x)
S_{\boldsymbol{\alpha, \gamma}}(xy)
S_{\boldsymbol{\beta,\gamma}}(y)
=
S_{\boldsymbol{\beta,\gamma}}(y)
S_{\boldsymbol{\alpha, \gamma}}(xy)
S_{\boldsymbol{\alpha, \beta}}(x)
\in \mathrm{End}(
\overset{\boldsymbol\alpha}{\W}\otimes
\overset{\boldsymbol\beta}{\W}\otimes
\overset{\boldsymbol\gamma}{\W})
\end{align}
for the matrix 
$S_{\boldsymbol{\alpha, \beta}}(z)
\in \mathrm{End}
(\overset{\boldsymbol\alpha}{\W}\otimes \overset{\boldsymbol\beta}{\W})$
constructed as (cf. \cite[sec.~VIII]{S09} and \cite[sec.~5]{KO2})
\begin{align}
S_{\boldsymbol{\alpha, \beta}}(z)
&=\varrho(z)\mathrm{Tr}_3\left(z^{{\bf h}_3}
{\mathscr S}^{(\epsilon_1)}_{\alpha_1, \beta_1, 3}
\cdots
{\mathscr S}^{(\epsilon_n)}_{\alpha_n, \beta_n, 3}\right)
\quad \,\;\;\;(\text{trace}),
\label{sdef0}\\
&=\varrho(z)\langle \chi_s|z^{{\bf h}_3/s}
{\mathscr S}^{(\epsilon_1)}_{\alpha_1, \beta_1, 3}
\cdots
{\mathscr S}^{(\epsilon_n)}_{\alpha_n, \beta_n, 3}
|\chi_t\rangle
\quad (\text{evaluation by boundary vectors}), \label{sdef}
\end{align}
where the scalar $\varrho(z)$ is inserted to control the 
normalization.
The trace or evaluation by boundary vectors are done 
with respect to the auxiliary Fock space $F = \overset{3}{F}$ 
signified by 3. 
 
To express the matrix elements 
of $S_{\boldsymbol{\alpha, \beta}}(z)$ uniformly,
we introduce the following notation for the basis of $\W$ (\ref{Wdef}): 
\begin{align}
&\W = \bigoplus_{m_1,\ldots, m_n} \!\!\C
|m_1,\ldots, m_n\rangle, \qquad
|m_1,\ldots, m_n\rangle = |m_1\rangle^{(\epsilon_1)}
\otimes \cdots \otimes 
|m_n\rangle^{(\epsilon_n)},\label{syk}\\
&|m\rangle^{(0)} = |m\rangle \in F\;\;(m \in \Z_{\ge 0}),\qquad
|m\rangle^{(1)} = v_m \in V\; \; (m\in \{0,1\}).\label{kby}
\end{align}
The range of the indices $m_i$ are to be understood as 
$\Z_{\ge 0}$ or $\{0,1\}$ according to $\epsilon_i=0$ or $1$ as in 
(\ref{kby}). 
It will crudely be denoted by $0 \le m_i \le 1/\epsilon_i$.
We use the shorthand 
$|{\bf m}\rangle = |m_1,\ldots, m_n\rangle$
for ${\bf m}=(m_1,\ldots, m_n)$ and 
write (\ref{syk}) as $\W= \bigoplus_{\bf m} \C |{\bf m}\rangle$.
In particular $|{\bf 0}\rangle$ with 
${\bf 0}:=(0,\ldots, 0)$ denotes the vacuum vector.
We set $|{\bf m}| = m_1+\cdots + m_n$.
In the later sections (e.g. Sections \ref{sec:IrreA} and \ref{sec:IrreD}) where 
the distinction between $F$ and $V$ is clear from the context,
we will denote $v_0$ and $v_1$ also by $|0\rangle$
and $|1\rangle$.

Let $S^{\mathrm{tr}}(z)$ and 
$S^{s,t}(z) \in \mathrm{End}(\W \otimes \W)$
denote the solutions (\ref{sdef0}) and (\ref{sdef}) 
of the Yang-Baxter equation, 
where the inessential labels 
$\boldsymbol{\alpha}, \boldsymbol{\beta}$
are now suppressed.
Their actions are described as 
\begin{align}
&S^{\mathrm{tr}}(z)\bigl(|{\bf i}\rangle \otimes |{\bf j}\rangle\bigr)
= \sum_{{\bf a},{\bf b}}
S^{\mathrm{tr}}(z)^{{\bf a},{\bf b}}_{{\bf i},{\bf j}}
|{\bf a}\rangle \otimes |{\bf b}\rangle,
\quad
S^{s,t}(z)\bigl(|{\bf i}\rangle \otimes |{\bf j}\rangle\bigr)
= \sum_{{\bf a},{\bf b}}
S^{s,t}(z)^{{\bf a},{\bf b}}_{{\bf i},{\bf j}}
|{\bf a}\rangle \otimes |{\bf b}\rangle,
\label{sact}
\end{align}
with the matrix elements constructed as
\begin{align}
&S^{\mathrm{tr}}(z)^{{\bf a},{\bf b}}_{{\bf i},{\bf j}}
=\varrho(z)\!\!\sum_{c_0, \ldots, c_{n-1}}\!\!
z^{c_0}
{\mathscr S}^{(\epsilon_1)\, a_1, b_1, c_0}_{
\phantom{(\epsilon_1)}\,i_1, j_1, c_1}
{\mathscr S}^{(\epsilon_2)\, a_2, b_2, c_1}_{
\phantom{(\epsilon_2)}\,i_2, j_2, c_2}
\cdots
{\mathscr S}^{(\epsilon_{n-1})\,a_{n\!-\!1}, b_{n\!-\!1}, c_{n\!-\!2}}_{
\phantom{(\epsilon_{n-1})}\,i_{n\!-\!1}, j_{n\!-\!1}, c_{n\!-\!1}}
{\mathscr S}^{(\epsilon_n)\, a_n, b_n, c_{n\!-\!1}}_{
\phantom{(\epsilon_n)}i_n, j_n, c_0},\label{sabij0}\\
&S^{s,t}(z)^{{\bf a},{\bf b}}_{{\bf i},{\bf j}}
=\varrho(z)\!\!\sum_{c_0, \ldots, c_n}\!\!
\frac{z^{c_0}(q^2)_{sc_0}}{(q^{s^2})_{c_0}(q^{t^2})_{c_n}}
{\mathscr S}^{(\epsilon_1)\,a_1, b_1, sc_0}_{
\phantom{(\epsilon_1)}\,i_1, j_1, c_1}
{\mathscr S}^{(\epsilon_2)\,a_2, b_2, c_1}_{
\phantom{(\epsilon_2)}\,i_2, j_2, c_2}\cdots
{\mathscr S}^{(\epsilon_{n-1})\,a_{n\!-\!1}, b_{n\!-\!1}, c_{n\!-\!2}}_{
\phantom{(\epsilon_{n-1})}\,i_{n\!-\!1}, j_{n\!-\!1}, c_{n\!-\!1}}
{\mathscr S}^{(\epsilon_n)\,a_n, b_n, c_{n\!-\!1}}_{
\phantom{(\epsilon_n)}\,i_n, j_n, tc_n}.\label{sabij}
\end{align} 
The factor $(q^2)_{sc_0}$ in (\ref{sabij}) originates in 
$\langle m | m' \rangle = \delta_{m,m'}(q^2)_m$.
See Section \ref{subsec2.1}.
From (\ref{ice0}) it follows that
\begin{align}
&S^{\mathrm{tr}}(z)^{{\bf a},{\bf b}}_{{\bf i},{\bf j}}= 0 \;\;
\text{unless}\;\; \;{\bf a} + {\bf b} = {\bf i} + {\bf j} \;\;
\text{and}\;\;  |{\bf a}| = |{\bf i}|,\;
|{\bf b}| = |{\bf j}|,\label{iceA}\\
&S^{s,t}(z)^{{\bf a},{\bf b}}_{{\bf i},{\bf j}} = 0\;\;
\text{unless}\;\; \;{\bf a} + {\bf b} = {\bf i} + {\bf j}\quad (1\le s,t\le 2),
\label{iceD}\\
&S^{2,2}(z)^{{\bf a},{\bf b}}_{{\bf i},{\bf j}}= 0\;\;
\text{unless}\;\;  |{\bf a}| \equiv|{\bf i}|,\;
|{\bf b}| \equiv |{\bf j}|\mod 2.\label{ice22}
\end{align}
Given such 
${\bf a},{\bf b},{\bf i}$ and ${\bf j}$, (\ref{ice0}) effectively reduces 
the sums over $c_i \in \Z_{\ge 0}$ in both (\ref{sabij0}) and 
(\ref{sabij}) into a {\em single} sum.
The latter property in (\ref{iceA}) implies the 
direct sum decomposition:
\begin{align}
S^{\mathrm{tr}}(z) = \bigoplus_{l,m \ge 0}  S^{\mathrm{tr}}_{l,m}(z),\quad
S^{\mathrm{tr}}_{l,m}(z) \in \mathrm{End}(\W_l \otimes \W_m),\quad
\W_l = \bigoplus_{{\bf m}, |{\bf m}|=l}\C|{\bf m}\rangle \subset \W,
\label{wl}
\end{align}
where the former sum ranges over $0 \le l, m\le n$ if 
$\epsilon_1\cdots \epsilon_n = 1$ and 
$l,m \in \Z_{\ge 0}$ otherwise. 
Similarly $S^{2,2}(z)$ decomposes into four components due to (\ref{ice22}).
The normalization factor $\varrho(z)$ can be taken depending on the components 
and will be specified in Section \ref{ss:ex}.

Assign a solid arrow to $F$ and 
a dotted arrow to $V$, and 
depict the matrix elements of 3D $R$ and 3D $L$ as 
\[
\begin{picture}(200,49)(-90,-35)
\put(-122,-20){$\Rm^{a,b,c}_{i,j,k}=$}
\put(-30,-10){\vector(-3,-1){40}}
\put(-78,-27){$\scriptstyle{c}$}
\put(-48,-16){\vector(0,1){20}}\put(-48,-16){\line(0,-1){16}}
\put(-48,-16){\vector(3,-1){18}} \put(-48,-16){\line(-3,1){18}}
\put(-51,7){$\scriptstyle{b}$}
\put(-72,-10){$\scriptstyle{i}$}
\put(-27,-26){$\scriptstyle{a}$}
\put(-50,-39){$\scriptstyle{j}$}
\put(-27,-12){$\scriptstyle{k}$}
\put(130,-12){
\put(-88,-7){${\mathscr L}^{a,b,c}_{i,j,k}=$}
\put(3,1){\vector(-3,-1){40}}
\multiput(-15,-18)(0,3){10}{.}\put(-13.5,12.6){\vector(0,1){1}}
\multiput(-30,0)(3,-1){11}{.}\put(5,-10.6){\vector(3,-1){1}}
\put(-44,-14){$\scriptstyle{c}$}
\put(-37,0){$\scriptstyle{i}$}
\put(-16,17){$\scriptstyle{b}$}
\put(-16,-26){$\scriptstyle{j}$}
\put(9,-14){$\scriptstyle{a}$}
\put(6,0){$\scriptstyle{k}$}
}
\end{picture}
\]
Then the construction (\ref{sabij0}) and 
(\ref{sabij}) are depicted as
\[
\begin{picture}(200,100)(-210,-55)


\put(-220,0){
\put(3,1){\vector(-3,-1){73}}
\put(-20,-53){$\scriptstyle{S^{\mathrm{tr}}(z)^{{\bf a},{\bf b}}_{{\bf i},{\bf j}}}$}

\qbezier(-70,-23)(-120,-30)(-65,6)
\qbezier(-65,6)(-20,29)(2,33)
\qbezier(2,33)(27,40)(50,40.3)
\qbezier(50,40.3)(103,41)(56,18.7)

\put(-70,-30){$\scriptstyle{c_0}$}
 
\multiput(-48,-32)(0,6){6}{\put(0,0){\line(0,1){4}}}
\multiput(-31,-22)(-6,2){6}{\put(0,0){\line(-3,1){3}}}

\put(-51,6){$\scriptstyle{b_1}$}\put(-48,3){\vector(0,1){1}}
\put(-74,-10){$\scriptstyle{i_1}$}
\put(-30,-29){$\scriptstyle{a_1}$}\put(-29,-23){\vector(3,-1){1}}
\put(-50,-39){$\scriptstyle{j_1}$}

\put(-30,-15){$\scriptstyle{c_1}$}

\multiput(-15,-18)(0,6){5}{\put(0,0){\line(0,1){4}}}
\multiput(1,-10)(-6,2){5}{\put(0,0){\line(-3,1){3}}}

\put(-36,0){$\scriptstyle{i_2}$}
\put(-18,17){$\scriptstyle{b_2}$}\put(-15,13){\vector(0,1){1}}
\put(-17,-25){$\scriptstyle{j_2}$}
\put(4,-16){$\scriptstyle{a_2}$}\put(4,-11){\vector(3,-1){1}}

\put(2,-4){$\scriptstyle{c_2}$}

\multiput(5.1,1.7)(3,1){7}{.} 
\put(6,2){
\put(21,7){\line(3,1){30}}

\multiput(35,1)(0,6){4}{\put(0,0){\line(0,1){4}}}
\multiput(46,8)(-6,2){4}{\put(0,0){\line(-3,1){3}}}

\put(15,16){$\scriptstyle{i_n}$}
\put(51,3){$\scriptstyle{a_n}$}\put(49,7){\vector(3,-1){1}}
\put(33,29){$\scriptstyle{b_n}$}\put(35,26){\vector(0,1){1}}
\put(34,-7){$\scriptstyle{j_n}$}

\put(16,-1){$\scriptstyle{c_{n-1}}$}

}

}


\put(3,1){\vector(-3,-1){73}}
\put(-20,-53){$\scriptstyle{S^{s,t}(z)^{{\bf a},{\bf b}}_{{\bf i},{\bf j}}}$}
\put(-115,-27){$\scriptstyle{\langle \chi_s(z) |}$}

\put(-83,-27){$\scriptstyle{sc_0}$}
 
\multiput(-48,-32)(0,6){6}{\put(0,0){\line(0,1){4}}}
\multiput(-31,-22)(-6,2){6}{\put(0,0){\line(-3,1){3}}}

\put(-51,8){$\scriptstyle{b_1}$}\put(-48,4){\vector(0,1){1}}
\put(-74,-10){$\scriptstyle{i_1}$}
\put(-30,-29){$\scriptstyle{a_1}$}\put(-29,-23){\vector(3,-1){1}}
\put(-50,-39){$\scriptstyle{j_1}$}

\put(-30,-15){$\scriptstyle{c_1}$}

\multiput(-15,-18)(0,6){5}{\put(0,0){\line(0,1){4}}}
\multiput(1,-10)(-6,2){5}{\put(0,0){\line(-3,1){3}}}

\put(-36,0){$\scriptstyle{i_2}$}
\put(-18,17){$\scriptstyle{b_2}$}\put(-15,13){\vector(0,1){1}}
\put(-17,-25){$\scriptstyle{j_2}$}
\put(4,-16){$\scriptstyle{a_2}$}\put(4,-11){\vector(3,-1){1}}

\put(2,-4){$\scriptstyle{c_2}$}

\multiput(5.1,1.7)(3,1){7}{.} 
\put(6,2){
\put(21,7){\line(3,1){30}}

\multiput(35,1)(0,6){4}{\put(0,0){\line(0,1){4}}}
\multiput(46,8)(-6,2){4}{\put(0,0){\line(-3,1){3}}}

\put(15,16){$\scriptstyle{i_n}$}
\put(51,3){$\scriptstyle{a_n}$}\put(49,7){\vector(3,-1){1}}
\put(33,29){$\scriptstyle{b_n}$}\put(35,26){\vector(0,1){1}}
\put(34,-7){$\scriptstyle{j_n}$}

\put(16,-1){$\scriptstyle{c_{n-1}}$}

\put(53,17){$\scriptstyle{tc_n}$}
}
 
\put(75,19){$\scriptstyle{|\chi_t(1) \rangle}$}
 
\end{picture}
\]
Here the broken arrows designate either 
solid or dotted arrows according to 
$\epsilon_i = 0$ or $1$ at the corresponding site.
Thus (\ref{sabij0}) and (\ref{sabij}) may be regarded as 
a ``matrix product construction"  of 
$S^{\mathrm{tr}}(z)$ and $S^{s,t}(z)$ in terms of 
3D $R$ and 3D $L$ with the auxiliary space $F$.

\subsection{Examples: Normalization of
$S^{\mathrm{tr}}(z)$ and $S^{1,1}(z)$ }\label{ss:ex}
Set 
\begin{align*}
{\bf e}_i = (0,\ldots, 0,\overset{i}{1},0,\ldots, 0)
\in \Z^n,\qquad
{\bf e}_{> m} = {\bf e}_{m+1} + \cdots + {\bf e}_n\;\; \;(1 \le m \le n).
\end{align*}
We calculate some typical matrix elements from 
(\ref{sabij0}) and (\ref{sabij}).
As the first example we consider 
$S^{\mathrm{tr}}_{l,m}(z)\, (0 \le l,m \le n)$ with 
$\epsilon_1\cdots \epsilon_n = 1$.
We normalize it as
\begin{align}\label{norm1}
S^{\mathrm{tr}}_{l,m}(z)\bigl(|{\bf e}_{> n-l}\rangle \otimes | {\bf e}_{> n-m}\rangle\bigr) = 
|{\bf e}_{>n-l}\rangle \otimes | {\bf e}_{>n-m}\rangle.
\end{align}
By the definition (\ref{sabij0})  the relevant matrix element is calculated as 
\begin{align*}
S^{\mathrm{tr}}_{l,m}(z)^{{\bf e}_{> n-l}, {\bf e}_{>n-m}}_{{\bf e}_{>n-l}, {\bf e}_{>n-m}}
= \varrho(z)\sum_{c\ge 0}z^c
({\mathscr L}^{0,0,c}_{0,0,c})^{n-\max(l,m)}
({\mathscr L}^{1,0,c}_{1,0,c})^{(l-m)_+}
({\mathscr L}^{0,1,c}_{0,1,c})^{(m-l)_+}
({\mathscr L}^{1,1,c}_{1,1,c})^{\min(l,m)},
\end{align*}
where $(x)_+ = \max(x,0)$.
Thus from (\ref{Lex}) we find that the condition (\ref{norm1}) leads to 
the choice $\varrho(z) = (-q)^{-(m-l)_+}(1-q^{|l-m|}z)$.

As the second example we consider 
$S^{\mathrm{tr}}_{l,m}(z)\, (l,m \in \Z_{\ge 0})$ with 
$\epsilon_1\cdots \epsilon_n = 0$.
We pick {\em any} $i$ such that $\epsilon_i=0$ 
and normalize it as
\begin{align}\label{norm2}
S^{\mathrm{tr}}_{l,m}(z) \bigl(|l{\bf e}_i \rangle \otimes |m{\bf e}_i\rangle \bigr) = 
|l{\bf e}_i \rangle \otimes |m{\bf e}_i\rangle.
\end{align}
The relevant matrix element reads
$1=S^{\mathrm{tr}}_{l,m}(z)^{l{\bf e}_i, m{\bf e}_i}_{l{\bf e}_i, m{\bf e}_i}=
\varrho(z)\sum_{c\ge 0}z^c {\mathscr R}^{l,m,c}_{l,m,c}$.
It is an easy exercise using (\ref{Rex}) to show that
this gives
$\varrho(z) = 
\frac{z^{-m}(q^{l-m}z;q^2)_{m+1}}{(q^{l-m+2}z^{-1};q^2)_m}$.
Moreover this result is independent of the choice of such $i$. 

As the last example we consider $S^{1,1}(z)$ and normalize it as
\begin{align}\label{norm3}
S^{1,1}(z)\bigl( |{\bf 0}\rangle \otimes |{\bf 0}\rangle \bigr) = |{\bf 0}\rangle \otimes |{\bf 0}\rangle.
\end{align}
The relevant matrix element reads
$1=S^{1,1}(z)^{{\bf 0}, {\bf 0}}_{{\bf 0},{\bf 0}}
= \varrho(z) \sum_{c\ge 0}\frac{z^c(q^2)_c}{(q)^2_c}$.
Thus we take
$\varrho(z) = \frac{(z;q)_\infty}{(-qz;q)_\infty}$.
Under the normalization specified by (\ref{norm1})--(\ref{norm3}),
all the matrix elements are {\em rational} in $z$ and $q$.

\subsection{Equivalence relations}\label{ss:er}

Let us write $S^{\mathrm{tr}}(z)$ and $S^{s,t}(z)$ in 
(\ref{sact}) as 
$S^{\mathrm{tr}}(z|\epsilon_1,\ldots, \epsilon_n)$
and 
$S^{s,t}(z|\epsilon_1,\ldots, \epsilon_n)$
when their dependence on 
$(\epsilon_1,\ldots, \epsilon_n) \in \{0,1\}^n$ in (\ref{Wdef})
is to be emphasized.
The following fact was briefly mentioned in \cite[sec.~VIII]{S09}.

\begin{proposition}\label{pr:equiv}
Suppose $(\epsilon_1,\ldots, \epsilon_n) \in \{0,1\}^n$ is a permutation of  
$(\epsilon'_1,\ldots, \epsilon'_n)$.
Let ${\mathcal  W}$ and ${\mathcal W}'$ be the spaces 
(\ref{Wdef}) associated to them.
Then there is an invertible linear map 
$\Phi: {\mathcal W}\otimes {\mathcal W} \rightarrow 
{\mathcal W}'\otimes  {\mathcal W}'$
such that 
\begin{align*}
&\Phi \,S^{\mathrm{tr}}(z|\epsilon_1,\ldots, \epsilon_n) = 
S^{\mathrm{tr}}(z|\epsilon'_1,\ldots, \epsilon'_n) \,\Phi,
\qquad
\Phi \, S^{s,t}(z|\epsilon_1,\ldots, \epsilon_n) = 
S^{s,t}(z|\epsilon'_1,\ldots, \epsilon'_n) \,\Phi.
\end{align*}
\end{proposition}

\begin{proof}
In the $RLLL=LLLR$ relation (\ref{RLLL}), 
take the trace over the space $\overset{1}{V}$.
The result reads
\begin{align*}
\phi_{2,3,4,5} \, {\mathscr L}_{2,3,6}{\mathscr R}_{4,5,6} = 
{\mathscr R}_{4,5,6}{\mathscr L}_{2,3,6}\, \phi_{2,3,4,5},
\end{align*}
where $\phi_{2,3,4,5}= 
\mathrm{Tr}_1\bigl({\mathscr L}_{1,2,4}{\mathscr L}_{1,3,5}\bigr)$.
It may be regarded as a linear map 
$\phi: (\overset{2}{V}\otimes \overset{4}{F}) \otimes 
(\overset{3}{V}\otimes \overset{5}{F})
\rightarrow
(\overset{4}{F}\otimes \overset{2}{V})\otimes 
(\overset{5}{F}\otimes \overset{3}{V})$.
When represented as 
$\phi (\overset{2}{v}_\alpha \otimes \overset{3}{v}_\beta)
= \sum_{\alpha',\beta'} \phi^{\alpha',\beta'}_{\alpha,\beta} 
\overset{2}{v}_{\alpha'} \otimes \overset{3}{v}_{\beta'}$
in terms of the four by four matrix 
$(\phi^{\alpha',\beta'}_{\alpha,\beta})$
with elements from $\mathrm{End}(\overset{4}{F}\otimes \overset{5}{F})$,
it looks as 
\begin{align*}
(\phi^{\alpha',\beta'}_{\alpha,\beta})=\begin{pmatrix}
1+{\bf k}_4{\bf k}_5 & 0 & 0 & 0\\
0 & {\bf k}_4 - q{\bf k}_5 & {\bf a}^+_4{\bf a}^-_5 & 0\\
0 & {\bf a}^+_5{\bf a}^-_4 & {\bf k}_5-q{\bf k}_4 & 0 \\
0 & 0 & 0 & 1+q^2{\bf k}_4{\bf k}_5
\end{pmatrix},
\end{align*}
where $(\alpha,\beta)=(0,0),(0,1),(1,0),(1,1)$ from the left to the right and 
similarly for $(\alpha',\beta')$ from the top to the bottom.
It is easy to check that the square of this matrix equals
$\mathrm{diag}((1+{\bf k}_4{\bf k}_5)^2,(1-q{\bf k}_4{\bf k}_5)^2, 
(1-q{\bf k}_4{\bf k}_5)^2, (1+q^2{\bf k}_4{\bf k}_5)^2)$.
Since the spectrum of ${\bf k}$ is $q^{\Z_{\ge 0}}$, $\phi$
is invertible.
It follows that reversing the product of ${\mathscr L}$ and ${\mathscr R}$ 
in $\overset{6}{F}$ is equivalent to a similarity transformation 
in the other spaces by $\phi$.
Applying this observation to (\ref{sdef0}) and (\ref{sdef}), we have
\begin{align*}
(\cdots \otimes 1\otimes \phi
\otimes 1 \otimes \cdots) S^{\mathrm{tr}}(z|\ldots,1,0,\ldots)
&= S^{\mathrm{tr}}(z|\ldots,0,1,\ldots)
(\cdots \otimes 1\otimes \phi
\otimes 1 \otimes \cdots),\\
(\cdots \otimes 1\otimes \phi
\otimes 1 \otimes \cdots) S^{s,t}(z|\ldots,1,0,\ldots)
&= S^{s,t}(z|\ldots,0,1,\ldots)
(\cdots \otimes 1\otimes \phi
\otimes 1 \otimes \cdots) 
\end{align*}
whenever there are consecutive $1,0$ or $0,1$ in $(\epsilon_1,\ldots, \epsilon_n)$.
Repeating this transposition one can convert  
$(\epsilon_1,\ldots, \epsilon_n)$ into $(\epsilon'_1,\ldots, \epsilon'_n)$.
By denoting the composition of the corresponding $\phi$'s by $\Phi$,
the assertion follows.
\end{proof}

Proposition \ref{pr:equiv} reduces the 
study of the $2^n$-families 
$S^{\mathrm{tr}}(z|\epsilon_1,\ldots, \epsilon_n)$
and 
$S^{s,t}(z|\epsilon_1,\ldots, \epsilon_n)$
to the $(n\!+\!1)$-families
\begin{align*}
S^{\mathrm{tr}}(z|\overbrace{1,\ldots,1}^\kappa, \overbrace{0,\ldots,0}^{n-\kappa}),
\quad
S^{s,t}(z|\overbrace{1,\ldots,1}^\kappa, \overbrace{0,\ldots,0}^{n-\kappa})
\qquad (0 \le \kappa \le n).
\end{align*}

\subsection{Results on homogeneous cases and present work}\label{ss:kr}

Let us temporarily suppress $z$ in $S^{\mathrm{tr}}(z)$ and $S^{s,t}(z)$ in 
(\ref{sact}) and write them as 
$S^{\mathrm{tr}}(\epsilon_1,\ldots, \epsilon_n)$
and 
$S^{s,t}(\epsilon_1,\ldots, \epsilon_n)$.
Let $U_{q}({\mathfrak g})$ be a quantum affine algebra and let 
$R_{U_q({\mathfrak g})}(M\otimes M')$ denote the 
quantum $R$ matrix acting on the 
$U_{q}({\mathfrak g})$-module $M\otimes M'$.
Known results concern the homogeneous cases 
$\epsilon_1=\cdots = \epsilon_n=0, 1$.
Leaving minor technical remarks aside\footnote{
For example, a slight gauge adjustment is necessary as in (\ref{ksk}).}, 
they are stated in the present convention as follows 
(`rep'  means representation).
\begin{align}
S^{\mathrm{tr}}(0,\ldots, 0)
&= \bigoplus_{l,m\ge 0} R_{U_q(A^{(1)}_{n-1})}\!(V_l\otimes V_m),\quad
V_l = \text{$l$-symmetric tensor rep},
\label{tr0}\\
S^{\mathrm{tr}}(1,\ldots, 1)
&= \bigoplus_{0 \le l,m\le n} \!\!
R_{U_{-q^{-1}}(A^{(1)}_{n-1})}\!(V^l\otimes V^m),\quad
V^l = \text{$l$-anti-symmetric tensor rep} ,
\label{tr1}\\
S^{1,1}(1,\ldots, 1)
&= R_{U_{-q^{-1}}(D^{(2)}_{n+1})}\!(V_{sp}\otimes V_{sp}),\quad
V_{sp} = \text{spin rep},
\label{11L}\\
S^{2,1}(1,\ldots, 1)
&= R_{U_{-q^{-1}}(B^{(1)}_{n})}(V_{sp}\otimes V_{sp}),\quad
V_{sp} = \text{spin rep},
\label{21L}\\
S^{2,2}(1,\ldots, 1)
&= R_{U_{-q^{-1}}(D^{(1)}_{n})}(V_{sp}\otimes V_{sp}),\quad
V_{sp} = (\text{spin rep}) \oplus \sigma(\text{spin rep}),
\label{22L}\\
S^{1,1}(0,\ldots, 0)
&= R_{U_q(D^{(2)}_{n+1})}\!(V_{osc}\otimes V_{osc}),\quad
V_{osc} = \text{$q$-oscillator rep},
\label{11R}\\
S^{1,2}(0,\ldots, 0)
&= R_{U_q(A^{(2)}_{2n})}\!(V_{osc}\otimes V_{osc}),\quad
V_{osc} = \text{$q$-oscillator rep},
\label{12R}\\
S^{2,2}(0,\ldots, 0)
&= R_{U_q(C^{(1)}_{n})}\!(V_{osc}\otimes V_{osc}),\quad
V_{osc} = (\text{$q$-osc. rep\!$^+$}) \oplus 
(\text{$q$-osc. rep\!$^-$}).
\label{22R}
\end{align}
The result (\ref{tr0}) is stated in  \cite{BS} where
$V_l \simeq \W_l$ (\ref{wl}) as a vector space.
See also \cite[Appendix B]{KO3} for a proof.
The result (\ref{tr1}) where $V^l \simeq \W_l$ (\ref{wl}) 
has not been stated explicitly in the literature
and it will be covered as a special case in this paper.
The results (\ref{11L})--(\ref{22L}) are due to 
\cite{KS} where $V_{sp} \simeq \W=V^{\otimes n}$.
The $\sigma$ in (\ref{22L}) is the order 2 Dynkin diagram automorphism 
of $D_n$.
See \cite[Remark 7.2]{KS} for $S^{1,2}(1,\ldots, 1)$.
The results (\ref{11R})--(\ref{22R})  and 
the $q$-oscillator representations are obtained in \cite{KO3}
where $V_{osc} \simeq \W = F^{\otimes n}$.
In (\ref{22R}) `$q$-osc. rep$^\pm$' 
denotes the even and odd irreducible sub-representations 
of $V_{osc}$ in \cite[eq.(2.20)]{KO3}.
Such a parity decomposition can also be inferred from (\ref{ice22}).
As for $S^{2,1}(0,\ldots, 0)$, it is reducible to  
$S^{1,2}(0,\ldots, 0)$ by \cite[eq.~(2.16)]{KO3}.

Inhomogeneous cases of $\epsilon_i$'s, i.e.  
{\em mixture} of $\Rm$ and $\mathscr{L}$, was first proposed in 
\cite[sec.~VIII]{S09} for the trace construction and in  
{\cite[sec.~5]{KO2} including the boundary vector construction.
These works manifested that the full problem is much larger
than the homogeneous case and indicated possible connections to  
quantum superalgebras.

This paper is the first systematic study on 
$S^{\mathrm{tr}}(\epsilon_1,\ldots, \epsilon_n)$ and 
$S^{1,1}(\epsilon_1,\ldots, \epsilon_n)$ for general 
inhomogeneous case 
$(\epsilon_1,\ldots, \epsilon_n) \in \{0,1 \}^n$.
The latter is a representative example of the boundary 
vector construction 
$S^{s,t}(\epsilon_1,\ldots, \epsilon_n)$.
The other cases $(s,t)\neq (1,1)$ 
are not included in this paper to avoid complexity of the presentation.

It is not the most essential problem 
nor our primary concern to seek a closed formula
for the matrix elements (\ref{sabij0}) and (\ref{sabij}) by manipulating 
the multiple sum therein.
(See Section \ref{ss:ex} and  Examples \ref{ex:r10}--\ref{ex:rB} however.)
Our main interest lies in the
{\em characterization} of 
$S^{\mathrm{tr}}(\epsilon_1,\ldots, \epsilon_n)$ and 
$S^{1,1}(\epsilon_1,\ldots, \epsilon_n)$
by a quantum group like object in the sense similar to the 
usual $R$ matrices characterized by $U_q({\mathfrak g})$ \cite{D86,Ji}.
As a guide to what will happen, compare 
the two homogeneous cases of $S^{1,1}(z)$ 
in (\ref{11L}) and (\ref{11R}) where
the spin representation of $U_{-q^{-1}}(D^{(2)}_{n+1})$ 
and the $q$-oscillator representations of $U_q(D^{(2)}_{n+1})$
are linked.
Thus it is not only the representation but also the 
algebra itself that are interpolated
with various choices of $(\epsilon_1,\ldots, \epsilon_n)$.
We will show that the resulting family of  
algebras offer examples of generalized 
quantum groups \cite{H,HY} 
which include a class of quantum superalgebras.

\section{Generalized quantum groups and quantum $R$ matrices}\label{sec:gg}

\subsection{Hopf algebras ${\mathcal U}_A(\epsilon_1,\ldots, \epsilon_n)$
and ${\mathcal U}_B(\epsilon_1,\ldots, \epsilon_n)$}
\label{subsec:U}

Set
\begin{align}\label{qidef}
p= \mathrm{i}q^{-\frac{1}{2}},\quad
q_i = \begin{cases}
q & \epsilon_i = 0,\\
-q^{-1} & \epsilon_i = 1,
\end{cases}\;\; (1 \le i \le n),\quad
\tilde{n}= \begin{cases}
n-1 & \; \text{for}\;\; {\mathcal U}_A(\epsilon_1,\ldots, \epsilon_n),\\
n & \; \text{for}\;\; {\mathcal U}_B(\epsilon_1,\ldots, \epsilon_n),
\end{cases}
\end{align}
where $\mathrm{i}=\sqrt{-1}$ and 
$\epsilon_i =0,1$ according to (\ref{Wdef}).
We assume $\tilde{n}\ge 1$ and often write
${\mathcal U}_A(\epsilon_1,\ldots, \epsilon_n)$
and 
${\mathcal U}_B(\epsilon_1,\ldots, \epsilon_n)$ 
as ${\mathcal U}_A$ and ${\mathcal U}_B$ for short.
When considering ${\mathcal U}_A$ all the indices (like $i$ in (\ref{qidef})) 
are to be understood as belonging to $\Z/n\Z$. 
We prepare the constants 
$(D_{i,j})_{0 \le i,j \le \tilde{n}}$ and $(r_i)_{0 \le i \le \tilde{n}}$:
\begin{align}
D_{i,j}&=D_{j,i}=
\prod_{k\in \langle i \rangle \cap \langle j \rangle}
q_k^{2\delta_{i,j}-1},
\qquad
\langle i \rangle = \begin{cases}
\{i,i+1\}\;\;\;& \text{for}\; \;{\mathcal U}_A, \nonumber\\
\{i,i+1\}\cap [1,n]\;\;\;& \text{for}\;\;  {\mathcal U}_B,
\end{cases}\\
r_i &  = q \;\;\text{for}\; \;{\mathcal U}_A,
\qquad 
r_i =\begin{cases}
p & i=0, n,\\
q & 0<i<n
\end{cases}\;\;\text{for}\; \;{\mathcal U}_B.\label{Dr}
\end{align}

\begin{example}
For ${\mathcal U}_A(\epsilon_1,\epsilon_2)$ and 
${\mathcal U}_A(\epsilon_1,\epsilon_2,\epsilon_3)$ one has
\begin{align*}
(D_{i,j})_{0\le i,j\le 1}= \begin{pmatrix}
q_1q_2 & q_1^{-1}q_2^{-1}\\
q_1^{-1}q_2^{-1} & q_1q_2
\end{pmatrix},
\qquad
(D_{i,j})_{0\le i,j\le 2} = \begin{pmatrix}
q_3q_1 & q^{-1}_1 & q^{-1}_3 \\
q^{-1}_1 & q_1q_2 & q^{-1}_2\\
q^{-1}_3 & q^{-1}_2 & q_2q_3 
\end{pmatrix},
\end{align*}
where the top left element is $D_{0,0}$.
Similarly for ${\mathcal U}_B(\epsilon_1,\epsilon_2,\epsilon_3)$ one has
\begin{align*}
(D_{i,j})_{0\le i,j\le 3} = \begin{pmatrix}
q_1 & q^{-1}_1 & 1 & 1\\
q^{-1}_1 & q_1q_2 & q^{-1}_2 & 1\\
1 & q^{-1}_2 & q_2q_3 & q^{-1}_3\\
1 & 1 & q^{-1}_3 & q_3
\end{pmatrix},
\qquad
(r_i)_{0 \le i \le 3} = (p,q,q, p).
\end{align*}
\end{example}

Let ${\mathcal U}_A$ and ${\mathcal U}_B$  
be the $\C(q^{\frac{1}{2}})$-algebras  
generated by $e_i, f_i, k^{\pm 1}_i\,(0 \le i \le \tilde{n})$
obeying the relations
\begin{equation}\label{urel}
\begin{split}
&k_i k^{-1}_i = k^{-1}_ik_i = 1,\quad [k_i, k_j]=0,\\
&k_i e_j = D_{i,j}e_j k_i,\quad k_i f_j = D_{i,j}^{-1}f_j k_i,\quad
[e_i,f_j] = \delta_{i,j}\frac{k_i-k^{-1}_i}{r_i-r^{-1}_i}\quad 
(0 \le i,j  \le \tilde{n}).
\end{split}
\end{equation}
They are Hopf algebras 
with coproduct $\Delta$, counit $\varepsilon$ and 
antipode ${\mathcal S}$ given by
\begin{align}
&\Delta k^{\pm 1}_i = k^{\pm 1}_i\otimes k^{\pm 1}_i,\quad
\Delta e_i = 1\otimes e_i + e_i \otimes k_i,\quad
\Delta f_i = f_i\otimes 1 + k^{-1}_i\otimes f_i, \label{Delta}\\
&\varepsilon(k_i) = 1, \quad \varepsilon(e_i) = \varepsilon(f_i) = 0,\quad
{\mathcal S}(k^{\pm 1}_i) = k_i^{\mp 1},\quad
{\mathcal S}(e_i)=-e_ik^{-1}_i, \quad {\mathcal S}(f_i) =  -k_if_i. \nonumber
\end{align}

With a supplement of appropriate Serre relations, 
the homogeneous cases are identified with the quantum affine algebras 
\cite{D86,Ji} as
\begin{equation}\label{equiv}
\begin{split}
{\mathcal U}_A(0,\ldots, 0) &= U_q(A^{(1)}_{n-1}),\quad
{\mathcal U}_A(1,\ldots, 1) = U_{-q^{-1}}(A^{(1)}_{n-1}),\\
{\mathcal U}_B(0,\ldots, 0) &= U_q(D^{(2)}_{n+1}),\quad
{\mathcal U}_B(1,\ldots, 1)= U_{-q^{-1}}(D^{(2)}_{n+1}).
\end{split}
\end{equation}
In the bottom left case, one actually needs to re-scale 
$f_0$ and $e_n$ by the factor $\mathrm{i}\frac{q+1}{q-1}$.
In the bottom right case, the choice of the branch 
$(-q^{-1})^{1/2}= p$ is assumed.
In general ${\mathcal U}_A(\epsilon_1,\ldots, \epsilon_n)$ 
and ${\mathcal U}_B(\epsilon_1,\ldots, \epsilon_n)$ are 
examples of generalized quantum groups \cite{H,HY}.
We let $\overline{\mathcal U}_A(\epsilon_1,\ldots, \epsilon_n)$ and 
$\overline{\mathcal U}_B(\epsilon_1,\ldots, \epsilon_n)$ denote the 
subalgebras of ${\mathcal U}_A(\epsilon_1,\ldots, \epsilon_n)$
and ${\mathcal U}_B(\epsilon_1,\ldots, \epsilon_n)$
without involving $e_0, f_0$ and $k^{\pm 1}_0$.

\subsection{Representation $\pi_x$}

Recall that $\W_l$ is defined 
by (\ref{wl}), (\ref{syk}) and (\ref{kby}).
Let $x$ be a parameter. 
\begin{proposition}\label{pr:relA}
The map 
$\pi_x: {\mathcal U}_A(\epsilon_1,\ldots, \epsilon_n)
 \rightarrow \mathrm{End}(\W_{l})$ 
defined by\footnote{
Image $\pi_x(g)$ is denoted by $g$ for simplicity.} 
\begin{equation}\label{actsA}
\begin{split}
e_i|{\bf m}\rangle&
= x^{\delta_{i,0}}[m_i]|{\bf m}-{\bf e}_i+{\bf e}_{i+1}\rangle,\\
f_i|{\bf m}\rangle&
= x^{-\delta_{i,0}}[m_{i+1}]|{\bf m}+{\bf e}_i-{\bf e}_{i+1}\rangle,\\
k_i|{\bf m}\rangle&
= (q_i)^{-m_i}(q_{i+1})^{m_{i+1}}|{\bf m}\rangle
\end{split}
\end{equation}
for $i \in \Z/n\Z$ is an irreducible representation,
where $0 \le l \le n$ if $\epsilon_1\cdots\epsilon_n=1$ and 
$l \in \Z_{\ge 0}$ otherwise.
\end{proposition}
\begin{proposition}\label{pr:relD}
The map 
$\pi_x: {\mathcal U}_B(\epsilon_1,\ldots, \epsilon_n) 
\rightarrow \mathrm{End}(\W)$ 
defined by
\begin{equation}\label{actsD}
\begin{split}
e_0|{\bf m}\rangle &
= x|{\bf m}+{\bf e}_1\rangle,\\
f_0|{\bf m}\rangle &
= x^{-1}[m_1]
|{\bf m}-{\bf e}_1\rangle,\\
k_0|{\bf m}\rangle&
= p^{-1}(q_1)^{m_1}|{\bf m}\rangle,\\
e_i|{\bf m}\rangle&
= [m_i]|{\bf m}-{\bf e}_i+{\bf e}_{i+1}\rangle\quad
(0<i<n),\\
f_i|{\bf m}\rangle&
= [m_{i+1}]|{\bf m}+{\bf e}_i-{\bf e}_{i+1}\rangle
\quad(0<i<n),\\
k_i|{\bf m}\rangle&
= (q_i)^{-m_i}(q_{i+1})^{m_{i+1}}|{\bf m}\rangle
\quad (0<i<n),\\
e_n|{\bf m}\rangle&
= [m_n]
|{\bf m}-{\bf e}_n\rangle,\\
f_n|{\bf m}\rangle &
= |{\bf m}+{\bf e}_n\rangle,\\
k_n|{\bf m}\rangle &
= p(q_n)^{-m_n}|{\bf m}\rangle
\end{split}
\end{equation}
is an irreducible representation.
\end{proposition}
We define $\pi_x(k_i^{-1})$ to be $\pi_x(k_i)^{-1}$.
In the rhs of (\ref{actsA}) and (\ref{actsD}), 
vectors $|{\bf m}'\rangle=|m'_1,\ldots, m'_n\rangle$ 
are to be understood as zero
unless $0 \le m'_i \le 1/\epsilon_i$ for all $1 \le i \le n$.
Thus for example in (\ref{actsD}) 
one has $e_0|{\bf e}_1\rangle = 0$ if $\epsilon_1=1$.
Similarly when
($\epsilon_i,\epsilon_{i+1})=(1,1)$, 
$e_i|{\bf m}\rangle$ with $0 < i < n$ is 
non-vanishing if and only if $(m_i,  m_{i+1})=(1,0)$.

Proposition \ref{pr:relA} and \ref{pr:relD} can be directly checked.
The irreducibility of (\ref{actsD}) is seen from 
$\W = {\mathcal U}_B|{\bf 0}\rangle$ and 
$|{\bf 0}\rangle \in  {\mathcal U}_B|{\bf m}\rangle$ for any 
${\bf m}$.

\begin{remark}\label{re:tanseki}
Up to the remark after (\ref{equiv}), 
the representations in Proposition \ref{pr:relA} and \ref{pr:relD}
reduce to the known ones 
in the homogeneous case $\epsilon_1=\cdots = \epsilon_n$:
\begin{align*}
\W_{l} &\simeq \text{$l$-fold symmetric tensor rep. of }   
U_q(A^{(1)}_{n-1})\;\;
\text{for}\;\;\epsilon_1=\cdots = \epsilon_n=0,\\
\W_{l} &\simeq \text{$l$-fold anti-symmetric tensor rep. of }   
U_{-q^{-1}}(A^{(1)}_{n-1}) \;\;
\text{for}\;\;\epsilon_1=\cdots = \epsilon_n=1,\\
\W &\simeq \text{$q$-oscillator rep. of }   U_q(D^{(2)}_{n+1})\;  \text{\cite{KO3}}
\;\; \text{for}\;\;\epsilon_1=\cdots = \epsilon_n=0, \\
\W &\simeq \text{spin rep. of }   U_{-q^{-1}}(D^{(2)}_{n+1})\; \text{\cite{KS}}
\;\; \text{for}\;\;\epsilon_1=\cdots = \epsilon_n=1.
\end{align*}
\end{remark}

\subsection{Relation with quantum superalgebras}\label{ss:qsa}

We adopt the convention in \cite{FSV} (see also \cite{Y91}) for the quantum
 superalgebras $A_q(m,m')$ and $B_q(m,m')$, which are related to the 
$q$-deformations of Lie superalgebras
$sl(m+1,m'+1)$ and $osp(2m+1,2m')$, respectively.

\subsubsection{$A_q$ and ${\mathcal U}_A$}\label{sss:AU}
We compare $A_{q^{1/2}}(\kappa-1,\kappa'-1)$ and
${\mathcal U} _A(0^\kappa,1^{\kappa'}) = {\mathcal U}_A
(\overbrace{0,\ldots,0}^\kappa,\overbrace{1,\ldots,1}^{\kappa'})$ 
with $\kappa + \kappa'=n$.
We assume $0< \kappa < n$.
As an illustration, consider $A_{q^{1/2}}(1,2)$
generated by 
$\tilde{e}_i, \tilde{f}_i, \tilde{k}^{\pm 1}_i\, (1 \le i \le 4)$.
(Tilde is assigned for distinction.)
Replacing $\tilde{k}_i^2$ with $\tilde{k}_i$
they satisfy \cite[eq.(3.2)]{FSV} 
\begin{equation}\label{ra12}
\begin{split}
&\tilde{k}_i \tilde{e}_j = G_{ij} \tilde{e}_j \tilde{k}_i,
\quad \tilde{k}_i \tilde{f}_j = G^{-1}_{ij} \tilde{f}_j \tilde{k}_i,
\quad (G_{i,j})_{1 \le i,j \le 4} = 
\begin{pmatrix}
q^2 & q^{-1} & 1 & 1\\
q^{-1} & 1 & q & 1\\
1 & q & q^{-2} & q\\
1 & 1 & q & q^{-2}
\end{pmatrix},
\\
&[\tilde{e}_1,\tilde{f}_1] = \frac{\tilde{k}_1-\tilde{k}^{-1}_1}{q-q^{-1}},\quad
[\tilde{e}_2,\tilde{f}_2]_+ = \frac{\tilde{k}_2-\tilde{k}^{-1}_2}{q-q^{-1}},\quad
[\tilde{e}_i,\tilde{f}_i] = -\frac{\tilde{k}_i-\tilde{k}^{-1}_i}{q-q^{-1}}\;(i=3,4),\\
&[\tilde{e}_i,\tilde{f}_j] = 0\;(i\neq j),
\end{split}
\end{equation}
where $[a, b]_+ = ab + ba$.
They are also to obey the so-called
Serre relations including, e.g. $\tilde{e}_2^2=0,
\tilde{e}_1^2\tilde{e}_2-[2]
\tilde{e}_1\tilde{e}_2\tilde{e}_1+\tilde{e}_2\tilde{e}_1^2=0$ 
and 
$\tilde{e}_2\tilde{e}_1\tilde{e}_2\tilde{e}_3+\tilde{e}_2\tilde{e}_3\tilde{e}_2\tilde{e}_1
-[2]\tilde{e}_2\tilde{e}_1\tilde{e}_3\tilde{e}_2+\tilde{e}_1\tilde{e}_2\tilde{e}_3\tilde{e}_2
+\tilde{e}_3\tilde{e}_2\tilde{e}_1\tilde{e}_2=0$ 
\cite[eq.(3.4)]{Y91}.

On the other hand 
the subalgebra $\overline{\mathcal U}_A(0^2,1^3)$ 
of ${\mathcal U}_A(0^2,1^3)$ 
(see the end of Section \ref{subsec:U} for the definition) 
has the generators 
$e_i, f_i, k_i^{\pm 1}\, (1 \le i \le 4)$ satisfying the relations (\ref{urel}): 
\begin{equation}\label{ua00111}
\begin{split}
&k_i e_j = D_{ij} e_j k_i,\quad k_i f_j = D^{-1}_{ij} f_j k_i,
\quad (D_{i,j})_{1 \le i,j \le 4} = 
\begin{pmatrix}
q^2 & q^{-1} & 1 & 1\\
q^{-1} & -1 & -q & 1\\
1 & -q & q^{-2} & -q\\
1 & 1 & -q & q^{-2}
\end{pmatrix},
\\
&[e_i,f_j] = \delta_{i,j} \frac{k_i-k^{-1}_i}{q-q^{-1}}.
\end{split}
\end{equation}
The difference of $G_{i,j}$ and $D_{i,j}$ are just by signs.
To compare $\overline{\mathcal U}_A(0^2,1^3)$ with $A_{q^{1/2}}(1,2)$
we add involutive elements $\theta_i$ ($i=2,3,4$) to 
$\overline{\mathcal U}_A(0^2,1^3)$ that (anti-)commute with $e_i,f_i$ as
\begin{align}\label{eft}
[e_i,\theta_j]  =  [f_i, \theta_j] = 0 \leftrightarrow G_{i,j} = D_{i,j},
\quad
[e_i,\theta_j]_+ = [f_i, \theta_j]_+ = 0 \leftrightarrow G_{i,j} = -D_{i,j}.
\end{align}
We assume that $\theta_i$'s commute with $k_j$'s.
Set
\begin{alignat*}{2}
&\tilde{e}_1=e_1,\; \quad\tilde{f}_1=f_1,\;\tilde{k}_1=k_1,\quad&
&\tilde{e}_2=e_2,\; \tilde{f}_2=f_2\theta_2,\;\tilde{k}_2=k_2\theta_2,\\
&\tilde{e}_3=e_3\theta_3,\; \tilde{f}_3=f_3,\;\tilde{k}_3=-k_3\theta_3,&\qquad
&\tilde{e}_4=e_4,\; \tilde{f}_4=f_4\theta_4,\;\tilde{k}_4=-k_4\theta_4.
\end{alignat*}
Then all the relations \eqref{ua00111} are transferred to 
\eqref{ra12}.
On the relevant space $\W=F^{\otimes 2} \otimes V^{\otimes 3}$,
$\theta_i$ ($i=2,3,4$) satisfying (\ref{eft}) is realized as
\begin{equation} \label{Aq_theta}
\theta_2|{\bf m}\rangle=(-1)^{m_3}|{\bf m}\rangle,\quad
\theta_3|{\bf m}\rangle=(-1)^{m_3+m_4}|{\bf m}\rangle,\quad
\theta_4|{\bf m}\rangle=(-1)^{m_4+m_5}|{\bf m}\rangle.
\end{equation}
Moreover the Serre
relations \cite[eq.(3.4)]{Y91} are all valid.
 Hence $\W$ is also an $A_{q^{1/2}}(1,2)$-module.

A similar correspondence holds for general case.
Let $\tilde{e}_i, \tilde{f}_i, \tilde{k}^{\pm 1}_i$ ($\tilde{k}_i$ corresponds to $k_i^2$ in 
\cite{FSV}) denote the generators of $A_{q^{1/2}}(\kappa-1,\kappa'-1)$.
Let $e_i, f_i, k^{\pm 1}_i\, (1 \le i < n)$ denote the 
generators of $\overline{\mathcal U} _A(0^\kappa,1^{\kappa'})$ 
with $\kappa+\kappa' = n$ and $0 < \kappa < n$.
Add involutive elements $\theta_\kappa,\ldots,\theta_{n-1}$ to
$\overline{\mathcal U} _A(0^\kappa,1^{\kappa'})$
that (anti-)commute with $e_i\,(1\le i <n)$ as 
\[
e_i\theta_\kappa=(-1)^{\delta_{i,\kappa}+\delta_{i,\kappa+1}}\theta_\kappa e_i,\;\;
e_i\theta_j=(-1)^{\delta_{i,j-1}+\delta_{i,j+1}}\theta_je_i\;(\kappa<j<n-1),\;\;
e_i\theta_{n-1}=(-1)^{\delta_{i,n-2}}\theta_{n-1}e_i,
\]
with $f_i$ similarly and commute with $k_j$'s. Set
\[
\tilde{e}_i=e_i\theta_i^{\chi(i\equiv \kappa-1)},\quad 
\tilde{f}_i=f_i\theta_i^{\chi(i\equiv \kappa)},\quad 
\tilde{k}_i=(-1)^{1+\delta_{i,\kappa}}k_i\theta_i \quad \text{for}\;\; 
\kappa \le i < n
\]
and $(\tilde{e}_i,\tilde{f}_i,\tilde{k}_i)= (e_i,f_i,k_i)$ for $1 \le i < \kappa$,
where $\chi(\mathrm{true})=1,
\chi(\mathrm{false})=0$ and 
$\equiv$ means the equality mod 2.
Then all the relations of $\overline{\mathcal U} _A(0^\kappa,1^{\kappa'})$
are transferred to those of $A_{q^{1/2}}(\kappa-1,\kappa'-1)$.
On $\W = F^{\otimes \kappa}\otimes V^{\otimes \kappa'}$, such 
$\theta_{\kappa},\ldots, \theta_{n-1}$ are realized as 
$\theta_j|{\bf m}\rangle=
(-1)^{\chi(j>\kappa)m_j+\chi(j\ge\kappa)m_{j+1}}|{\bf m}\rangle$. 
Then the relations 
of $A_{q^{1/2}}(\kappa-1,\kappa'-1)$ 
including the Serre ones are all valid.
Thus we conclude that $\W$ is also an 
$A_{q^{1/2}}(\kappa-1,\kappa'-1)$-module.

\subsubsection{$B_q$ and ${\mathcal U}_B$}
As opposed to the $A_q$ case, we need to take the 
deformation parameter $q$ of $B_q$ to be $p^{-1}$ to adjust conventions.
Thus we compare $B_{p^{-1}} (\kappa',\kappa)$ and
${\mathcal U} _B(0^\kappa,1^{\kappa'})$ 
with $\kappa + \kappa'=n$ for $0 < \kappa < n$.
As an illustration, consider $B_{p^{-1}}(2,2)$ generated by 
$\tilde{e}_i, \tilde{f}_i, \tilde{k}^{\pm 1}_i\, (1 \le i \le 4)$.
Replacing $\tilde{k}_i^2$ with $\tilde{k}_i$ again
they satisfy \cite[eq.(3.2)]{FSV} 
\begin{equation}\label{rd22}
\begin{split}
&\tilde{k}_i \tilde{e}_j = 
G_{ij} \tilde{e}_j \tilde{k}_i,\quad \tilde{k}_i \tilde{f}_j 
= G^{-1}_{ij} \tilde{f}_j \tilde{k}_i,
\quad (G_{i,j})_{1 \le i,j \le 4} = 
\begin{pmatrix}
q^2 & -q^{-1} & 1 & 1\\
-q^{-1} & 1 & -q & 1\\
1 & -q & q^{-2} & -q\\
1 & 1 & -q & -q^{-1}
\end{pmatrix},
\\
&[\tilde{e}_1,\tilde{f}_1] 
= -\frac{\tilde{k}_1-\tilde{k}^{-1}_1}{q-q^{-1}},\quad
[\tilde{e}_2,\tilde{f}_2]_+ 
= -\frac{\tilde{k}_2-\tilde{k}^{-1}_2}{q-q^{-1}},\quad
[\tilde{e}_3,\tilde{f}_3] 
= \frac{\tilde{k}_3-\tilde{k}^{-1}_3}{q-q^{-1}},\\
&[\tilde{e}_4,\tilde{f}_4] 
= \frac{\tilde{k}_4-\tilde{k}^{-1}_4}{p-p^{-1}},\quad
[\tilde{e}_i,\tilde{f}_j] = 0\;(i\neq j)
\end{split}
\end{equation}
and the Serre relations \cite[eq.(3.4)]{Y91}.

On the other hand the subalgebra 
$\overline{\mathcal U}_B(0^2,1^2)$ of ${\mathcal U}_B(0^2,1^2)$
has the generators 
$e_i, f_i, k_i^{\pm 1}\, (1 \le i \le 4)$ satisfying the relations (\ref{urel}): 
\begin{equation}\label{ud0011}
\begin{split}
&k_i e_j = D_{ij} e_j k_i,\quad k_i f_j = D^{-1}_{ij} f_j k_i,
\quad (D_{i,j})_{1 \le i,j \le 4} = 
\begin{pmatrix}
q^2 & q^{-1} & 1 & 1\\
q^{-1} & -1 & -q & 1\\
1 & -q & q^{-2} & -q\\
1 & 1 & -q & -q^{-1}
\end{pmatrix},
\\
&[e_i,f_j] = \delta_{i,j} \frac{k_i-k^{-1}_i}{r_i-r_i^{-1}},\quad
(r_1,\ldots, r_4) = (q,q,q,p).
\end{split}
\end{equation}
To compare $\overline{\mathcal U}_B(0^2,1^2)$ with $B_{p^{-1}}(2,2)$
we add involutive elements $\theta_i$ ($i=1,2$) 
to $\overline{\mathcal U}_B(0^2,1^2)$
that (anti-)commute with $e_i,f_i$ as in \eqref{eft}. We assume that
$\theta_i$'s commute with $k_j$'s. Set 
\begin{alignat*}{2}
&\tilde{e}_1=e_1\theta_1,\; \tilde{f}_1=f_1,\;\tilde{k}_1=-k_1\theta_1,\quad&
&\tilde{e}_2=e_2,\; \tilde{f}_2=f_2\theta_2,\;\tilde{k}_2=-k_2\theta_2,\\
&\tilde{e}_3=e_3,\; \quad \tilde{f}_3=f_3,\;\tilde{k}_3=k_3,&\qquad
&\tilde{e}_4=e_4,\; \tilde{f}_4=f_4,\;\quad\tilde{k}_4=k_4.
\end{alignat*}
Then all the relations \eqref{ud0011} are transferred to \eqref{rd22}.
On $\W = F^{\otimes 2}\otimes V^{\otimes 2}$, 
$\theta_2, \theta_3$ satisfying (\ref{eft}) are realized as
\begin{equation*} 
\theta_1|{\bf m}\rangle=(-1)^{m_1+m_2}|{\bf m}\rangle,\quad
\theta_2|{\bf m}\rangle=(-1)^{m_2}|{\bf m}\rangle.
\end{equation*}
Moreover the Serre
relations \cite[eq.(3.4)]{Y91} are all valid.
Hence $\W$ is also a $B_{p^{-1}}(2,2)$-module.

General case is similar.
Let $\tilde{e}_i, \tilde{f}_i, \tilde{k}^{\pm 1}_i$($\tilde{k}_i$ corresponds to $k_i^2$ in 
\cite{FSV}) denote the generators of $B_{p^{-1}}(\kappa',\kappa)$.
Let $e_i, f_i, k_i^{\pm1}\, (1 \le i \le n)$ denote the generators of 
$\overline{\mathcal U} _B
(0^\kappa,1^{\kappa'})$ 
with $\kappa+\kappa' = n$
and $0<\kappa<n$. 
Add elements $\theta_1,\ldots,\theta_\kappa$ to
$\overline{\mathcal U} _B(0^\kappa,1^{\kappa'})$ 
that (anti-)commute with $e_i\,(1 \le i \le n)$ as
\[
e_i\theta_j=(-1)^{\delta_{i,j-1}+\delta_{i,j+1}}\theta_j e_i\;\; (j<\kappa),\quad
e_i\theta_\kappa=(-1)^{\delta_{i,\kappa}+\delta_{i,\kappa-1}}
\theta_\kappa e_i,
\]
with $f_i$ similarly and commute with $k_j$'s. 
Set 
\[
\tilde{e}_i=e_i\theta_i^{\chi(i\equiv \kappa-1)},\quad 
\tilde{f}_i=f_i\theta_i^{\chi(i\equiv \kappa)},\quad 
\tilde{k}_i=-k_i\theta_i\qquad \text{for }\;\; 1 \le i\le \kappa
\]
($\equiv$ is again mod 2)
and $(\tilde{e}_i,\tilde{f}_i,\tilde{k}_i)= (e_i,f_i,k_i)$ for $\kappa < i \le n$.
Then all the relations of $\overline{\mathcal U} _B(0^\kappa,1^{\kappa'})$ are 
transferred to those of $B_{p^{-1}}(\kappa',\kappa)$.
On $\W = F^{\otimes \kappa} \otimes V^{\otimes \kappa'}$, such 
$\theta_1,\ldots, \theta_{\kappa}$ are realized as 
$\theta_j|{\bf m}\rangle=
(-1)^{\chi(j\le \kappa)m_j+\chi(j<\kappa)m_{j+1}}|{\bf m}\rangle$. 
Then the relations 
of $B_{p^{-1}}(\kappa',\kappa)$ 
including the Serre ones are all valid.
Thus we conclude that $\W$ is also
a $B_{p^{-1}}(\kappa',\kappa)$-module.

\subsection{Quantum $R$ matrices}\label{ss:qr}

Consider the linear equation on 
$R \in \mathrm{End}(\W_l \otimes \W_m)$ for 
${\mathcal U}_A$ and 
$R \in \mathrm{End}(\W\otimes \W)$ for 
${\mathcal U}_B$:
\begin{align}\label{eqrc}
\Delta'(g)R = R \,\Delta(g)
\quad 
\forall g \in {\mathcal U}_A\;\;\text{or}\;\; {\mathcal U}_B,
\end{align}
where $\pi_x\otimes \pi_y$ has been omitted on the both sides and 
$\Delta'$ is the coproduct opposite to (\ref{Delta}).
Namely $\Delta' = P \circ \Delta \circ P$ where
$P(u \otimes v) = v \otimes u$ is the exchange of the components. 
A little inspection of the representations $\pi_x, \pi_y$ tells that 
$R$ depends on $x$ and $y$ only via the ratio $z=x/y$.
Henceforth we write $R$ as $R(z)$.
Suppose $(\epsilon_1,\ldots, \epsilon_n)  = 
(1^\kappa,0^{n-\kappa})$.
For $0 < \kappa < n$ we will show that 
the ${\mathcal U}_A$-module $\W_l \otimes \W_m$ and 
the ${\mathcal U}_B$-module $\W\otimes \W$ are both irreducible
in Propositions \ref{pr:irrAn-1}, \ref{pr:irrAn-2} and \ref{pr:irrD}.
(The same fact holds also for $\kappa=0, n$ due to the 
earlier results mentioned in Remark \ref{re:tanseki}.)
Therefore $R$ is determined (if exists) by postulating (\ref{eqrc}) 
for $g=k_r,e_r$ and $f_r$ with $0 \le r \le \tilde{n}$ up to an overall scalar.
Explicitly these conditions read 
\begin{align}
(k_r \otimes k_r)R(z) &= R(z)(k_r\otimes k_r),\label{kR}\\
(e_r\otimes1 + k_r \otimes e_r) R(z) &= 
R(z)(1\otimes e_r + e_r\otimes k_r),\label{eR}\\
(1\otimes f_r + f_r \otimes k^{-1}_r) R(z) &= 
R(z)(f_r\otimes 1 + k^{-1}_r\otimes f_r)\label{fR}
\end{align}
for $0 \le r \le \tilde{n}$, where 
$\pi_x\otimes \pi_y$ is again omitted.
We call the intertwiner $R(z)$ 
the {\em quantum $R$ matrix}.
Its existence will be established in Theorem \ref{th:Scom},
which also provides an explicit construction. 
From (\ref{sybe}) and Theorem \ref{th:main}
it satisfies the Yang-Baxter equation
\begin{align}\label{yber}
R_{12}(x)R_{13}(xy)R_{23}(y)
= R_{23}(y)R_{13}(xy)R_{12}(x),
\end{align}
which is an equality in  
$\mathrm{End}(\W_k \otimes \W_l \otimes \W_m)$ for some $k,l,m$ for 
${\mathcal U}_A$ and 
$\mathrm{End}(\W \otimes \W \otimes \W)$ for 
${\mathcal U}_B$.

For ${\mathcal U}_B$ 
we introduce a gauge transformed quantum $R$ matrix by
\begin{align}\label{gtr}
{\tilde R}(z) = (K^{-1}\otimes 1) R(z)(1\otimes K),\qquad
K|{\bf m}\rangle = p^{-m_1-\cdots - m_n}|{\bf m}\rangle.
\end{align}
It is easy to see that ${\tilde R}(z)$ 
also satisfies the Yang-Baxter equation.
We fix the normalization of $R(z)$ by 
\begin{alignat}{2}
&R(z)\bigl(|{\bf e}_{> n-l}\rangle \otimes | {\bf e}_{>n-m}\rangle\bigr) =
|{\bf e}_{>n-l}\rangle \otimes | {\bf e}_{>n-m}\rangle&\;&
\text{for } \;{\mathcal U}_A(1,\ldots, 1),\label{rnor1}\\
&R(z)\bigl(|l{\bf e}_i \rangle \otimes |m{\bf e}_i\rangle\bigr)
=  |l{\bf e}_i \rangle \otimes |m{\bf e}_i\rangle&&
\text{for } \;{\mathcal U}_A(\epsilon_1,\ldots, \epsilon_n)
 \;\text{with}\;\epsilon_1\cdots \epsilon_n = 0,\label{rnor2}\\
&{\tilde R}(z)(|{\bf 0}\rangle \otimes |{\bf 0}\rangle)
= |{\bf 0}\rangle \otimes |{\bf 0}\rangle&&
\text{for } \;{\mathcal U}_B(\epsilon_1,\ldots, \epsilon_n),\label{rnor3}
\end{alignat}
where $i$ in (\ref{rnor2}) is taken to be the same as that in (\ref{norm2}).

In view of Section \ref{ss:qsa}, these $R$ matrices with $0 < \kappa <n$ are 
related to quantum superalgebras.
However they do not fall in the known examples, e.g.
\cite{BSh,PS,GLZT} since the structure of the space
$\W = V^{\otimes \kappa}\otimes F^{\otimes n-\kappa}$
is distinct from them.

\begin{example}\label{ex:r10}
Consider ${\mathcal U}_A(1,0)$.
For $l, m \ge 1$, one has $\W_m = \C |0,m\rangle \oplus 
\C|1,m-1\rangle \subset \W=V \otimes F$ and similarly for $\W_l$. 
The action of $R(z)$ on $\W_l \otimes \W_m$ is given by
\begin{align*}
R(z) (|0,l\rangle \otimes |0,m\rangle) 
&= |0,l\rangle \otimes |0,m\rangle,\\
R(z) (|1,l-1\rangle \otimes |0,m\rangle) 
&= \frac{1-q^{2m}}{z-q^{l+m}}|0,l\rangle \otimes |1,m-1\rangle
+\frac{q^mz-q^l}{z-q^{l+m}}|1,l-1\rangle \otimes |0,m\rangle,\\
R(z) (|0,l\rangle \otimes |1,m-1\rangle) 
&= \frac{q^lz-q^m}{z-q^{l+m}}|0,l\rangle \otimes |1,m-1\rangle
+\frac{(1-q^{2l})z}{z-q^{l+m}}|1,l-1\rangle \otimes |0,m\rangle,\\
R(z) (|1,l-1\rangle \otimes |1,m-1\rangle) 
&= \frac{1-q^{l+m}z}{z-q^{l+m}}|1,l-1\rangle \otimes |1,m-1\rangle.
\end{align*}
These formulas are deduced from the spectral decomposition 
(\ref{spd2}).
Equating them to $S^{\mathrm tr}(z|1,0)$ by Theorem \ref{th:main}
already leads to a highly nontrivial identity on the sum (\ref{sabij0})
involving the 3D $R$.
\end{example}

\begin{example}\label{ex:r110}
Consider ${\mathcal U}_A(1,1,0)$.
For $l, m \ge 2$, one has $\W_m = \C |0,0,m\rangle 
\oplus \C|0,1,m-1\rangle \oplus \C|1,0,,m-1\rangle
\oplus \C|1,1,m-2\rangle \subset \W=V \otimes V \otimes F$ and similarly for $\W_l$. 
The action of $R(z)$ on $\W_l \otimes \W_m$ is given by
\begin{align*}
R(z)(|0,i_2,i_3\rangle \otimes |0,j_2,j_3\rangle)&=
[R(z)(|i_2,i_3\rangle \otimes |j_2,j_3\rangle)]|_{|0\bullet \bullet\rangle
\otimes |0\bullet \bullet\rangle},\\
R(z)(|i_1,0,i_3\rangle \otimes |j_1,0,j_3\rangle)&=
[R(z)(|i_1,i_3\rangle \otimes |j_1,j_3\rangle)]|_{|\bullet 0 \bullet\rangle
\otimes |\bullet 0 \bullet\rangle},\\
R(z)(|1,i_2,i_3\rangle \otimes |1,j_2,j_3\rangle)&=
\frac{1-q^{l+m}z}{z-q^{l+m}}
[R(z)(|i_2,i_3\rangle \otimes |j_2,j_3\rangle)]|_{|1\bullet \bullet\rangle
\otimes |1\bullet \bullet\rangle},\\
R(z)(|i_1,1,i_3\rangle \otimes |j_1,1,j_3\rangle)&=
\frac{1-q^{l+m}z}{z-q^{l+m}}
[R(z)(|i_1,i_3\rangle \otimes |j_1,j_3\rangle)]|_{|\bullet 1 \bullet\rangle
\otimes |\bullet 1 \bullet\rangle},\\
R(z)(|0,0,l\rangle \otimes |1,1,m-2\rangle)&=
\frac{q^lz-q^m}{z-q^{l+m}}
[R(zq)(|0,l\rangle\otimes |1,m-2\rangle)]|_{|0\bullet \bullet\rangle
\otimes |1\bullet \bullet\rangle}\\
&+
\frac{(1-q^{2l})z}{z-q^{l+m}}
[R(zq)(|0,l-1\rangle\otimes |1,m-1\rangle)]|_{|1\bullet \bullet\rangle
\otimes |0\bullet \bullet\rangle},\\
R(z)(|1,1,l-2\rangle \otimes |0,0,m\rangle)&=
\frac{q(1-q^{2m})}{z-q^{l+m}}
[R(zq)(|1,l-1\rangle\otimes |0,m-1\rangle)]|_{|0\bullet \bullet\rangle
\otimes |1\bullet \bullet\rangle}\\
&+
\frac{q^mz-q^l}{z-q^{l+m}}
[R(zq)(|1,l-2\rangle\otimes |0,m\rangle)]|_{|1\bullet \bullet\rangle
\otimes |0\bullet \bullet\rangle},\\
R(z)(|1,0,l-1\rangle \otimes |0,1,m-1\rangle)&=
\frac{q^{m-1}(1-q^2)}{z-q^{l+m}}
[R(zq)(|1,l-1\rangle\otimes |0,m-1\rangle)]|_{|0\bullet \bullet\rangle
\otimes |1\bullet \bullet\rangle}\\
&+
\frac{q^mz-q^l}{z-q^{l+m}}
[R(zq)(|0,l-1\rangle\otimes |1,m-1\rangle)]|_{|1\bullet \bullet\rangle
\otimes |0\bullet \bullet\rangle}\\
&+
\frac{1-q^{2m-2}}{z-q^{l+m}}
[R(zq)(|0,l\rangle\otimes |1,m-2\rangle)]|_{|0\bullet \bullet\rangle
\otimes |1\bullet \bullet\rangle},\\
R(z)(|0,1,l-1\rangle \otimes |1,0,m-1\rangle)&=
\frac{q(1-q^{2l-2})z}{z-q^{l+m}}
[R(zq)(|1,l-2\rangle\otimes |0,m\rangle)]|_{|1\bullet \bullet\rangle
\otimes |0\bullet \bullet\rangle}\\
&+
\frac{q^lz-q^m}{z-q^{l+m}}
[R(zq)(|1,l-1\rangle\otimes |0,m-1\rangle)]|_{|0\bullet \bullet\rangle
\otimes |1\bullet \bullet\rangle}\\
&+
\frac{q^{m-1}(1-q^2)z}{z-q^{l+m}}
[R(zq)(|0,l-1\rangle\otimes |1,m-1\rangle)]|_{|1\bullet \bullet\rangle
\otimes |0\bullet \bullet\rangle},
\end{align*}
where $R(z)$ on the rhs is the one in Example \ref{ex:r10} for 
appropriate $l$ and $m$.
The notation $[X]|_{|\bullet 0 \bullet\rangle
\otimes |\bullet 0 \bullet\rangle}$ for instance stands for the 
replacement $|a,b\rangle \otimes |c,d\rangle 
\mapsto |a,0,b\rangle \otimes |c,0,d\rangle$ for 
all the monomials contained in $X$. 
Thus for example, 
\begin{align*}
&R(z)(|1,0,l-1\rangle \otimes |0,1,m-1\rangle) 
=
\frac{(q^{2m}-q^2)(q^m-q^lz)}{(q^{l+m}-z)(q^{l+m}-q^2z)}|0,0,l\rangle \otimes |1,1,m-2\rangle\\
&+\frac{(q^2-1)q^{l+m}+(q^2-q^{2+2l}-q^{2+2m}+q^{2l+2m})z}{(q^{l+m}-z)(q^{l+m}-q^2z)}
|0,1,l-1\rangle \otimes |1,0,m-1\rangle\\
&+\frac{q(q^m-q^lz)(q^l-q^mz)}{(q^{l+m}-z)(q^{l+m}-q^2z)}
|1,0,l-1\rangle \otimes |0,1,m-1\rangle
+\frac{(q^{2l}-q^2)(q^l-q^mz)z}{(q^{l+m}-z)(q^{l+m}-q^2z)}|1,1,l-2\rangle \otimes |0,0,m\rangle.
\end{align*}
Again these formulas are derived from the spectral decomposition (\ref{spd2}) and 
comparison with Example \ref{ex:r10}.
\end{example}

\begin{example}\label{ex:rB}
Consider ${\mathcal U}_B(1,0)$.
$\tilde{R}(z)$ (\ref{gtr}) acts on $\W\otimes \W$ with $\W=V \otimes F$,
satisfies (\ref{iceD}) 
and contains integer powers of $q$ only. 
Proposition \ref{spd4} leads to the following
$(i=0,1)$:
\begin{align*}
\tilde{R}(z)(|i,0\rangle \otimes |i,0\rangle) 
&=  |i,0\rangle \otimes |i,0\rangle,\\
\tilde{R}(z)(|0,0\rangle \otimes |1,0\rangle)
&= \frac{(1+q)z}{1+qz} |1,0\rangle \otimes |0,0\rangle
-\frac{q(1-z)}{1+qz}|0,0\rangle \otimes |1,0\rangle,\\
\tilde{R}(z)(|1,0\rangle \otimes |0,0\rangle)
&= \frac{1+q}{1+qz} |0,0\rangle \otimes |1,0\rangle
+\frac{1-z}{1+qz}|1,0\rangle \otimes |0,0\rangle,\\
\tilde{R}(z)(|i,1\rangle \otimes |i,0\rangle)
&= \frac{1+q}{1+qz}|i,0\rangle \otimes |i,1\rangle
+\frac{1-z}{1+qz}|i,1\rangle \otimes |i,0\rangle,\\
\tilde{R}(z)(|1,1\rangle \otimes |0,0\rangle)
&= \frac{(1+q)(1+q^2)}{(1+qz)(1+q^2z)}|0,0\rangle \otimes |1,1\rangle
+\frac{q(1+q)(1-z)}{(1+qz)(1+q^2z)}|0,1\rangle \otimes |1,0\rangle\\
&+\frac{(1+q)(1-z)}{(1+qz)(1+q^2z)}|1,0\rangle \otimes |0,1\rangle
+\frac{(1-z)(1-qz)}{(1+qz)(1+q^2z)}|1,1\rangle \otimes |0,0\rangle,\\
\tilde{R}(z)(|0,1\rangle \otimes |1,0\rangle)
&=-\frac{q(1+q)(1-z)}{(1+qz)(1+q^2z)}|0,0\rangle \otimes |1,1\rangle
-\frac{q(1-z)(1-qz)}{(1+qz)(1+q^2z)}|0,1\rangle \otimes |1,0\rangle\\
&+\frac{(1+q)z(1+q-q(1-q)z)}{(1+qz)(1+q^2z)}|1,0\rangle \otimes |0,1\rangle
+ \frac{(1+q)(1-z)z}{(1+qz)(1+q^2z)}|1,1\rangle \otimes |0,0\rangle,\\
\tilde{R}(z)(|i,2\rangle \otimes |i,0\rangle) 
&= \frac{(1+q)(1+q^2)}{(1+qz)(1+q^2z)}|i,0\rangle \otimes |i,2\rangle
+\frac{(1+q)(1+q^2)(1-z)}{(1+qz)(1+q^2z)}|i,1\rangle \otimes |i,1\rangle\\
&+\frac{(1-z)(1-qz)}{(1+qz)(1+q^2z)}|i,2\rangle \otimes |i,0\rangle,\\
\tilde{R}(z)(|i,1\rangle \otimes |i,1\rangle) 
&=-\frac{q(1+q)(1-z)}{(1+qz)(1+q^2z)}|i,0\rangle \otimes |i,2\rangle
+\frac{(1+q)(1+q+q^2)z-q(q+z^2)}{(1+qz)(1+q^2z)}
|i,1\rangle \otimes |i,1\rangle\\
&+\frac{(1+q)(1-z)z}{(1+qz)(1+q^2z)}|i,2\rangle \otimes |i,0\rangle,\\
\tilde{R}(z)(|i,0\rangle \otimes |i,2\rangle) 
&= \frac{q^2(1-z)(1-qz)}{(1+qz)(1+q^2z)}|i,0\rangle \otimes |i,2\rangle
-\frac{q(1+q)(1+q^2)(1-z)z}{(1+qz)(1+q^2z)}
|i,1\rangle \otimes |i,1\rangle\\
&+\frac{(1+q)(1+q^2)z^2}{(1+qz)(1+q^2z)}|i,2\rangle \otimes |i,0\rangle.
\end{align*}
\end{example}

\section{Main result: 
$S^{\mathrm{tr}}(z)$ and $S^{1,1}(z)$ as quantum $R$ matrices}\label{sec:qR}

Denote the $R(z)$ for  
${\mathcal U}_A(\epsilon_1,\ldots, \epsilon_n)$ by
$R_A(z|\epsilon_1,\ldots, \epsilon_n)$ 
and ${\tilde R}(z)$ (\ref{gtr}) for  
${\mathcal U}_B(\epsilon_1,\ldots, \epsilon_n)$ by
${\tilde R}_B(z|\epsilon_1,\ldots, \epsilon_n)$.
The main result of this paper is the following.
\begin{theorem}\label{th:main}
Suppose $(\epsilon_1,\ldots, \epsilon_n)  = 
(1^\kappa, 0^{n-\kappa})$.
For any $0 \le \kappa \le n$ the following identification holds:
\begin{align*}
S^{\mathrm{tr}}(z|\epsilon_1,\ldots, \epsilon_n) 
= R_A(z|\epsilon_1,\ldots, \epsilon_n),\qquad
S^{1,1}(z|\epsilon_1,\ldots, \epsilon_n) 
= {\tilde R}_B(z|\epsilon_1,\ldots, \epsilon_n).
\end{align*}
\end{theorem}

The two equalities hold in 
$\mathrm{End}(\W_l \otimes \W_m)$ for each $l,m$ and in 
$\mathrm{End}(\W \otimes \W)$, respectively.
Combined with Proposition \ref{pr:equiv},
Theorem \ref{th:main} tells that 
$S^{\mathrm{tr}}(z|\epsilon_1,\ldots, \epsilon_n)$ and 
$S^{1,1}(z|\epsilon_1,\ldots, \epsilon_n)$ with 
{\em arbitrary} $(\epsilon_1,\ldots, \epsilon_n) \in \{0,1\}^n$
are equivalent to the quantum $R$ matrices of the generalized quantum groups.
For $(\epsilon_1,\ldots, \epsilon_n)$ not of the above form,
the right hand sides are yet to be characterized uniquely.
See also the comments below.

The rest of the paper is devoted to a proof of 
Theorem \ref{th:main}.
It consists of three Parts.
In Part I (Section \ref{sec:proof}) we prove that 
$S^{\mathrm{tr}}$ and $S^{1,1}$ possess the same 
commutativity (\ref{eqrc}) with ${\mathcal U}_A$ and ${\mathcal U}_B$ as 
the quantum $R$ matrices (Theorem \ref{th:Scom}).
This will be done for an arbitrary sequence
$(\epsilon_1,\ldots, \epsilon_n)  \in \{0,1\}^n$.
In Part II (Section \ref{sec:IrreA}) and Part III (Section \ref{sec:IrreD}) 
we show that the relevant ${\mathcal U}_A$-module 
${\mathcal W}_l \otimes {\mathcal W}_m$ and 
the ${\mathcal U}_B$-module 
${\mathcal W} \otimes {\mathcal W}$ are irreducible
for the choice $(\epsilon_1,\ldots, \epsilon_n)  = (1^\kappa, 0^{n-\kappa})$.
This is an indispensable claim to guarantee that the $R$ matrices are 
characterized as the commutant of ${\mathcal U}_A$ and ${\mathcal U}_B$
up to a normalization.
Finally the agreement of the normalization is assured by 
(\ref{norm1})--(\ref{norm3}) and 
(\ref{rnor1})--(\ref{rnor3}).
We have not proved the irreducibility of 
${\mathcal U}_A$-module 
${\mathcal W}_l \otimes {\mathcal W}_m$ and 
the ${\mathcal U}_B$-module 
${\mathcal W} \otimes {\mathcal W}$
for $(\epsilon_1,\ldots, \epsilon_n)$ not of the above form
although we expect they are so.

In Part II and III we will utilize the following fact.
\begin{proposition} \label{pr:ht*lt}
Let $\geh$ be a finite-dimensional simple Lie algebra and $U_q(\geh)$ be
its quantized enveloping algebra with the standard Drinfeld-Jimbo generators 
$e_i, f_i, k^{\pm 1}_i$. Let $V,V'$ be irreducible $U_q(\geh)$-modules.
Let $u_h,u_l$ be highest and lowest weight vectors of $V$, i.e.
$e_i u_h = f_iu_l = 0$ for all $i$. 
Let $u'_h,u'_l$ be highest and lowest weight vectors of $V'$.
Then each of the vectors $u_h\ot u'_l$ and $u_l\ot u'_h$ generates $V\ot V'$.
\end{proposition}

\begin{proof}
Set $M=U_q(\geh)(u_h\ot u'_l)$. It is easy to see that $M$ contains $u_h \otimes V'$.
We are left to show that for a weight vector $u$, if $u \otimes V'$ is contained in $M$,
then $f_i u \otimes V'$ is also contained in $M$. Note that for any $u'\in V'$,
$\Delta(f_i)(u \otimes u')=f_i u \otimes u' + k_i^{-1} u \otimes f_i u'$.
The lhs belongs to $M$ by the definition of $M$ and so does the second
term of the rhs by the assumption. Hence $f_i u \otimes u' \in M$. The other case is
similar.
\end{proof}

\section{Proof Part I: 
Commutativity with ${\mathcal U}_A$ and ${\mathcal U}_B$.}\label{sec:proof}

For $S^{1,1}(z)$ we introduce a slight gauge transformation
\begin{align}\label{ksk}
\tilde{S}^{1,1}(z) = (K\otimes 1)S^{1,1}(z)(1\otimes K^{-1}),
\end{align}
where $K$ is defined in (\ref{gtr}).
The main property of $S^{\mathrm{tr}}(z)$ and $S^{1,1}(z)$ is the 
commutativity with the generalized quantum groups 
${\mathcal U}_A$ and ${\mathcal U}_B$ identical with (\ref{eqrc}).
\begin{theorem}\label{th:Scom}
For an arbitrary sequence $(\epsilon_1,\ldots, \epsilon_n)  \in \{0,1\}^n$,
the following commutativity holds: 
\begin{align*}
\Delta'(g)S^{\mathrm{tr}}(z|\epsilon_1,\ldots, \epsilon_n)
 &= S^{\mathrm{tr}}(z|\epsilon_1,\ldots, \epsilon_n) \Delta(g)\quad 
\; \forall g \in {\mathcal U}_A(\epsilon_1,\ldots, \epsilon_n),\\
\Delta'(g) \tilde{S}^{1,1}(z|\epsilon_1,\ldots, \epsilon_n)
 &= \tilde{S}^{1,1}(z|\epsilon_1,\ldots, \epsilon_n)\Delta(g)\quad 
 \forall g \in {\mathcal U}_B(\epsilon_1,\ldots, \epsilon_n),
\end{align*} 
where $\Delta(g)$ and $\Delta'(g)$ stand for the tensor product
representation 
$(\pi_x \otimes \pi_y)\Delta(g)$ and $(\pi_x \otimes \pi_y)\Delta'(g)$ of those 
in Proposition \ref{pr:relA} and \ref{pr:relD} and $z=x/y$.
\end{theorem} 

\begin{proof}
It suffices to show that 
$S(z)=S^{\mathrm{tr}}(z)$ and  $S^{1,1}(z)$ satisfy 
\begin{align}
(k_r \otimes k_r)S(z) &= S(z)(k_r\otimes k_r),\label{ck}\\
(\tilde{e}_r\otimes 1+k_r\ot e_r) S(z) &= 
S(z)(1\ot \tilde{e}_r+ e_r\otimes k_r),\label{ce}\\
(1\otimes f_r +\tilde{f}_r \otimes k^{-1}_r) S(z) &= 
S(z)(f_r\otimes 1 + k^{-1}_r\otimes \tilde{f}_r)\label{cf}
\end{align}
for $0 \le r \le \tilde{n}$, where 
$\tilde{e}_r = e_r, \tilde{f}_r = f_r$ for $S^{\mathrm{tr}}(z)$ and 
$\tilde{e}_r = K^{-1}e_r K,\tilde{f}_r = K^{-1}f_r K$ for $S^{1,1}(z)$. 
$\pi_x \otimes \pi_y$ is again omitted.
It is easy to see that (\ref{ck}) is guaranteed by (\ref{iceA})--(\ref{iceD}).
In what follows we demonstrate a proof of (\ref{ce}).
The relation (\ref{cf}) can be verified similarly.
It unifies the earlier proofs 
for $S^{1,1}(z|1,\ldots,1)$ in \cite{KS} ,
$S^{1,1}(z|0,\ldots,0)$ in \cite{KO3} and 
$S^{\mathrm{tr}}(z|0,\ldots,0)$ in \cite[Prop.17]{KO3}.

Consider the action of the both sides of (\ref{ce}) on 
a base vector $|{\bf i}\rangle \otimes |{\bf j}\rangle \in 
\W\otimes \W$:
\begin{align}
y^{-\delta_{r,0}}
(\tilde{e}_r\otimes 1+k_r\ot e_r)S(z)
(|{\bf i}\rangle \otimes |{\bf j}\rangle) 
&= 
\sum_{{\bf a}, {\bf b}}A^{\bf{a,b}}_{\bf{i,j}}(z)
|{\bf a}\rangle \otimes |{\bf b}\rangle,\label{Ah}\\
y^{-\delta_{r,0}}
S(z)(1\ot \tilde{e}_r+ e_r\otimes k_r)
(|{\bf i}\rangle \otimes |{\bf j}\rangle) 
&= 
\sum_{{\bf a}, {\bf b}}B^{\bf{a,b}}_{\bf{i,j}}(z)
|{\bf a}\rangle \otimes |{\bf b}\rangle,\label{Bh}
\end{align}
where we have 
multiplied $y^{-\delta_{r,0}}$ to confine the dependence on 
$x$ and $y$ to the ratio $z=x/y$.
We are to show the equality of the matrix elements
$A^{\bf{a,b}}_{\bf{i,j}}(z)=B^{\bf{a,b}}_{\bf{i,j}}(z)$.

(i) Case $0 \le r <n$ for $S^{\mathrm{tr}}(z)$ and 
Case $0 < r < n$ for $S^{1,1}(z)$.
$\tilde{e}_r = e_r$ holds also for $S^{1,1}(z)$.
For $S^{\mathrm{tr}}(z)$, the index $0$ is to be identified with $n$.
The action of $e_r$ and $k_r$  in  (\ref{actsA}) and (\ref{actsD}) only 
concerns the $r$\,th and $(r\!+\!1)$th components
$|m_r\rangle^{(\epsilon_r)} \otimes 
|m_{r+1}\rangle^{(\epsilon_{r+1})}$ of $\W$.
Denoting them simply by $|m_r, m_{r+1} \rangle$, 
we depict (\ref{Ah})  by the following diagram: 
\[
\begin{picture}(120,90)(-40,-80)

\put(-23,4){$|i_r, i_{r+1}\rangle\otimes |j_r, j_{r+1}\rangle$}

\put(-8,-3){\vector(-1,-1){25}}
\put(-16,-22){$S(z)$}

\put(51,-3){\vector(1,-1){25}}
\put(39,-22){$S(z)$}

\put(-125,-40){$|a_r\!+\!1,a_{r+1}\!-\!1\rangle\otimes |b_r,b_{r+1}\rangle$}
\put(60,-40){$|a_r,a_{r+1}\rangle\otimes |b_r\!+\!1,b_{r+1}\!-\!1\rangle$}

\put(-33,-47){\vector(1,-1){25}}
\put(-19,-57){$e_r\!\otimes\! 1$}
\put(-80,-67){$z^{\delta_{r,0}}[a_r\!+\!1]$}
\put(76,-47){\vector(-1,-1){25}}
\put(31,-57){$k_r\!\otimes\! e_r$}
\put(68,-67){$(q_r)^{-a_r}(q_{r+1})^{a_{r+1}}[b_r\!+\!1]$}

\put(-23,-85){$|a_r,a_{r+1}\rangle\otimes |b_r,b_{r+1}\rangle$}
\end{picture}
\]
Thus we have
\begin{align}
A^{\bf{a,b}}_{\bf{i,j}}(z)
=\sum_{c_0,\ldots, c_n}&z^{c_0+\delta_{r,0}}
U(c_0, \ldots, c_{r-1}, c_{r+1}, \ldots, c_n)
\Bigl([a_r\!+\!1]
{\mathscr S}^{(\epsilon_r)\, a_r+1, b_r,c_{r-1}}_{
\phantom{(\epsilon_r)}\,i_r,j_r,c_r}
{\mathscr S}^{(\epsilon_{r+1})\, a_{r+1}-1,b_{r+1},c_r}_{
\phantom{(\epsilon_{r+1})}\, i_{r+1},j_{r+1},c_{r+1}}\nonumber\\
&+(q_r)^{-a_r}(q_{r+1})^{a_{r+1}}[b_r\!+\!1]
{\mathscr S}^{(\epsilon_r)\, a_r,b_r+1,c_{r-1}}_{
\phantom{(\epsilon_r)}\,i_r,j_r,c_r+1}
{\mathscr S}^{(\epsilon_{r+1})\, a_{r+1},b_{r+1}-1,c_r+1}_{
\phantom{(\epsilon_{r+1})}\,i_{r+1},j_{r+1},c_{r+1}}\Bigr)
\label{A3}
\end{align}
for some $U(c_0, \ldots, c_{r-1}, c_{r+1}, \ldots, c_n)$ 
which is independent of $z$.
In the second term we have shifted the dummy summation variable 
$c_r$ to $c_r+1$. 
This has the effect of letting the two terms have the identical 
constraints $b_{l}+c_{l-1} = j_{l}+c_{l}\, (l=r,r+1)$ and the
common $z$-dependence $z^{c_0 + \delta_{r,0}}$.
Similarly the diagram for (\ref{Bh}) looks as
\[
\begin{picture}(120,90)(-40,-80)

\put(-23,4){$|i_r, i_{r+1}\rangle\otimes |j_r, j_{r+1}\rangle$}

\put(-8,-3){\vector(-1,-1){25}}
\put(-18,-22){$1\!\otimes\! e_r$}

\put(-42,-12){$[j_r]$}

\put(51,-3){\vector(1,-1){25}}
\put(32,-22){$e_r\!\otimes\! k_r$}

\put(68,-12){$z^{\delta_{r,0}}(q_r)^{-j_r}(q_{r+1})^{j_{r+1}}[i_r]$}

\put(-125,-40){$|i_r,i_{r+1}\rangle\otimes |j_r\!-\!1,j_{r+1}\!+\!1\rangle$}
\put(60,-40){$|i_r\!-\!1,i_{r+1}\!+\!1\rangle\otimes |j_r,j_{r+1}\rangle$}

\put(-33,-47){\vector(1,-1){25}}
\put(-19,-57){$S(z)$}

\put(76,-47){\vector(-1,-1){25}}
\put(40,-57){$S(z)$}

\put(-23,-85){$|a_r,a_{r+1}\rangle\otimes |b_r,b_{r+1}\rangle$}
\end{picture}
\]
This leads to the expression
\begin{align}
B^{\bf{a,b}}_{\bf{i,j}}(z)
=\sum_{c_0,\ldots, c_n}&z^{c_0+\delta_{r,0}}
U(c_0, \ldots, c_{r-1}, c_{r+1}, \ldots, c_n)
\Bigl([j_r]
{\mathscr S}^{(\epsilon_r)\, a_r, b_r,c_{r-1}}_{
\phantom{(\epsilon_r)}\,i_r,j_r-1,c_r+1}
{\mathscr S}^{(\epsilon_{r+1})\, a_{r+1},b_{r+1},c_r+1}_{
\phantom{(\epsilon_{r+1})}\,i_{r+1},j_{r+1}+1,c_{r+1}} \nonumber\\
&+(q_r)^{-j_r}(q_{r+1})^{j_{r+1}}[i_r]
{\mathscr S}^{(\epsilon_r)\, a_r,b_r,c_{r-1}}_{
\phantom{(\epsilon_r)}\,i_r-1,j_r,c_r}
{\mathscr S}^{(\epsilon_{r+1})\, a_{r+1},b_{r+1},c_r}_{
\phantom{(\epsilon_{r+1})}\,i_{r+1}+1,j_{r+1},c_{r+1}}\Bigr)
\label{B3}
\end{align}
with the same $U(c_0, \ldots, c_{r-1}, c_{r+1}, \ldots, c_n)$ as (\ref{A3}).
This time $c_r$ has been shifted to $c_r+1$ in the first term 
by the same reason as in (\ref{A3}).
Comparing (\ref{A3}) and (\ref{B3}) and noting the 
conservation law (\ref{ice0}), we find that 
$A^{\bf{a,b}}_{\bf{i,j}}(z)=B^{\bf{a,b}}_{\bf{i,j}}(z)$
is reduced to the equality of the 
quantities in the parenthesis.
Writing $(a_r, b_r, i_r,  j_r)$ as  $(a,b,i,j)$,
$(a_{r+1}, b_{r+1}, i_{r+1},  j_{r+1})$ as $(a',b',i',j')$,
$(c_{r-1}, c_r, c_{r+1})$ as $(c, k, k')$,
$(\epsilon_r, \epsilon_{r+1})$ as $(\epsilon, \epsilon')$ and 
$(q_r, q_{r+1})$ as $(\rho, \rho')$, 
it reads
\begin{equation}\label{sseq}
\begin{split}
&[a+1]\,
{\mathscr S}^{(\epsilon)\,a+1, b, c}_{\phantom{(\epsilon)}\,i,  j, k}\;
{\mathscr S}^{(\epsilon')\,a'-1, b', k}_{\phantom{(\epsilon)}\,i',  j', k'}
+\rho^{-a}(\rho')^{a'}[b+1]\,
{\mathscr S}^{(\epsilon)\,a, b+1, c}_{\phantom{(\epsilon)}\,i,  j, k+1}\;
{\mathscr S}^{(\epsilon')\,a', b'-1, k+1}_{\phantom{(\epsilon)}\,i', j', k'}\\
&=
[j]\,
{\mathscr S}^{(\epsilon)\,a, b, c}_{\phantom{(\epsilon)}\,i, j-1, k+1}\;
{\mathscr S}^{(\epsilon')\,a', b', k+1}_{\phantom{(\epsilon)}\,i',  j'+1, k'}
+\rho^{-j}(\rho')^{j'}[i]\,
{\mathscr S}^{(\epsilon)\,a, b, c}_{\phantom{(\epsilon)}\,i-1, j, k}\;
{\mathscr S}^{(\epsilon')\,a', b', k}_{\phantom{(\epsilon)}\,i'+1,  j', k'}.
\end{split}
\end{equation}
Note that $(\rho,\rho') = ((-1)^\epsilon q^{1-2\epsilon},
(-1)^{\epsilon'} q^{1-2\epsilon'})$
by the definition (\ref{qidef}).
For $(\epsilon,\epsilon')=(0,0)$, (\ref{sseq}) coincides with \cite[eq.(A.16)]{KO3}.
For $(\epsilon,\epsilon')=(1,1)$,  (\ref{sseq})  reads
\begin{align*}
\delta^{a,a'}_{0,1}{\mathscr L}^{1,b}_{i,j}{\mathscr L}^{0,b'}_{i',j'}
+\delta^{b,b'}_{0,1}(-q^{-1})^{-a+a'}{\mathscr L}^{a,1}_{i,j}{\mathscr L}^{a',0}_{i',j'}
= 
\delta^{1,0}_{j,j'}{\mathscr L}^{a,b}_{i,0}{\mathscr L}^{a',b'}_{i',1}
+
\delta^{1,0}_{i,i'}(-q^{-1})^{-j+j'}{\mathscr L}^{a,b}_{0,j}{\mathscr L}^{a',b'}_{1,j'},
\end{align*}
where all the indices are in $\{0,1\}$ and 
$\delta^{a,a'}_{0,1}= \delta^a_0\delta^{a'}_1$, etc.
These $2^8$ relations can directly 
be checked by substituting (\ref{lak})  and using (\ref{ac2}).
For $(\epsilon,\epsilon')=(0,1)$,  (\ref{sseq})  reads
\begin{align*}
&\delta^{a'}_{1}[a+1]\Rm^{a+1,b,c}_{i,j,k}{\mathscr L}^{0,b',k}_{i',j',k'}
+\delta^{b'}_{1}q^{-a}(-q^{-1})^{a'}[b+1]\Rm^{a,b+1,c}_{i,j,k+1}
{\mathscr L}^{a',0,k+1}_{i',j',k'}\\
&- 
\delta^{0}_{j'}[j]\Rm^{a,b,c}_{i,j-1,k+1}{\mathscr L}^{a',b',k+1}_{i',1,k'}
-
\delta^{0}_{i'}q^{-j}(-q^{-1})^{j'}[i]
\Rm^{a,b,c}_{i-1,j,k}{\mathscr L}^{a',b',k}_{1,j',k'}=0.
\end{align*}
Among $2^4$ choices of $(a',b',i',j')$, there are four  
that lead to nontrivial relations for some values of $k'-k$.
They are given by 
$(a',b',i',j',k')=(1,0,0,0,k), \,
(1,1,0,1,k),\,
(0,1,0,0,k+1)$ and $(1,1,1,0,k+1)$.
The corresponding relations read
\begin{align}
&[a+1] \Rm^{a+1,b,c}_{i,j,k} - [j] \Rm^{a,b,c}_{i,j-1,k+1}
-q^{-j+k}[i] \Rm^{a,b,c}_{i-1,j,k}=0,\label{hkr1}\\
&q^{k+1}[a+1]\Rm^{a+1,b,c}_{i,j,k}+q^{-a-1}[b+1]\Rm^{a,b+1,c}_{i,j,k+1}
-q^{-j-1}[i]\Rm^{a,b,c}_{i-1,j,k}=0,\label{hkr2}\\
&q^{-a}[b+1]\Rm^{a,b+1,c}_{i,j,k+1}
+q^{k+2}[j]\Rm^{a,b,c}_{i,j-1,k+1}-
q^{-j}(1-q^{2k+2})[i]\Rm^{a,b,c}_{i-1,j,k}=0,\label{hkr3}\\
&(1-q^{2k+2})[a+1]\Rm^{a+1,b,c}_{i,j,k}-q^{-a+k}[b+1]\Rm^{a,b+1,c}_{i,j,k+1}
-[j]\Rm^{a,b,c}_{i,j-1,k+1}=0. \label{hkr4}
\end{align}
Eqs.  (\ref{hkr1}) and (\ref{hkr3}) are equivalent to 
the known identities named $t_{2 1}$ and $t_{1 1}$ 
appearing before \cite[eq.(A.2)]{KO3}.
Any element of $\Rm$  in (\ref{hkr2}) and (\ref{hkr4}) can be 
converted into the form $\Rm^{a,b,c}_{\bullet, \bullet,\bullet}$ 
by using (\ref{hkr1})  to decrease $a$ and (\ref{hkr3})  to decrease $b$.
The resulting expressions turn out to be identically zero 
because of $(a+b,b+c)=(i+j-1,j+k)$. See (\ref{Rex}).
For $(\epsilon,\epsilon')=(1,0)$,  (\ref{sseq})  reads
\begin{align*}
&\delta^0_a {\mathscr L}^{1,b,c}_{i,j,k}\;\Rm^{a'-1,b',k}_{i',j',k'}
+\delta^0_b(-q^{-1})^{-a}q^{a'}
{\mathscr L}^{a,1,c}_{i,j,k+1}\;\Rm^{a',b'-1,k+1}_{i',j',k'}\\
&- \delta^1_j {\mathscr L}^{a,b,c}_{i,0,k+1}\;\Rm^{a',b',k+1}_{i',j'+1,k'}
-\delta_i^1 (-q^{-1})^{-j}q^{j'}
{\mathscr L}^{a,b,c}_{0,j,k}\;\Rm^{a',b',k}_{i'+1,j',k'} = 0.
\end{align*}
Among $2^4$ choices of $(a,b,i,j)$, there are four  
that lead to nontrivial relations for some values of $c-k$.
They are given by 
$(a,b,i,j,c)=(0,0,0,1,k+1),\, (1,0,1,1,k+1),\,
(0,0,1,0,k)$ and $(0,1,1,1,k)$.
The corresponding relations, after removing primes, read
\begin{align}
&\Rm^{a-1,b,c}_{i,j,k}-q^{a+c+2}\Rm^{a,b-1,c+1}_{i,j,k}
-\Rm^{a,b,c+1}_{i,j+1,k}=0, \label{hzk}\\
&q^{a+1}\Rm^{a,b-1,c}_{i,j,k-1}+q^c\Rm^{a,b,c}_{i,j+1,k-1} - 
q^{j+1}\Rm^{a,b,c-1}_{i+1,j,k-1}=0,\label{ann}\\
&q^c\Rm^{a-1,b,c}_{i,j,k}+q^a(1-q^{2c+2})\Rm^{a,b-1,c+1}_{i,j,k}
-q^j\Rm^{a,b,c}_{i+1,j,k}=0,\label{mak}\\
&\Rm^{a-1,b,c}_{i,j,k}-(1-q^{2c+2})\Rm^{a,b,c+1}_{i,j+1,k}
-q^{c+j+2}\Rm^{a,b,c}_{i+1,j,k}=0. \label{yum}
\end{align}
Eqs. (\ref{hzk}) and (\ref{mak}) are equivalent to 
(A.3) and (A.2) in \cite{KO3}, respectively.
Any element of $\Rm$  in (\ref{ann}) (resp. (\ref{yum})) can be 
converted into the form 
$\Rm_{i,j,k-1}^{\bullet, \bullet, \bullet}$ 
(resp. $\Rm_{i,j,k}^{\bullet, \bullet,\bullet}$)
by using (\ref{mak}) to decrease $i$ and (\ref{hzk}) to decrease $j$.
The resulting expressions are identically zero.

(ii) Case $r=n$ for $S^{1,1}(z)$.
The action of $e_n$ and $k_n$ in (\ref{actsD}) only concern
the $n$\,th component $|m_n\rangle^{(\epsilon_n)}$ 
of $|{\bf m}\rangle$.
Denoting it simply by $|m_n\rangle$, we depict (\ref{Ah}) as
\[
\begin{picture}(120,90)(-35,-80)

\put(0,0){$|i_n\rangle\otimes |j_n\rangle$}

\put(-8,-3){\vector(-1,-1){25}}
\put(-16,-22){$S(z)$}

\put(51,-3){\vector(1,-1){25}}
\put(38,-22){$S(z)$}

\put(-65,-40){$|a_n\!+\!1\rangle\otimes |b_n\rangle$}
\put(55,-40){$|a_n\rangle\otimes |b_n\!+\!1\rangle$}

\put(-33,-47){\vector(1,-1){25}}
\put(-19,-57){${\tilde e}_n\!\otimes\! 1$}
\put(-75,-67){$p^{-1}[a_n+1]$}

\put(76,-47){\vector(-1,-1){25}}
\put(29,-57){$k_n\!\otimes\! e_n$}
\put(68,-67){$p(q_n)^{-a_n}[b_n+1]$}

\put(0,-77){$|a_n\rangle\otimes |b_n\rangle$}
\end{picture}
\]
where ${\tilde e}_n = p^{-1}e_n$ has been used. 
Thus we have
\begin{equation}\label{A2}
\begin{split}
p^{-1}A^{\bf{a,b}}_{\bf{i,j}}(z)
&=\sum_{c_0,\ldots, c_n}
\frac{z^{c_0}}{(q)_{c_n}}X(c_0,\ldots, c_{n-1})\\
\times &\left(p^{-2}[a_n+1](1-q^{c_n})
{\mathscr S}^{(\epsilon_n)\,a_n+1,b_n,c_{n-1}}_{
\phantom{(\epsilon_n)}\, i_n,j_n,c_n-1}
+(q_n)^{-a_n}[b_n+1]
{\mathscr S}^{(\epsilon_n)\,a_n,b_n+1,c_{n-1}}_{
\phantom{(\epsilon_n)}\, i_n,j_n,c_n}\right),
\end{split}
\end{equation}
where $c_n$ has been shifted to $c_n-1$ in the first term.
$X(c_0,\ldots, c_{n-1})$ is independent of $z$.
Similarly (\ref{Bh}) with $r=n$ is depicted as
\[
\begin{picture}(120,90)(-40,-80)

\put(0,0){$|i_n\rangle\otimes |j_n\rangle$}

\put(-56,-12){$p^{-1}[j_n]$}
\put(-8,-3){\vector(-1,-1){25}}
\put(-16,-22){$1\!\otimes\! {\tilde e}_n$}

\put(51,-3){\vector(1,-1){25}}
\put(30,-22){$e_n\!\otimes\! k_n$}
\put(68,-12){$[i_n]\,p(q_n)^{-j_n}$}

\put(-65,-40){$|i_n\rangle\otimes |j_n\!-\!1\rangle$}
\put(55,-40){$|i_n\!-\!1\rangle\otimes |j_n\rangle$}

\put(-33,-47){\vector(1,-1){25}}
\put(-17,-57){$S(z)$}

\put(76,-47){\vector(-1,-1){25}}
\put(39,-57){$S(z)$}

\put(0,-77){$|a_n\rangle\otimes |b_n\rangle$}
\end{picture}
\]
This leads to
\begin{equation}\label{B2}
\begin{split}
p^{-1}B^{\bf{a,b}}_{\bf{i,j}}(z)
&=\sum_{c_0,\ldots, c_n}
\frac{z^{c_0}}{(q)_{c_n}}X(c_0,\ldots, c_{n-1})\\
\times &\left(p^{-2}[j_n]
{\mathscr S}^{(\epsilon_n)\,a_n,b_n,c_{n-1}}_{
\phantom{(\epsilon_n)}\, i_n,j_n-1,c_n}
+(q_n)^{-j_n}[i_n](1-q^{c_n})
{\mathscr S}^{(\epsilon_n)\,a_n,b_n,c_{n-1}}_{
\phantom{(\epsilon_n)}\, i_n-1,j_n,c_n-1}\right),
\end{split}
\end{equation}
where $c_n$ has been shifted to $c_n-1$ in the second term.
$X(c_0,\ldots, c_{n-1})$ is the same as in (\ref{A2}). 
From (\ref{A2}), (\ref{B2}) and (\ref{ice0}), 
$A^{\bf{a,b}}_{\bf{i,j}}(z)=B^{\bf{a,b}}_{\bf{i,j}}(z)$ is reduced to 
the equality of the quantities in the parenthesis:
\begin{equation}\label{ior}
\begin{split}
&-q[a+1](1-q^k)
{\mathscr S}^{(\epsilon)\,a+1,b,c}_{
\phantom{(\epsilon)}\, i,j,k-1}
+\rho^{-a}[b+1]
{\mathscr S}^{(\epsilon)\,a,b+1,c}_{
\phantom{(\epsilon)}\, i,j,k}\\
&+q[j]
{\mathscr S}^{(\epsilon)\,a,b,c}_{
\phantom{(\epsilon)}\, i,j-1,k}
-\rho^{-j}[i](1-q^{k})
{\mathscr S}^{(\epsilon)\,a,b,c}_{
\phantom{(\epsilon)}\, i-1,j,k-1}=0,
\end{split}
\end{equation}
where we have set 
$(a_n,b_n,c_{n-1},i_n,j_n,c_n,\epsilon_n,q_n)
=(a,b,c,i,j,k,\epsilon,\rho)$.
Thus $\rho = (-1)^{\epsilon}q^{1-2\epsilon}$
by (\ref{qidef}).
For $\epsilon=1$, (\ref{ior}) reads
\begin{align*}
-q(1-q^k)\delta^a_0\,
{\mathscr L}^{1,b,c}_{i,j,k-1}
+(-q)^a\delta^b_0\,
{\mathscr L}^{a,1,c}_{i,j,k}
+q\delta^1_j\,
{\mathscr L}^{a,b,c}_{i,0,k}
-(-q)^j(1-q^{k})\delta^1_i\,
{\mathscr L}^{a,b,c}_{0,j,k-1}=0,
\end{align*}
which can directly be verified by using (\ref{Lex}).
For $\epsilon=0$,  (\ref{ior}) represents the relation 
obtained by replacing ${\mathscr S}^{(0)}$ by $\Rm$ and 
$\rho$ by $q$. 
All the elements of $\Rm$ can be converted to the the form 
$\Rm^{a,b,c} _{\bullet,\bullet,\bullet}$ 
by decreasing $a$ by (\ref{hkr1}) and $b$ by (\ref{hkr3}).
The resulting expression turns out to be identically zero.

(iii) Case $r=0$ for $S^{1,1}(z)$.
The action of $e_0$ and $k_0$ in (\ref{actsD}) only concern the first 
component $|m_1\rangle^{(\epsilon_1)}$ of $|{\bf m}\rangle$.
Denoting it simply by $|m_1\rangle$, we depict (\ref{Ah}) and (\ref{Bh}) as
\[
\begin{picture}(120,90)(70,-80)

\put(0,0){$|i_1\rangle\otimes |j_1\rangle$}

\put(-8,-3){\vector(-1,-1){25}}
\put(-16,-22){$S(z)$}

\put(51,-3){\vector(1,-1){25}}
\put(38,-22){$S(z)$}

\put(-65,-40){$|a_1\!-\!1\rangle\otimes |b_1\rangle$}
\put(55,-40){$|a_1\rangle\otimes |b_1\!-\!1\rangle$}

\put(-33,-47){\vector(1,-1){25}}
\put(-19,-57){${\tilde e}_0\!\otimes\! 1$}
\put(-39,-67){$xp$}

\put(76,-47){\vector(-1,-1){25}}
\put(29,-57){$k_0\!\otimes\! e_0$}
\put(68,-67){$yp^{-1}(q_1)^{a_1}$}

\put(0,-77){$|a_1\rangle\otimes |b_1\rangle$}

\put(220,0){
\put(0,0){$|i_1\rangle\otimes |j_1\rangle$}

\put(-37,-12){$yp$}
\put(-8,-3){\vector(-1,-1){25}}
\put(-16,-22){$1\!\otimes\! {\tilde e}_0$}

\put(51,-3){\vector(1,-1){25}}
\put(30,-22){$e_0\!\otimes\! k_0$}
\put(68,-12){$xp^{-1}(q_1)^{j_1}$}

\put(-65,-40){$|i_1\rangle\otimes |j_1\!+\!1\rangle$}
\put(55,-40){$|i_1\!+\!1\rangle\otimes |j_1\rangle$}

\put(-33,-47){\vector(1,-1){25}}
\put(-17,-57){$S(z)$}

\put(76,-47){\vector(-1,-1){25}}
\put(39,-57){$S(z)$}

\put(0,-77){$|a_1\rangle\otimes |b_1\rangle$}
}

\end{picture}
\]
where ${\tilde e}_0 = pe_0$ has been used.
From this and $z=x/y$ we have
\begin{align*}
p^{-1}A^{\bf{a,b}}_{\bf{i,j}}(z)
&=\sum_{c_0,\ldots, c_n}z^{c_0+1}(-q;q)_{c_0}Y(c_1,\ldots, c_n)
\left(
{\mathscr S}^{(\epsilon_1)\,a_1-1,b_1,c_0}_{
\phantom{(\epsilon_1)}\, i_1,j_1,c_1}
+p^{-2}(q_1)^{a_1}(1+q^{c_0+1})
{\mathscr S}^{(\epsilon_1)\,a_1,b_1-1,c_0+1}_{
\phantom{(\epsilon_1)}\, i_1,j_1,c_1}
\right),\\
p^{-1}B^{\bf{a,b}}_{\bf{i,j}}(z)
&=\sum_{c_0,\ldots, c_n}z^{c_0+1}(-q;q)_{c_0}Y(c_1,\ldots, c_n)
\left((1+q^{c_0+1})
{\mathscr S}^{(\epsilon_1)\,a_1,b_1,c_0+1}_{
\phantom{(\epsilon_1)}\, i_1,j_1+1,c_1}
+p^{-2}(q_1)^{j_1}
{\mathscr S}^{(\epsilon_1)\,a_1,b_1,c_0}_{
\phantom{(\epsilon_1)}\, i_1+1,j_1,c_1}
\right)
\end{align*}
with a common $Y(c_1,\ldots, c_n)$ independent of $z$.
We have shifted $c_0$ to $c_0+1$
in the second term of $p^{-1}A^{\bf{a,b}}_{\bf{i,j}}(z)$ and 
in the first term of $p^{-1}B^{\bf{a,b}}_{\bf{i,j}}(z)$.
Now $A^{\bf{a,b}}_{\bf{i,j}}(z)=B^{\bf{a,b}}_{\bf{i,j}}(z)$ is reduced to 
\begin{align}\label{ior2}
{\mathscr S}^{(\epsilon)\,a-1,b,c}_{
\phantom{(\epsilon)}\, i,j,k}
-q\rho^{a}(1+q^{c+1})
{\mathscr S}^{(\epsilon)\,a,b-1,c+1}_{
\phantom{(\epsilon)}\, i,j,k}
-(1+q^{c+1})
{\mathscr S}^{(\epsilon)\,a,b,c+1}_{
\phantom{(\epsilon)}\, i,j+1,k}
+q\rho^{j}
{\mathscr S}^{(\epsilon)\,a,b,c}_{
\phantom{(\epsilon)}\, i+1,j,k}=0
\end{align}
for $\rho = (-1)^\epsilon q^{1-2\epsilon}$.
For $\epsilon=1$, this can be verified directly from (\ref{Lex}).
For $\epsilon=0$, all the elements of $\Rm$ can be expressed in the 
form $\Rm^{\bullet, \bullet, \bullet}_{i,j,k}$ by 
using (\ref{hzk}) to decrease $j$ and (\ref{mak}) to decrease $i$.
The result turns out to be identically zero.
The proof of (\ref{ce}) is completed.
\end{proof}

\section{Proof Part II: 
Irreducibility of ${\mathcal W}_l \otimes {\mathcal W}_m$ for ${\mathcal U}_A$}
\label{sec:IrreA}

Here we consider 
${\mathcal U}_A$ of the form 
${\mathcal U}_A(1^\kappa,0^{n-\kappa})\, 
(0 \le \kappa \le n)$ and show 
that the ${\mathcal U}_A$-module $\W_l\otimes \W_m$ is irreducible.
See (\ref{wl}) for the definition of $\W_l\subset \W$.
We assume that $0 \le l,m \le n$ if $\kappa=n$ and 
$l,m \in \Z_{\ge 0}$ otherwise.
We will flexibly write 
$|a_1,a_2,\ldots, a_n\rangle \otimes |b_1,b_2,\ldots b_n \rangle \in 
\W\otimes \W$ as
\begin{align*}
&(|a_1\rangle \otimes |b_1\rangle) \bt (|a_2\rangle \otimes |b_2\rangle) \bt \cdots 
\bt (|a_n\rangle \otimes |b_n\rangle) \;\;\text{or}\\
&(|a_1,\ldots, a_j\rangle \otimes |b_1,\ldots, b_j\rangle)
\bt 
(|a_{j+1},\ldots, a_n\rangle \otimes |b_{j+1},\ldots, b_n\rangle)\;\;\text{for some $j$}
\end{align*}
and so on. 
The vectors $v_0=|0\rangle^{(1)},  v_1 =|1\rangle^{(1)}\in V$ 
(\ref{kby}) will simply be denoted by
$|0\rangle, |1\rangle$.
They are to be distinguished from 
$|0\rangle=|0\rangle^{(0)}, |1\rangle=|1\rangle^{(0)} \in F$ 
from  the context. (See the remark after (\ref{kby}).)
We treat the cases $\kappa=n,n-1$ and $1 \le \kappa \le n-2$ separately. 
We include the results on the spectral decompositions although the 
concrete forms of the eigenvalues in (\ref{spd1}), (\ref{spd2}) and 
(\ref{spd3}) are not necessary 
for our main issue, namely, the proof of the irreducibility.

\subsection{Case $\kappa=n$}\label{ss:kan-1}
As mentioned in 
(\ref{equiv}), the relevant algebra ${\mathcal U}_A(1,\ldots, 1)$ 
supplemented with the Serre relation is 
$U_{-q^{-1}}(A^{(1)}_{n-1})$.
The representation $\W_l$ in Proposition \ref{pr:relA} is the 
$(-q^{-1})$-analogue of the 
$l$-fold anti-symmetric tensor representation.
Thus we assume $0 \le l,m \le n$.
It is known that ${\mathcal W}_l \otimes {\mathcal W}_m$ is 
an irreducible $U_{-q^{-1}}(A^{(1)}_{n-1})$-module, and 
the quantum $R$ matrix is given for example in \cite{DO}.
We recall it as a preparation for the next case $\kappa=n-1$.
Note that Proposition \ref{pr:relA} 
with $\W= V^{\otimes n}$ and $\forall q_i = -q^{-1}$ gives
\begin{equation}\label{eacts}
\begin{split}
&(\pi_x \otimes \pi_y)\Delta(e_i)
(|\ldots, m_i,m_{i+1},\ldots\rangle \otimes |\ldots, m'_i,m'_{i+1},\ldots\rangle)\\
&\quad= y^{\delta_{i,0}}\delta_{m'_i}^1\delta_{m'_{i+1}}^0
|\ldots, m_i,m_{i+1},\ldots\rangle \otimes |\ldots, 0,1,\ldots\rangle\\
&\quad+ x^{\delta_{i,0}}\delta_{m_i}^1\delta_{m_{i+1}}^0(-q)^{m'_i-m'_{i+1}}
|\ldots, 0,1,\ldots\rangle \otimes |\ldots, m'_i,m'_{i+1},\ldots\rangle
\quad (i \in \Z_n).
\end{split}
\end{equation}

\subsubsection{Singular vectors}
For $r \ge 1$ we define
\begin{equation}\label{Jvec}
\begin{split}
&{\mathcal J}_{r,j} = \sum_{
\begin{subarray}{c} (i_1,\ldots, i_r) \in \{0,1\}^r \\
i_1+\cdots+i_r=j
\end{subarray}}q^{\mathrm{inv}(i_1,\ldots, i_r)}
|i_1,\ldots, i_r\rangle \otimes |{\bar i}_1,\ldots, {\bar i}_r\rangle \in V^{\otimes r} \otimes V^{\otimes r}
\quad (0 \le j \le r),\\
&\mathrm{inv}(i_1,\ldots, i_r) = \sum_{1 \le s<t\le r}i_s{\bar i}_t,\qquad {\bar i} = 1-i.
\end{split}
\end{equation}
It is characterized by the recursion relations
\begin{align}
{\mathcal J}_{r,j} &= {\mathcal J}_{r-1,j-1} \bt (|1\rangle \otimes |0\rangle)
+q^j {\mathcal J}_{r-1,j}
 \bt (|0\rangle \otimes |1\rangle)\label{Jrec1}\\
&= (|0\rangle \otimes |1\rangle)\bt 
{\mathcal J}_{r-1,j} +q^{r-j} (|1\rangle \otimes |0\rangle)\bt
{\mathcal J}_{r-1,j-1}\label{Jrec2}
\end{align}
with the initial condition 
${\mathcal J}_{1,0} = |0\rangle \otimes |1\rangle$ and 
${\mathcal J}_{1,1} = |1\rangle \otimes |0\rangle$,
where ${\mathcal J}_{r,j}$ with $j \not\in [0,r]$ is to be understood as 0.
For example, the $r=2$ case reads
\begin{align*}
{\mathcal J}_{2,0} = |0,0\rangle \otimes |1,1\rangle,\quad
{\mathcal J}_{2,1} = |0,1\rangle \otimes |1,0\rangle+q|1,0\rangle \otimes |0,1\rangle,\quad
{\mathcal J}_{2,2} = |1,1\rangle \otimes |0,0\rangle.
\end{align*}
We also understand that ${\mathcal J}_{0,0}$ is the object 
that formally makes the above recursion relations valid for $r=1$.
Note that ${\mathcal J}_{n,j} \in \W_j \otimes \W_{n-j}$.

\begin{lemma}\label{le:eJ}
$(\pi_x \otimes \pi_y)\Delta(e_i){\mathcal J}_{n,j} = 0$ holds for $1 \le i \le n-1$ and  $0 \le j \le n$.
\end{lemma}
\begin{proof}
By using (\ref{eacts}) and the above example, 
the case $n=2$ can be directly checked.
Then the assertion follows by induction on $n$ thanks to (\ref{Jrec1}) and (\ref{Jrec2}).
\end{proof}

For $0 \le l,m \le n$ define the following vector in 
$\W_l \otimes \W_m$:
\begin{align}
\xi_t = \xi^{l,m}_t = (|{\bf 0}_{n-s-t}\rangle \otimes |{\bf 0}_{n-s-t}\rangle) \bt
{\mathcal J}_{s,l-t}\bt (|{\bf 1}_{t}\rangle \otimes |{\bf 1}_{t}\rangle)
\quad
(s=l+m-2t)
\end{align}
for $(l+m-n)_+ \le t \le \min(l,m)$, where the symbol $(x)_+$ is defined 
after (\ref{norm1}).
We have set 
$|{\bf i}_t \rangle = 
|i\rangle^{\otimes t} \in V^{\otimes t}$
for $i=0,1$.
Note that $\xi^{l,m}_{\min(l,m)} 
= |{\bf e}_{>n-l}\rangle \otimes |{\bf e}_{>n-m}\rangle$.
Using Lemma \ref{le:eJ} one can show
\begin{proposition}\label{pr:sin1}
The weight vectors in $\W_l \otimes \W_m$ annihilated by 
$(\pi_x \otimes \pi_y)\Delta(e_i)$ for $1 \le i \le n-1$ are given by $\xi_t$ with 
$(l+m-n)_+ \le t \le \min(l,m)$ up to an overall scalar.
\end{proposition}

\subsubsection{Spectral decomposition}
A direct calculation shows
\begin{lemma}\label{le:srec1}
For $(l+m-n)_+ \le t < \min(l,m)$ the following relations hold:
\begin{align*}
(\pi_x \otimes \pi_y)\Delta(e_{n-l-m+t+1}\cdots e_{n-1} e_{n-t-1}\cdots e_1e_0)\xi_t &=
(-[2])^{\delta_{t,0}}q^{-1}(q^{l+m-2t}y-x)\xi_{t+1},\\
(\pi_x \otimes \pi_y)\Delta(f_0f_1\cdots f_{n-l-m+t} f_{n-1}f_{n-2}\cdots f_{n-t})\xi_t &=
(qx y)^{-1}(q^{l+m-2t}y-x)\xi_{t+1}.
\end{align*}
\end{lemma}

Set $z=x/y$ and 
let $R(z) \in \mathrm{End}(\W_l \otimes \W_m)$ be 
the quantum $R$ matrix satisfying (\ref{eqrc})
normalized as (\ref{rnor1}).
Due to Proposition \ref{pr:sin1} it has the spectral decomposition 
\begin{align}\label{spdec1}
PR(z) = \sum_{s=(l+m-n)_+}^{\min(l,m)} \rho_s(z){\mathscr P}^{l,m}_s,
\end{align}
where $P$ is defined after (\ref{eqrc})
and ${\mathscr P}^{l,m}_s: \W_l \otimes \W_m \rightarrow \W_m \otimes \W_l$ is 
the projector characterized by
\begin{align}
&{\mathscr P}^{l,m}_s : \xi^{l,m}_{s'} \mapsto \delta_{s,s'}\xi^{m,l}_s,
\label{pxi}\\
&(\pi_y \otimes \pi_x)\Delta(g){\mathscr P}^{l,m}_s 
=  {\mathscr P}^{l,m}_s (\pi_x \otimes \pi_y)\Delta(g)
\label{pcom}
\end{align} 
for all $g \in \overline{\mathcal U}_A(1,\ldots, 1)$.
See the end of Section \ref{subsec:U} for the definition of 
$\overline{\mathcal U}_A(1,\ldots, 1)$.
The combination $PR(z)$ is the intertwiner of 
$\pi_x\otimes \pi_y$ and $\pi_y\otimes \pi_x$ 
denoted by $\check{R}(z)$ in \cite{Ji}.
Substituting (\ref{spdec1}) into either (\ref{eR}) or (\ref{fR}) with $r=0$
one gets
\begin{align}\label{recr}
\frac{\rho_{s+1}(z)}{\rho_s(z)} = \frac{1-q^{l+m-2s}z}{z-q^{l+m-2s}}.
\end{align}
From (\ref{rnor1}) it follows that  
$\rho_{\min(l,m)}(z) = 1$ and 
\begin{align}\label{spd1}
PR(z) = \sum_{s=(l+m-n)_+}^{\min(l,m)}
\left(\prod_{i=s+1}^{\min(l,m)}
\frac{z-q^{l+m-2i+2}}{1-q^{l+m-2i+2}z}\right)
{\mathscr P}^{l,m}_s.
\end{align}

\subsection{Case $\kappa=n-1$}\label{ss:kan-2}

Consider 
${\mathcal U}_A$ of the form 
${\mathcal U}_A(1,\ldots,1,0)$.
We show 
that the ${\mathcal U}_A$-module 
$\W_l \otimes \W_m$ with 
$\W = V^{\otimes n-1} \otimes F$ is irreducible and 
present the spectral decomposition of the associated quantum $R$ matrix.
We assume $l,m \in \Z_{\ge 0}$.

\subsubsection{Singular vectors}
For $0 \le s \le \min(n-1,l,m)$
define the following vector in 
$\W_l \otimes \W_m$:
\begin{align}
\xi_s = \xi^{l,m}_s=
(|{\bf 0}_{n-s-1}\rangle \otimes |{\bf 0}_{n-s-1}\rangle) \bt
\sum_{j=0}^s (-1)^jq^{j(m-s+1)}
{\mathcal J}_{s,s-j}\bt
(|l+j-s\rangle \otimes |m-j\rangle).
\end{align}
The ${\mathcal J}_{s,s-j}$ is defined by (\ref{Jvec}).
Note that
$\xi_0 = |l{\bf e}_n \rangle \otimes |m{\bf e}_n\rangle$.

\begin{proposition}\label{pr:sin2}
The weight vectors in $\W_l \otimes \W_m$ annihilated by 
$(\pi_x \otimes \pi_y)\Delta(e_i)$ for $1 \le i \le n-1$ are given by $\xi_s$ with 
$0 \le s \le \min(n-1,l,m)$ up to an overall scalar.
\end{proposition}

\subsubsection{Spectral decomposition}
\begin{lemma}\label{le:srec2}
For $1 \le s \le \min(n-1,l,m)$ the following relations hold:
\begin{align*}
&(\pi_x \otimes \pi_y)\Delta(
e_{n-1}\cdots e_1 e_{n-s}\cdots e_{n-1} e_0) \xi_s = 
(y-q^{l+m-2s+2} x)(-[2])^{\delta_{s,n-1}}\xi_{s-1},\\
&(\pi_x \otimes \pi_y)\Delta(f_0f_1\cdots f_{n-s-1}) \xi_s =
(xy)^{-1}(y-q^{l+m-2s+2} x)\xi_{s-1}.
\end{align*}
\end{lemma}

Set $z=x/y$ and 
let $R(z) \in \mathrm{End}(\W_l \otimes \W_m)$ 
be the quantum $R$ matrix satisfying (\ref{eqrc})
normalized as (\ref{rnor2}).
Due to Proposition \ref{pr:sin2} it has the spectral decomposition
\begin{align}\label{spdec2}
PR(z) = \sum_{s=0}^{\min(n-1,l,m)} \rho_s(z){\mathscr P}^{l,m}_s,
\end{align}
where the projector 
${\mathscr P}^{l,m}_s: \W_l \otimes \W_m \rightarrow \W_m \otimes \W_l$ is 
characterized by (\ref{pxi}) and 
(\ref{pcom}) for all $g \in \overline{\mathcal U}_A(1,\ldots, 1,0)$.
Substituting (\ref{spdec2}) into either (\ref{eR}) or (\ref{fR}) with $r=0$
one gets formally the same relation as (\ref{recr}).
In (\ref{rnor2}) (and also (\ref{norm2})), 
the $i$ should be taken as $n$, which leads to
$PR(z) (|l{\bf e}_n\rangle \otimes |m{\bf e}_n\rangle)
=|m{\bf e}_n\rangle \otimes |l{\bf e}_n\rangle$.
This implies $\rho_0(z)=1$ and 
\begin{align}\label{spd2}
PR(z) = \sum_{s=0}^{\min(n-1,l,m)}
\left(\prod_{i=1}^s
\frac{1-q^{l+m-2i+2}z}{z-q^{l+m-2i+2}}\right)
{\mathscr P}^{l,m}_{s}.
\end{align}
This formally coincides with (\ref{spd1}) up to an overall factor and the 
range of $s$.

\subsubsection{Irreducibility of 
${\mathcal W}_l \otimes {\mathcal W}_m$}\label{ss:akn-1}

Consider the direct sum decomposition
\begin{align*}
{\mathcal W}_l \otimes {\mathcal W}_m
&= \bigoplus_{0\le j\le s_l, 0 \le k \le s_m} X_{j,k}
\quad (s_k=\min(k,n-1)),\\
X_{j,k} &= \bigoplus \,\C(q)|i_1,\ldots, i_{n-1},l-j\rangle
\otimes  |i'_1,\ldots, i'_{n-1},m-k\rangle,
\end{align*}
where the latter direct 
sum is over $i_1,\ldots, i_{n-1}, i'_1, \ldots, i'_{n-1} \in \{0,1\}$
such that $(i_1+\cdots+ i_{n-1}, i'_1 + \cdots  + i'_{n-1})=(j,k)$.
The subalgebra of ${\mathcal U}_A(1,\ldots, 1,0)$
generated by $e_i,f_i,k^{\pm 1}_i$ with $1 \le i \le n-2$
(with the Serre relations)
is isomorphic to $U_{-q^{-1}}(A_{n-2})$. 
Let the same symbol denote its coproduct action.
Then we have
\begin{lemma}\label{le:lh}
\begin{align*}
X_{j,k} &= U_{-q^{-1}}(A_{n-2})u_{j,k},\\
u_{j,k} &:=|\overbrace{1,\ldots, 1}^j,0,\ldots,0,l-j\rangle \otimes 
|0,\ldots, 0,\overbrace{1,\ldots,1}^k,m-k\rangle.
\end{align*}
\end{lemma}
\begin{proof}
As an element of a $U_{-q^{-1}}(A_{n-2})$-module,
$u_{j,k}$ is 
$(\text{lowest wt.~vec.})\otimes (\text{highest wt.~vec.})$
in the tensor product of the
antisymmetric tensor representations of order $j$ and $k$.
Thus the assertion follows from Proposition \ref{pr:ht*lt}.
\end{proof}

We use the notation 
${\bf e}_{[g,h]} = {\bf e}_g+{\bf e}_{g+1}+\cdots + {\bf e}_h$.
The notation 
$|{\bf e}_{i_1}+\cdots + {\bf e}_{i_a}+\bullet\rangle 
\otimes 
|{\bf e}_{i'_1}+\cdots + {\bf e}_{i'_b}+\bullet\rangle$
will be used only when $\max(i_1, \ldots, i_a, i'_1, \ldots, i'_b)<n$ and 
is to be understood as 
$|{\bf e}_{i_1}+\cdots + {\bf e}_{i_a}+(l-a){\bf e}_n\rangle 
\otimes 
|{\bf e}_{i'_1}+\cdots + {\bf e}_{i'_b}+(m-b){\bf e}_n\rangle$.
Thus $u_{j,k} = |{\bf e}_{[1,j]}+\bullet\rangle \otimes
|{\bf e}_{[n-k,n-1]}+\bullet\rangle$.
\begin{proposition}\label{pr:irrAn-1}
The ${\mathcal U}_A(1,\ldots, 1,0)$-module 
${\mathcal W}_l \otimes {\mathcal W}_m$ is irreducible.
\end{proposition}

\begin{proof}
Let $W$ be a nonzero submodule of 
${\mathcal W}_l \otimes {\mathcal W}_m$.
Due to Lemma \ref{le:lh} it suffices to show that 
all the $u_{j,k}$ are generated from a vector in $W$.
We show this by induction on $j+k\ge 0$.
By Proposition \ref{pr:sin2} and Lemma \ref{le:srec2},
we can generate $u_{0,0}=\xi_0$ by applying $e_i$'s and $f_i$'s 
appropriately to any nonzero vector in $W$. Thus $j+k=0$ case is true.
Set $X_s = \oplus_{j+k=s}X_{j,k}$.
Let us show that all the $u_{j,k}$ with $j+k=s$ are generated 
by assuming that $X_{s-1}$ has already been generated.

(i) Case $s \le n-1$.
{\em Step 1}.  We show that $X_{s,0}$ is generated.
Set
\begin{align*}
\zeta_s = u_{s,0},\;\;
\zeta_j = |{\bf e}_{[1,j]}+{\bf e}_{[j+2,s]}+\bullet\rangle\otimes
|{\bf e}_{j+1}+\bullet\rangle \quad (0 \le j \le s-1).
\end{align*} 
They are vectors in $X_s$. 
We have ($(\pi_x\otimes \pi_y)\Delta(e_i)$ simply denoted by $e_i$ and 
similarly for $f_i$)
\begin{align*}
&c. f_{j+1}f_{j+2}\cdots f_{n-1} (
|{\bf e}_{[1,j]}+{\bf e}_{[j+2,s]}+\bullet\rangle \otimes |m {\bf e}_n\rangle)
= \zeta_{j+1} - q \zeta_{j}\quad (0 \le j \le s-2),\\
&c. f_{s}f_{s+1}\cdots f_{n-1} u_{s-1,0}
= [l-s+1]\zeta_{s} +q^{s-l-1}[m] \zeta_{s-1},\\
&c. e_{0} (
|{\bf e}_{[2,s]}+\bullet\rangle \otimes |m {\bf e}_n\rangle)
=xq^{-m}[l-s+1]\zeta_s+y[m]\zeta_0. 
\end{align*}
where $c.$ means multiplication by a nonzero rational function of $q$ 
which does not involve $x$ and $y$.
Regarding them as the $s\!+\!1$ linear equations on 
$\zeta_0,\ldots, \zeta_s$, one finds that the coefficient matrix is invertible
for generic $x$ and $y$.
Moreover all the lhs belong to $X_{s-1}$.
Thus $u_{s,0}=\zeta_s$ is generated.
Then by Lemma \ref{le:lh},  $X_{s,0}$ is generated.

{\em Step 2}. Set $s=j+k$. We show that 
$X_{j-1,k+1}$ is generated assuming that $X_{j,k}$ 
(and $X_{s-1}$) are already generated.
This claim follows from $(1 \le j \le s)$
\begin{align*}
&f_{j+k}\cdots f_{n-2}f_{n-1}(
|{\bf e}_{[1,j-1]}+\bullet\rangle 
\otimes |{\bf e}_{[j,j+k-1]}+\bullet\rangle)\\
&= q^{j-l-1}[m-k]|{\bf e}_{[1,j-1]}+\bullet\rangle 
\otimes |{\bf e}_{[j,j+k]}+\bullet\rangle
+[l-j+1]|{\bf e}_{[1,j-1]}+{\bf e}_{j+k}+\bullet\rangle 
\otimes |{\bf e}_{[j,j+k-1]}+\bullet\rangle.
\end{align*} 
The lhs belongs to $X_{j-1,k}\subseteq X_{s-1}$ and 
the second term on the rhs does to $X_{j,k}$.
Therefore 
$|{\bf e}_{[1,j-1]}+\bullet\rangle \otimes |{\bf e}_{[j,j+k]}+\bullet\rangle$ is 
generated. 
Applying $e_{n-2}e_{n-3}\cdots e_{j+k}$ to it we get $u_{j-1,k+1}$.
Then $X_{j-1,k+1}$ is generated by Lemma \ref{le:lh}.

By Step 1 and applying Step 2 repeatedly in the order $j=s,s-1,\ldots, 1$, 
we get $X_s$. 

(ii) Case $s \ge n$.
By the induction we assume that $u_{j',k'} \in X_{s-1}$ is already generated.
From $j'+k'=s-1\ge n-1$, we have
\begin{align*}
c. f_{j'+1}\cdots f_{n-2}f_{n-1} u_{j',k'} = u_{j'+1,k'},\quad
c. e_{k-1}\cdots e_1 e_0 u_{j',k'} = y u_{j',k'+1}.
\end{align*}
Thus $X_s$ is generated.
By (i) and (ii) the induction step has been proved.
\end{proof}

\subsection{Case $0 \le \kappa\le n-2$}\label{ss:kan-3}

Consider ${\mathcal U}_A$ of the form 
${\mathcal U}_A(1^\kappa,0^{n-\kappa})$ with $0 \le \kappa \le n-2$.
We show that the ${\mathcal U}_A$-module 
$\W_l \otimes \W_m$ with 
$\W = V^{\otimes \kappa} \otimes F^{\otimes n-\kappa}$ 
is irreducible and 
present the spectral decomposition of the associated quantum $R$ matrix.
We assume $l,m \in \Z_{\ge 0}$ and $n \ge 2$.
Due to $\kappa \le n-2$, the rightmost two components in 
$\W$ is $F \otimes F$.

\subsubsection{Singular vectors}
For $0 \le s \le \min(l,m)$ introduce the 
following vector in $\W_l \otimes \W_m$:
\begin{align}\label{xi4}
\xi_s = \xi^{l,m}_s=
(|{\bf 0}_{n-2}\rangle \otimes |{\bf 0}_{n-2}\rangle) \bt
\sum_{j=0}^s (-1)^jq^{j(2s-m-j-1)+ms}{s \brack j}
(|j,l-j\rangle \otimes |s-j,m-s+j\rangle),
\end{align}
where 
$|{\bf 0}_{n-2}\rangle = |0\rangle \otimes \cdots
\otimes |0\rangle 
\in V^{\otimes \kappa}\otimes F^{\otimes n-\kappa-2}$.
Note that $\xi_0 = |l{\bf e}_n\rangle \otimes |m{\bf e}_n\rangle$.

\begin{proposition}\label{pr:sin3}
The weight vectors in $\W_l \otimes \W_m$ annihilated by 
$(\pi_x \otimes \pi_y)\Delta(e_i)$ for $1 \le i \le n-1$ 
are given by $\xi_s$ with 
$0 \le s \le \min(l,m)$ up to an overall scalar.
\end{proposition}

\subsubsection{Spectral decomposition}
\begin{lemma}\label{le:srec3}
For $0 \le s \le \min(l,m)$ the following relations hold:
\begin{align*}
&(\pi_x \otimes \pi_y)\Delta(
e_{n-2}\cdots e_0) \xi_s 
= A\xi_{s+1}+ B f^2_{n-1}\xi_{s-1}+ C f_{n-1}\xi_s,\\
&\quad A = -\frac{q^{-l-s}[l-s][l+m+1-s][m-s](x-q^{l+m-2s}y)}
{[l+m+1-2s][l+m-2s]},\\
&\quad  B = \frac{q^{-2-m+3s}[s](q^{l+m+2-2s}x-y)}
{[l+m+1-2s][l+m+2-2s]},\\
&\quad C= \frac{([l-s][l+m+2-s]-[m-s][s])x+(l\leftrightarrow m, x \leftrightarrow y)}
{[l+m-2s][l+m+2-2s]},\\
&(\pi_x \otimes \pi_y)\Delta(f_0\cdots f_{n-2}) \xi_s 
= (xy)^{-1}(q^{l+m}x-q^{2s-2}y)[s]\xi_{s-1}.
\end{align*}
\end{lemma}

Set $z=x/y$ and 
let $R(z) \in \mathrm{End}(\W_l \otimes \W_m)$ 
be the quantum $R$ matrix satisfying (\ref{eqrc})
normalized as (\ref{rnor2}).
Due to Proposition \ref{pr:sin3}  it has the spectral decomposition
\begin{align}\label{spdec3}
PR(z) = \sum_{s=0}^{\min(l,m)} \rho_s(z){\mathscr P}^{l,m}_s,
\end{align}
where the projector 
${\mathscr P}^{l,m}_s: \W_l \otimes \W_m \rightarrow \W_m \otimes \W_l$ is 
characterized by (\ref{pxi}) and 
(\ref{pcom}) for all 
$g \in \overline{\mathcal U}_A(\epsilon_1,\ldots, \epsilon_{n-2},0,0)$.
From Lemma \ref{le:srec3} one gets formally the same relation as (\ref{recr}).
The normalization condition (\ref{rnor2}) tells that
$PR(z) (|l{\bf e}_n\rangle \otimes |m{\bf e}_n\rangle)
=|m{\bf e}_n\rangle \otimes |l{\bf e}_n\rangle$.
Thus we have $\rho_0(z)=1$ and 
\begin{align}\label{spd3}
PR(z) = \sum_{s=0}^{\min(l,m)}
\left(\prod_{i=1}^s
\frac{1-q^{l+m-2i+2}z}{z-q^{l+m-2i+2}}\right)
{\mathscr P}^{l,m}_{s},
\end{align}
which is formally identical with (\ref{spd2}) except for the range of $s$.

\subsubsection{Irreducibility of 
${\mathcal W}_l \otimes {\mathcal W}_m$}\label{ss:akn-2}

Consider the direct sum decomposition
\begin{align*}
{\mathcal W}_l \otimes {\mathcal W}_m
&= \bigoplus_{0\le j\le t_l, 0 \le k \le t_m} Y_{j,k}
\quad (t_k=\min(\kappa,k)),\\
Y_{j,k} &= \bigoplus \,\C(q)|i_1,\ldots, i_n\rangle
\otimes  |i'_1,\ldots, i'_{n}\rangle,
\end{align*}
where the latter direct sum is over 
$i_1,\ldots, i_{\kappa}, i'_1, \ldots, i'_{\kappa} \in \{0,1\}$
and $i_{\kappa+1},\ldots, i_n, i'_{\kappa+1},\ldots, i'_n \in \Z_{\ge 0}$ 
such that $(i_1+\cdots+ i_{\kappa}, i'_1 + \cdots  + i'_{\kappa})=(j,k)$
and $(i_{\kappa+1}+\cdots + i_n, i'_{\kappa+1}+\cdots + i'_n)=(l-j,m-k)$. 
The subalgebra of ${\mathcal U}_A(1^\kappa,0^{n-\kappa})$
generated by $e_i,f_i,k^{\pm 1}_i$ with $1 \le i \le \kappa-1$ 
(resp. $\kappa+1\le i \le n-1$)  with the Serre relations
is isomorphic to $U_{-q^{-1}}(A_{\kappa-1})$ (resp. $U_q(A_{n-\kappa-1})$).
Let the same symbols denote their coproduct action.
Then $U_{-q^{-1}}(A_{\kappa-1})$ and $U_q(A_{n-\kappa-1})$ are commuting and we have
\begin{lemma}\label{le:lh2}
\begin{align*}
Y_{j,k} &= U_{-q^{-1}}(A_{\kappa-1})U_q(A_{n-\kappa-1})v_{j,k},\\
v_{j,k} &:=|\overbrace{1,\ldots, 1}^j,\overbrace{0,\ldots,0}^{n-j-1},l-j\rangle \otimes 
|\overbrace{0,\ldots,0}^{\kappa-k},
\overbrace{1,\ldots,1}^k,m-k,\overbrace{0,\ldots, 0}^{n-\kappa-1}\rangle.
\end{align*}
\end{lemma}
\begin{proof}
As an element of a $U_{-q^{-1}}(A_{\kappa-1})$-module,
$v_{j,k}$ is the $(\text{lowest wt. vec.})\otimes (\text{highest wt. vec.})$ in the 
tensor product of the
antisymmetric tensor representations of order $j$ and $k$.
As an element of a $U_q(A_{n-\kappa-1})$-module,
$v_{j,k}$ is the $(\text{highest wt. vec.})\otimes (\text{lowest wt. vec.})$ in the 
tensor product of the
symmetric tensor representations of order $l-j$ and $m-k$.
Thus the assertion follows from Proposition \ref{pr:ht*lt}.
\end{proof}

\begin{proposition}\label{pr:irrAn-2}
The ${\mathcal U}_A(1^\kappa, 0^{n-\kappa})$-module 
${\mathcal W}_l \otimes {\mathcal W}_m$ is irreducible.
\end{proposition}
\begin{proof}
Let $W$ be a nonzero submodule of ${\mathcal W}_l \otimes {\mathcal W}_m$.
Due to Lemma \ref{le:lh2} it suffices to show that 
all the $v_{j,k}$ are generated from a vector in $W$.
By Proposition \ref{pr:sin3} and Lemma \ref{le:srec3},
we can generate all the $\xi_s$ (\ref{xi4}).
By applying $U_q(A_{n-\kappa-1})$ to them further we can generate
$Y_{0,0}$. 
It contains the vector 
$|l{\bf e}_n\rangle \otimes |m{\bf e}_{\kappa+1}\rangle$.
Then $v_{j,k}$ is generated as 
$v_{j,k} = c. F_{\kappa}F_{\kappa-1}\cdots F_{\kappa-k+1}
E_0E_1\cdots E_{j-1}(|l{\bf e}_n\rangle \otimes |m{\bf e}_{\kappa+1}\rangle)$,
where
$E_i=x^{-1}e_ie_{i-1}\cdots e_0$ and 
$F_i = f_i f_{i+1}\cdots f_\kappa$.
\end{proof}

\section{Proof Part III: 
Irreducibility of ${\mathcal W} \otimes {\mathcal W}$ for ${\mathcal U}_B$}
\label{sec:IrreD}

Consider ${\mathcal U}_B = 
{\mathcal U}_B(1^{\kappa},0^{\kappa'})\,
(\kappa'=n-\kappa)$. In this section we show the 
irreducibility of the ${\mathcal U}_B$-module $\W\otimes \W$ and 
present the spectral decomposition of the associated quantum $R$ matrix.
We assume 
$\kappa,\kappa'\ge1$, since the $\kappa'=0$ case was treated in \cite{KS} and the 
$\kappa=0$ case in \cite{KO3,KOproc}. We follow the convention for the vector 
in $\W\otimes\W$ in the beginning of Section \ref{sec:IrreA}.
For a subset $J$ of $\{0,1,\ldots,n\}$ define the subalgebra ${\mathcal U}_{B,J}$
by the one generated by $e_i,f_i,k^{\pm 1}_i$ for $i\in J$.

\subsection{Singular vectors and spectral decomposition}

Although our algebra ${\mathcal U}_B$ and module $\W$ are different from
$U_q(D_{n+1}^{(2)})$ and $F^{\otimes n}$ treated in \cite{KO3}, 
the action of generators in Proposition \ref{pr:relD} is quite similar to 
\cite[Prop.~1]{KO3},
and consequently, the following propositions remain to be valid. 

\begin{proposition}
The weight vectors in $\W \otimes \W$ annihilated by 
$(\pi_x \otimes \pi_y)\Delta(e_i)$ for $1 \le i \le n$ are given by
\[
\xi_l=\sum_{m=0}^l (-p)^{-m} q^{m(l-(m+1)/2)}{l\brack m}
|m{\bf e}_n\rangle\ot|(l-m){\bf e}_n\rangle
\]
for some $l\in\Z_{\ge0}$ up to an overall scalar.
\end{proposition}

\begin{proposition} \label{pr:recD}
We have 
\begin{align*}
&(\pi_x\ot\pi_y)\Delta(e_{n-1}\cdots e_1e_0)\xi_l
=\frac1{1-q^{2l+1}}\{(q^{l+1}x+y)\xi_{l+1}+q^l(x+q^ly)f_n^2\xi_{l-1}\}\quad(l\ge1),\\
&(\pi_x\ot\pi_y)\Delta(e_{n-1}\cdots e_1e_0)\xi_0
=\frac1{1-q}\{(qx+y)\xi_1-\mathrm{i}q^{1/2}(x+y)f_n\xi_0\},\\
&(\pi_x\ot\pi_y)\Delta(f_0f_1\cdots f_{n-1})\xi_l
={\mathrm i}[l]q^{-1/2}(q^lx^{-1}+y^{-1})\xi_{l-1}\quad(l\ge1).
\end{align*}
\end{proposition}

\begin{proposition}\label{spd4}
For ${\mathcal U}_B$, $PR(z)$ has the following spectral decomposition.
\[
PR(z)=\sum_{l=0}^\infty\prod_{j=1}^l\frac{z+q^j}{1+q^jz}
{\mathscr P}_l,
\]
where ${\mathscr P}_l$ 
is the projector on the space generated from $\xi_l$ over
${\mathcal U}_{B,\{1,\ldots,n\}}$.
\end{proposition}

\subsection{Irreducibility of $\W \otimes \W$}

We prove the irreducibility of the ${\mathcal U}_B$-module $\W\ot\W$. 
Set $K=\{1,\ldots,\kappa\},K'=\{\kappa+1,\ldots,n\}$. The subalgebra
${\mathcal U}_{B,K-1}\,(K-1:=\{0,\ldots,\kappa-1\})$ 
(resp. ${\mathcal U}_{B,K'}$) is isomorphic to $U_{-q^{-1}}(B_\kappa)$
(resp. $U_q(B_{\kappa'})$). 
We use the same symbol $|{\bf i}_t\rangle
=|i\rangle^{\ot t}\in V^{\ot t}(t\le\kappa),V^{\ot\kappa}\ot F^{\ot(t-\kappa)}
(t>\kappa)$ as in Section \ref{sec:IrreA}. 
For instance one can write
\[
\xi_l=(|{\bf 0}_{n-1}\rangle\ot|{\bf 0}_{n-1}\rangle)\boxtimes\overline{\xi}_l,\quad
\overline{\xi}_l=\sum_{m=0}^l(-p)^{-m}q^{m(l-(m+1)/2)}{l\brack m}|m\rangle\ot|l-m\rangle.
\]
We also use a notation 
$\xi_l^{(t)}=(|{\bf 0}_{t-1}\rangle\ot|{\bf 0}_{t-1}\rangle)\boxtimes\overline{\xi}_l$.

\begin{lemma} \label{lem:1}
Let $a\le\kappa,{\bf a}=\sum_{i=1}^{a-1}\alpha_i{\bf e}_i,
{\bf a}'=\sum_{i=1}^{a-1}\alpha'_i{\bf e}_i$. 
For ${\bf a}$, set $\W_{{\bf a},K'}=
\langle|{\bf a}+\sum_{i=\kappa+1}^n\gamma_i{\bf e}_i\rangle
\mid\gamma_i\in\Z_{\ge0}\rangle$, where $\langle Y \rangle$ 
means the linear span of the set $Y$ of the vectors.
Then we have 
\[
\W_{{\bf a},K'}\ot\W_{{\bf a}',K'}=\sum_{l=0}^\infty{\mathcal U}_{B,K'}
(|\alpha_1,\ldots,\alpha_{a-1}\rangle\ot|\alpha'_1,\ldots,\alpha'_{a-1}\rangle)
\boxtimes \xi^{(n-a+1)}_l.
\]
\end{lemma}

\begin{proof}
Follow carefully the proof of \cite[Prop.~7]{KOproc}. The proof also works in the 
present setting.
\end{proof}

A direct calculation shows

\begin{lemma} \label{lem:2}
For $a\le\kappa,{\bf a}=\sum_{i=1}^{a-1}\alpha_i{\bf e}_i,
{\bf a}'=\sum_{i=1}^{a-1}\alpha'_i{\bf e}_i$ 
we have
\begin{align*}
&(\pi_x\ot\pi_y)\Delta(f_a\cdots f_{n-1})
(|\alpha_1,\ldots,\alpha_{a-1}\rangle\ot|\alpha'_1,\ldots,\alpha'_{a-1}\rangle)
\boxtimes \xi^{(n-a+1)}_l \\
&\quad=[l](|\alpha_1,\ldots,\alpha_{a-1}\rangle\ot|\alpha'_1,\ldots,\alpha'_{a-1}\rangle)
\boxtimes(|0\rangle\ot|1\rangle+\mathrm{i}q^{l-1/2}
|1\rangle\ot|0\rangle)\boxtimes \xi^{(n-a)}_{l-1}.
\end{align*}
\end{lemma}

\begin{lemma} \label{lem:4}
For any $l\in\Z_{\ge0}$, $(V^{\ot\kappa}\ot V^{\ot\kappa})
\boxtimes \xi^{(n-\kappa)}_l$ is generated from 
$\{(|{\bf 0}_\kappa\rangle\ot|{\bf 0}_\kappa\rangle)\boxtimes \xi^{(n-\kappa)}_m
\mid m\in\Z_{\ge0}\}$ over ${\mathcal U}_B$.
\end{lemma}

\begin{proof}
We prove $(V^{\ot k})^{\ot2}\boxtimes \xi_l^{(n-k)}$ is generated from
$\{(|{\bf 0}_k\rangle\ot|{\bf 0}_k\rangle)\boxtimes \xi^{(n-k)}_m
\mid m\in\Z_{\ge0}\}$ by induction on $k$ ($0\le k\le\kappa$). 
When $k=0$, there is nothing to prove. Suppose the statement
is valid with $k-1$, that is, $(V^{\ot (k-1)})^{\ot2}\boxtimes \xi_l^{(n-k+1)}$ is
generated. Then the following vectors are generated.
\begin{align}
&(\pi_x\ot\pi_y)\Delta(e_0e_1\cdots e_{k-1})(|{\bf 0}_{k-1}\rangle\ot|{\bf 1}_{k-1}\rangle)
\boxtimes \xi_l^{(n-k+1)} \nonumber \\
&\qquad=(y|{\bf 0}_k\rangle\ot|{\bf 1}_k\rangle
+xp^{-1}|1,{\bf 0}_{k-1}\rangle\ot|0,{\bf 1}_{k-1}\rangle)\xi_l^{(n-k)},
\label{eq1}\\
&(\pi_x\ot\pi_y)\Delta(e_ae_{a+1}\cdots e_{k-1})
(|{\bf 0}_{a-1},1,{\bf 0}_{k-a-1}\rangle\ot|{\bf 1}_{k-1}\rangle)
\boxtimes \xi_l^{(n-k+1)} \nonumber \\
&\qquad=(|{\bf 0}_{a-1},1,{\bf 0}_{k-a}\rangle
\ot|{\bf 1}_{a-1},0,{\bf 1}_{k-a}\rangle\nonumber\\
&\qquad\quad-q|{\bf 0}_a,1,{\bf 0}_{k-a-1}\rangle
\ot|{\bf 1}_a,0,{\bf 1}_{k-a-1}\rangle)
\boxtimes \xi_l^{(n-k)}\quad(1\le a\le k-1).
\label{eq2}
\end{align}
By Lemma \ref{lem:2} with 
$a=k,{\bf a}={\bf 0},{\bf a'}={\bf e}_1+\cdots+{\bf e}_{k-1}$ 
and $l$ replaced with $l+1$, we can also generate 
\begin{equation} \label{eq3}
(|{\bf 0}_k\rangle\ot|{\bf 1}_k\rangle+\mathrm{i}q^{l+1/2}
|{\bf 0}_{k-1},1\rangle\ot|{\bf 1}_{k-1},0\rangle)\boxtimes \xi_l^{(n-k)}.
\end{equation}
Take the coefficients of $|{\bf 0}_k\rangle\ot|{\bf 1}_k\rangle,
|{\bf 0}_{a-1},1,{\bf 0}_{k-a}\rangle\ot|{\bf 1}_{a-1},0,{\bf 1}_{k-a}\rangle$
($1\le a\le k$) (written dropping $\boxtimes\, \xi_l^{(n-k)}$) from \eqref{eq1},
\eqref{eq2}, \eqref{eq3} and make a matrix $C=(c_{ij})_{1\le i,j\le k+1}$ where 
\[
c_{ij}=\left\{
\begin{array}{ll}
y&(i=j=1)\\
xp^{-1}&(i=1\,\&\,j=2)\\
1&(i=j\,\&\,2\le i\le k\text{ or }i=k+1\,\&\,j=1)\\
-q&(i=j-1\,\&\,2\le i\le k)\\
\mathrm{i}q^{l+1/2}\quad&(i=j=k+1).
\end{array}
\right.
\]
Since $\det C=\mathrm{i}q^{l+1/2}y-p^{-1}q^{k-1}x\ne0$, 
$(|{\bf 0}_k\rangle\ot|{\bf 1}_k\rangle)\boxtimes \xi_l^{(n-k)}$ is also generated.
Now notice that $\mathcal{U}_{B,\{0,\ldots,k-1\}}$ is isomorphic to $U_{-q^{-1}}(B_k)$
with the opposite ordering of Dynkin indices. With this identification, $V^{\ot k}$
is the spin representation with $|{\bf 1}_k\rangle$ as a highest weight vector.
Hence, $|{\bf 0}_k\rangle\ot|{\bf 1}_k\rangle$ is a tensor product
of a lowest weight vector and a highest one with respect to $U_{-q^{-1}}(B_k)$.
By Proposition \ref{pr:ht*lt} $(V^{\ot k})^{\ot2}\boxtimes \xi_l^{(n-k)}$ is generated
and the induction proceeds.
\end{proof}

\begin{proposition}\label{pr:irrD}
$\W\ot\W$ is irreducible.
\end{proposition}

\begin{proof}
Let $W$ be a nonzero submodule of $\W^{\ot2}$. By applying $e_i$ ($1\le i\le n$)
on a nonzero vector in $W$, we arrive at some singular vector $\xi_l$. By Proposition
\ref{pr:recD}, one can generate all $\xi_m$ for $m\in\Z_{\ge0}$. Then Lemma \ref{lem:4}
shows  $(V^{\ot\kappa})^{\ot2}\boxtimes \xi_l\subset W$ for any $l$.
The claim follows from Lemma \ref{lem:1}.
\end{proof}

\section*{Acknowledgments}
The authors thank 
Toshiyuki Tanisaki, Zengo Tsuboi and Hiroyuki Yamane for useful discussion.
Special thanks are due to Shouya Maruyama for a critical reading of the manuscript.
This work is supported by Australian Research Council and 
Grants-in-Aid for Scientific Research No.~23340007 and No.~24540203
from JSPS.

\end{document}